\documentclass[11pt]{article}
\usepackage{a4}
\usepackage{amsmath, amsfonts,amssymb,amscd,amsthm}
\usepackage{graphicx}

\usepackage{subfigure}
\usepackage{pict2e}
\usepackage{color}
\usepackage[percent]{overpic}
\usepackage[colorlinks=true,linkcolor=black,urlcolor=black]{hyperref}
\usepackage[english]{babel}
\usepackage{algorithm, algorithmic}

\newtheorem{theorem}{Theorem}
\newtheorem{theorem*}{Theorem}

\newtheorem{definition}[theorem]{Definition}
\newtheorem{lemma}[theorem]{Lemma}
\newtheorem{corollary}[theorem]{Corollary}
\newtheorem{remark}[theorem]{Remark}
\newtheorem{example}[theorem]{Example}

\DeclareMathOperator{\R}{\mathbb{R}}
\DeclareMathOperator{\gamadi}{\textcolor{black}{\mathcal{S}}}

\DeclareMathOperator{\dxi}{\mathrm{d} \xi_0 \mathrm{d}\xi'}

\DeclareMathOperator{\dt}{\mathrm{d}t}
\DeclareMathOperator{\timeint}{\int \limits_0^\infty}

\begin{document}
\title{Time domain boundary elements for dynamic contact problems}
\author{ Heiko Gimperlein\thanks{Maxwell Institute for Mathematical Sciences and Department of Mathematics, Heriot--Watt University, Edinburgh, EH14 4AS, United Kingdom, email: h.gimperlein@hw.ac.uk.}  \thanks{Institute for Mathematics, University of Paderborn, Warburger Str.~100, 33098 Paderborn, Germany.} \and Fabian Meyer\thanks{Institute of Applied Analysis and Numerical Simulation, Universit\"{a}t Stuttgart, Pfaffenwaldring 57, 70569 Stuttgart, Germany.} \and Ceyhun \"{O}zdemir\thanks{Institute of Applied Mathematics, Leibniz University Hannover, 30167 Hannover, Germany. \newline H.~G.~acknowledges support by ERC Advanced Grant HARG 268105 and the EPSRC Impact Acceleration Account. } \and Ernst P.~Stephan${}^\S$}
\date{\emph{(dedicated to Erwin Stein on the occasion of his 85th birthday)}}

\providecommand{\keywords}[1]{{\textit{Key words:}} #1}

\maketitle \vskip 0.5cm
\begin{abstract}
\noindent \textcolor{black}{This article considers a unilateral contact problem for the wave equation. The problem is reduced to a variational inequality for the Dirichlet-to-Neumann operator for the wave equation on the boundary, which is solved in a saddle point formulation using boundary elements in the time domain. As a model problem, also a variational inequality for the single layer operator is considered. A priori estimates are obtained for Galerkin approximations {both to the variational inequality and the mixed formulation} in the case of a flat contact area, where the existence of solutions to the continuous problem is known. Numerical experiments demonstrate the performance of the proposed {mixed method}. They indicate the stability and convergence beyond flat geometries.  }
\end{abstract}
\keywords{boundary element method; variational inequality; mixed method; a priori error estimates; wave equation.}

\section{Introduction}
\label{intro}

\textcolor{black}{Contact problems play an important role in numerous applications in mechanics, from fracture dynamics and crash tests to rolling car tires \cite{wrig}. As the contact takes place at the interface of two materials, for time-independent problems boundary elements and coupled finite / boundary elements provide an efficient and much-studied tool for numerical simulations \cite{gwinsteph, ency}. The analysis of such problems is well-understood in the context of elliptic variational inequalities.}\\

\textcolor{black}{While contact for time-dependent problems is of clear practical relevance, neither its analysis nor rigorous boundary element methods have been much explored. There is an extensive computational literature, including \cite{doyen2, doyen,  hauret2,hauret, cont2, cont1}, but analytically even the existence of solutions to these free boundary problems is only known for flat contact area \cite{cooper, lebeau}. \textcolor{black}{Some rigorous results have recently been obtained for Nitsche stabilized finite elements \cite{chouly}.}}\\

\textcolor{black}{In this work we propose a time domain boundary element method for a $3d$ dynamic contact problem in the case of the scalar wave equation, as a model problem for elasticity. We provide a priori error estimates for our numerical scheme in the case of a flat contact area, and our numerical experiments indicate the convergence and efficiency  also for curved contact geometries. Motivating references from the time-independent setting include \cite{gatica, eck}.}\\


\textcolor{black}{For the precise statement of the problem, consider} a  Lipschitz domain $\Omega \subset \mathbb{R}^n$ with boundary $\Gamma = \partial \Omega$, \textcolor{black}{where either $\Omega$ is bounded or $\Omega = \mathbb{R}^{n-1}$.} Let $G$ be a bounded Lipschitz subset of $\Gamma$. We consider a unilateral contact problem for the wave equation for the displacement $w : \mathbb{R}\times\Omega \to \mathbb{R}$. {It corresponds to a simplified model for a crack in $G$ between $\Omega$ and a non-penetrable material in $\mathbb{R}^n\setminus \overline{\Omega}$. The contact conditions for non-penetration are described} in terms of the traction $-\mu \frac{\partial w}{\partial \nu}\big|_{G}$ and prescribed forces $h$: 
\begin{align}\label{contactprob}
\begin{cases}
w\textcolor{black}{|_{\mathbb{R}\times G}} \geq 0\ ,\ -\mu \frac{\partial w}{\partial \nu}\textcolor{black}{\big|_{\mathbb{R}\times G}}\geq h\ ,\\ w\textcolor{black}{|_{\mathbb{R}\times G}}>0\ \Longrightarrow\ -\mu \frac{\partial w}{\partial \nu}\textcolor{black}{\big|_{\mathbb{R}\times G}}=h\ .
\end{cases}
\end{align}
\\ The full system of equations for the contact problem is given by:
\begin{align}\label{1.1}
\begin{cases}
\frac{\partial^2 w}{\partial t^2}=c_s^2 \Delta w\ ,\hspace{1cm}  &\text{for}\ (t,x)\in\mathbb{R}\times\Omega\ ,
\\ w=0\ ,\hspace{1cm}  &\text{on} \ \textcolor{black}{\mathbb{R}\times} \Gamma \setminus G \ ,
\\ w\geq0\ ,\ -\mu \frac{\partial w}{\partial \nu}\geq h\ ,\hspace{1cm}  &\text{on} \ \textcolor{black}{\mathbb{R}\times}G \ ,
\\\big(-\mu \frac{\partial w}{\partial \nu}-h)\ w=0, &\text{on} \ \textcolor{black}{\mathbb{R}\times}G \ ,
\\ w=0, &\text{for}\ (t,x)\in(-\infty,0)\times\Omega\ .
\end{cases}
\end{align} 
\textcolor{black}{where $c_s$ denotes the speed of the wave.} \textcolor{black}{While we focus on the physically relevant three dimensional case, $n=3$, the analysis of the numerical schemes can be adapted to $n=2$.} Like for time-independent contact, a formulation as a nonlinear problem on the contact area $G$ leads to efficient numerical approximations.  \\

\textcolor{black}{In this article we formulate \eqref{1.1} as a variational inequality on $G$ in terms of the Dirichlet-to-Neumann operator \eqref{1.4} for the wave equation. Similar to time-independent problems, the Dirichlet-to-Neumann operator \textcolor{black}{is} computed in terms of boundary integral operators as $\frac{1}{2}(W-(1-K')V^{-1}(1-K))$, where $V, K, K'$ and $W$ are the layer potentials defined in \eqref{Voperator}-\eqref{Woperator}.} {\textcolor{black}{Because the contact area and contact forces are often relevant in applications, we replace the variational inequality for the Dirichlet-to-Neumann operator by an equivalent mixed system}, which we discretize with a time domain Galerkin boundary element method. {The  resulting discretized nonlinear inequality in space-time simultaneously approximates the displacement $w$ and the contact forces $-\mu \frac{\partial w}{\partial \nu}$ on $G$. It is} solved with a Uzawa algorithm, either as a time-stepping scheme or in space-time.}\\

{The resulting boundary element method is analyzed in the case of a flat contact area, a situation where the existence of solutions to the contact problem \eqref{1.1} is known. We obtain a priori estimates for the numerical error, both for the variational inequality and a mixed formulation:
\textcolor{black}{
\begin{theorem*}
 Let $h \in H^{\frac{3}{2}}_\sigma(\mathbb{R}^+,{H}^{-\frac{1}{2}}(G))$ and let $u \in H^{\frac{1}{2}}_\sigma(\mathbb{R}^+,\tilde{H}^{\frac{1}{2}}(G))^+$, respectively $u_{\Delta t,h}\in \tilde{K}_{t,h}^+{\subset H^{\frac{1}{2}}_\sigma(\mathbb{R}^+,\tilde{H}^{\frac{1}{2}}(G))^+ }$ be the solutions of the continuous variational inequality \eqref{contVI}, respectively its discretization \eqref{discreteVI}. Then the following estimate holds:
 \begin{align}
 \|u-u_{\Delta t,h}\|_{-\frac{1}{2},\frac{1}{2},\sigma,\star }^2 \lesssim_\sigma \inf \limits_{\phi_{\Delta t,h}\in \tilde{K}_{t,h}^+}(\textcolor{black}{\|h-p_Q \gamadi_\sigma u\|_{\frac{1}{2},-\frac{1}{2},\sigma}} \|u-\phi_{\Delta t,h}\|_{-\frac{1}{2},\frac{1}{2},\sigma,\star}+ 
 \|u-\phi_{\Delta t,h}\|_{\frac{1}{2},\frac{1}{2},\sigma,\star}^2).
\end{align}  
 \end{theorem*}
This result is shown as Theorem \ref{apriori1} in Section \ref{apriori}.
}
\textcolor{black}{\begin{theorem*}
The discrete mixed formulation \eqref{discMixedFormulation} of the contact problem admits a unique solution. The following a priori estimates hold: 
\begin{align}
\| \lambda -\lambda_{\Delta t_2,h_2} \|_{0,-\frac{1}{2}, \sigma}  &\lesssim  \inf \limits_{\tilde{\lambda}_{\Delta t_2,h_2}} \| \lambda - \tilde{\lambda}_{\Delta t_2,h_2} \|_{0,-\frac{1}{2},\sigma} +(\Delta t_1)^{-\frac{1}{2}}  \|u- u_{\Delta t_1,h_1}\|_{-\frac{1}{2},\frac{1}{2},\sigma, \ast} \ ,\\
\|u-u_{\Delta t_1,h_1}\|_{-\frac{1}{2},\frac{1}{2},\sigma, \ast} &\lesssim_\sigma \inf \limits_{v_{\Delta t_1,h_1}}
\|u-v_{\Delta t_1,h_1}\|_{\frac{1}{2},\frac{1}{2},\sigma, \ast}\nonumber \\& \qquad
+ \inf \limits_{\tilde{\lambda}_{\Delta t_2,h_2}}\left\{\|\tilde{\lambda}_{\Delta t_2,h_2} - \lambda\|_{\frac{1}{2},-\frac{1}{2},\sigma} +\|\tilde{\lambda}_{\Delta t_2,h_2} - {\lambda}_{\Delta t_2,h_2}\|_{\frac{1}{2},-\frac{1}{2},\sigma}\right\} \ . 
\end{align}
\end{theorem*}
The result is shown as Theorem \ref{mixedtheorem} in Section \ref{mixedsection}.}\\

For the mixed problem, a key part of the proof is an inf-sup condition for the space-time discretization \textcolor{black}{in Theorem \ref{discInfSup}}. We also demonstrate the convergence of the Uzawa algorithms. Numerical experiments for the mixed formulation confirm the theoretical results and indicate the efficiency and convergence of our approach, beyond flat geometries.} \\

{In addition to the contact problem \eqref{1.1}, as a simpler test case we also consider a punch problem, {which models a rigid body indenting an elastic half-space. The relevant boundary conditions for the wave equation are given by:}  
\begin{align}\label{stampprob1}
\begin{cases}
{-\mu \frac{\partial w}{\partial \nu}\big|_{\textcolor{black}{\mathbb{R}\times}G} \geq 0\ ,\ w\big|_{\textcolor{black}{\mathbb{R}\times}G}\geq h\ ,}\\ {-\mu \frac{\partial w}{\partial \nu}\big|_{\textcolor{black}{\mathbb{R}\times}G}>0\ \Longrightarrow\ w\big|_{\textcolor{black}{\mathbb{R}\times}G}=h\ .}
\end{cases}
\end{align}
For the half-space they lead} to a variational inequality for the single layer operator $V$ \textcolor{black}{defined in \eqref{Voperator}, instead of the Dirichlet-to-Neumann operator \eqref{1.4} in the contact problem}. We obtain similar theoretical and numerical results in this case. \textcolor{black}{In particular, a priori estimates are obtained for the Galerkin approximation of the variational inequality (Theorem \ref{apriori2}) and an equivalent mixed formulation (Theorem \ref{mixedVthm}).} \\

\textcolor{black}{Our approach also relates to the recent interest in coupled and nonlinear interface problems for wave propagation, solved by time domain boundary element methods. In particular, we refer to the fundamental articles \cite{ajrt, bls} for the coupling of FEM and BEM, as well as \cite{aimi1} for an energetic Galerkin formulation of the coupling. Reference \cite{banjai2} considers a nonlinear boundary \textcolor{black}{value} problem. A first analysis of the time domain Dirichlet-to-Neumann operator goes back to \cite{banjai}.}\\

\textcolor{black}{The current work provides a first step towards efficient boundary elements for dynamic contact. For both stationary and dynamic contact, the relevance of adaptive methods to approximate the non-smooth solutions is well-known \cite{gwinsteph,cont1, ency}. Recent advances in the a posteriori error analysis and resulting adaptive mesh refinement procedures for time domain boundary elements \cite{apost, gimperleinreview} will therefore be of interest for the dynamic contact considered here, with a particular view towards tire dynamics \cite{banz2}.}\\

\noindent \emph{The article is organized as follows:} 
Section \ref{faframework} recalls the boundary integral operators associated to the wave equation as well as their mapping properties between suitable space-time anisotropic Sobolev spaces. Section \ref{formulation} reduces the contact problem to a variational inequality on the contact boundary \textcolor{black}{and discusses the existence and uniqueness of solutions both for the contact problem and a related, simpler Dirichlet-to-Neumann equation in a half space}. Section \ref{apriori} describes the discretization and proves a priori error estimates for the contact problem as well as for the Dirichlet-to-Neumann equation. {The a priori error analysis for the mixed formulation is then presented in Section \ref{mixedsection}.} \textcolor{black}{A simpler contact problem, which involves the single layer operator $V$ only, is analyzed in Section \ref{Vsection}.} Section \ref{algo} derives a time stepping scheme for the Dirichlet-to-Neumann equation, which forms the basis for both a time stepping and a space-time Uzawa algorithm for the nonlinear contact problem. Section \ref{experiments} presents numerical experiments {based on the mixed formulation}.

\noindent \textcolor{black}{\emph{Notation:} We write $f \lesssim g$ provided there exists a constant $C$ such that $f \leq C g$. If the constant $C$ is allowed to depend on a parameter $\sigma$, we write $f \lesssim_\sigma g$.}

\section{Formulation of punch and contact problems}
\label{physics}

\textcolor{black}{The mathematical formulation of the physical dynamical contact and punch problems involves the time-dependent Lam\'{e} equation for the displacement $u$ of a linearly elastic body in terms of the stress $\sigma(u)$:
$$\textstyle{\frac{\partial^2 u}{\partial t^2}} - \mathrm{div} \ \sigma(u)=0$$
The contact problem is described by non-penetration boundary conditions at the contact boundary $G$: With $\nu$  the unit normal to $G$,  the normal components $u_n$ and $\sigma_n$ of the displacement $u$, respectively stress $\sigma(u) \nu$, satisfy 
\begin{align*}
\begin{cases}
u_n|_{\mathbb{R}\times G} \geq 0\ ,\sigma_n|_{\mathbb{R}\times G}\geq h\ ,\\ u_n|_{\mathbb{R}\times G}>0\ \Longrightarrow\ \sigma_n|_{\mathbb{R}\times G}=h\ .
\end{cases}
\end{align*}
}
\textcolor{black}{The mathematical analysis of the time-dependent contact problem has proven difficult \cite{eckbook}, and there are few rigorous works on its numerical analysis. Even the existence of weak solutions is proven only for viscoelastic materials or modified contact conditions, such as in \cite{cocou}. 
A second boundary condition with unilateral constraints is the punch (or stamp) problem \cite{stamp, eskin, lure}, which considers a punch indenting a linearly elastic material, where the domain of contact between the punch and the material is not known. Compared to the contact problem the relevant boundary conditions exchange the roles of $\sigma_n$ and $u_n$:
\begin{align*}
\begin{cases}
\sigma_n|_{\mathbb{R}\times G} \geq 0\ ,u_n|_{\mathbb{R}\times G}\geq h\ ,\\ \sigma_n|_{\mathbb{R}\times G}>0\ \Longrightarrow\ u_n|_{\mathbb{R}\times G}=h\ .
\end{cases}
\end{align*}
The analysis presents similar difficulties as the contact problem, and the existence of solutions to the general problem is open.}

\textcolor{black}{As a step towards the numerical analysis of a hyperbolic equation with unilateral constraints \textcolor{black}{and without dissipative terms}, such as the contact and punch problems,  we investigate two simplified model problems. They replace the Lam\'{e} equation by the scalar wave equation in the physical limit when transversal stresses can be neglected \cite{cooper, eskin, lebeau, lure}. We still refer to the simplified problems as contact and punch problems, respectively. The methods here developed for the wave equation are expected to be useful in the study of the general vector-valued dynamic contact and punch problems in elasticity.}

\textcolor{black}{We describe the contact problem considered in this paper between the elastic half-spaces $\mathbb{R}^n_-$ and $\mathbb{R}^n_+$ with coordinates $x = (x',x_n) = (x_1, \dots, x_{n-1}, x_n)$. 
The scalar displacement $w$ satisfies the wave equation  in both half-spaces:
$$\textstyle{\frac{\partial^2 w}{\partial t^2}}- \Delta w = 0\ ,\hspace{1cm}  \text{for}\ (t,x)\in\mathbb{R}\times \mathbb{R}^n_{\pm}\ .$$
Contact takes place in the subdomain $G$ of the boundary $\partial \mathbb{R}^{n}_- = \partial \mathbb{R}^{n}_+= \mathbb{R}^{n-1} \times \{0\}$, which may be thought of as a crack between the bodies. If $w^+$ denotes the displacement of the body at the upper face of $G$ and $w^-$ the displacement of the body at the lower face of $G$, non-penetrability is described by the condition $w^+ - w^- \geq 0$. At points in $G$ without contact, where $w^+ - w^- > 0$, the normal stresses at both the upper and lower face vanish: $\sigma_{x_n}^{\pm} = 0$. If we assume that the opening crack is symmetric with respect to $G$, $w^+ = - w^-$ and $\sigma_{x_n}^{+} = -\sigma_{x_n}^{-}$, we obtain the contact boundary conditions \textcolor{black}{\eqref{contactprob}} for the solution to the wave equation $w = w^+$ with $\sigma_{x_n}^{+} = -\mu \frac{\partial w}{\partial \nu}\textcolor{black}{= -\mu \frac{\partial w}{\partial x_n}}$. Outside the contact region, the two half-spaces are rigidly attached, leading to the full  \textcolor{black}{initial-boundary value problem  \eqref{1.1} for the wave equation}.
}

\textcolor{black}{We also formulate the punch problem for the elastic half-space $\mathbb{R}^n_+$. 
Denote the surface of the punch by $x_n = \phi(t, x')\leq 0$ and assume $\phi(t, 0) = 0$, $\phi \to -\infty$ as $|x_1, \dots, x_{n-1}| \to \infty$. Let $G'$ be the unknown domain of contact and $\eta$ the displacement of the punch in $x_n$-direction. We denote the normal displacement of the plane $\{x_n = 0\}$ by $w$ and its normal stress by $\sigma_{x_n}$. At $x_n=0$ we then have
\begin{equation*}w = \phi + \eta, \quad \sigma_{x_n}\geq 0 \quad \text{ in } \mathbb{R}\times G'\end{equation*}
and 
\begin{equation*}w \geq \phi + \eta, \quad \sigma_{x_n}= 0 \quad \text{ in } \mathbb{R}\times \mathbb{R}^{n-1}\setminus \overline{G'} \ .\end{equation*}
As in the case of the contact problem, the punch conditions are complemented by Dirichlet boundary conditions outside $G'$. 
With $\sigma_{x_n} = -\mu \frac{\partial w}{\partial \nu}  \textcolor{black}{= -\mu \frac{\partial w}{\partial x_n}} $ and $h =  \textcolor{black}{2}(\phi + \eta)$, this can be summarized as the boundary condition \eqref{stampprob1}.
}

\section{Boundary integral operators and Sobolev spaces}
\label{faframework}

We introduce the single layer {potential} in time domain  as
$$ S \varphi(t,x)=\textcolor{black}{2}\int_{\mathbb{R}^+ \times \Gamma} \textcolor{black}{\gamma}(t- \tau,x,y)\ \varphi(\tau,y)\ d\tau\ ds_y\ ,$$
where {$(t,x)\in \mathbb{R}^+ \times \Omega$ and} $\textcolor{black}{\gamma}$ is a fundamental solution to the wave equation. Specifically in 3 dimensions, $ \textcolor{black}{\gamma}(t- \tau,x,y) = \frac{\delta(t-\tau - |x-y|) \delta(y)}{ \textcolor{black}{4 \pi} |x-y|}$, for the Dirac distribution $\delta$, and the single layer potential is given by
\begin{align*}
S \varphi(t,x) &=\frac{1}{ \textcolor{black}{2} \pi} \int_\Gamma \frac{\varphi(t-|x-y|,y)}{|x-y|}\ ds_y \ .
\end{align*}
{We similarly define the double-layer potential as
$$D\varphi(t,x)=  \textcolor{black}{2}\int_{\mathbb{R}^+\times \Gamma} \frac{\partial  \textcolor{black}{\gamma}}{\partial n_y}(t- \tau,x,y)\ \varphi(\tau,y)\ d\tau\ ds_y \ .$$}

For the Dirichlet-to-Neumann operator, we require the {single--layer operator} $V$, its normal derivative $K'$, the double--layer operator $K$ and hypersingular operator $W$ for $x \in \Gamma$, $t>0$:  
\begin{align}
V \varphi(t,x)&={   \textcolor{black}{2}\int_{\mathbb{R}^+\times \Gamma}  \textcolor{black}{\gamma}(t- \tau,x,y)\ \varphi(\tau,y)\ d\tau\ ds_y\, , }\label{Voperator}\\
K\varphi(t,x)&= \textcolor{black}{2}\int_{\mathbb{R}^+\times \Gamma} \frac{\partial  \textcolor{black}{\gamma}}{\partial n_y}(t- \tau,x,y)\ \varphi(\tau,y)\ d\tau\ ds_y\ , \label{Koperator}\\
K' \varphi(t,x)&=  \textcolor{black}{2}\int_{\mathbb{R}^+\times \Gamma} \frac{\partial  \textcolor{black}{\gamma}}{\partial n_x}(t- \tau,x,y)\ \varphi(\tau,y)\ d\tau\ ds_y\, , \label{Kpoperator}\\
W \varphi(t,x)&= \textcolor{black}{2}\int_{\mathbb{R}^+\times \Gamma} \frac{\partial^2  \textcolor{black}{\gamma}}{\partial n_x \partial n_y}(t- \tau,x,y)\ \varphi(\tau,y)\ d\tau\ ds_y \ . \label{Woperator}
\end{align}

\begin{remark}\label{vanish}
When $\Omega = \mathbb{R}^n_+$, the normal derivative of $\textcolor{black}{\gamma}$ vanishes on $\Gamma= \partial \mathbb{R}^n_+$. Therefore, $K\varphi=K'\varphi=0$ in this case.
\end{remark}

{The boundary integral operators} are considered between space-time aniso\-tropic Sobolev spaces $H_\sigma^s(\mathbb{R}^+,\widetilde{H}^{r}(\Gamma))$, see  \cite{HGEPSN} or \cite{haduong}. {To define them, if $\partial\Gamma\neq \emptyset$, first extend $\Gamma$ to a closed, orientable Lipschitz manifold $\widetilde{\Gamma}$. }

{On $\Gamma$ one defines the usual Sobolev spaces of supported distributions:
$$\widetilde{H}^r(\Gamma) = \{u\in H^r(\widetilde{\Gamma}): \mathrm{supp}\ u \subset {\overline{\Gamma}}\}\ , \quad\ r \in \mathbb{R}\ .$$
Furthermore, ${H}^r(\Gamma)$ is the quotient space $ H^r(\widetilde{\Gamma}) / \widetilde{H}^r({\widetilde{\Gamma}\setminus\overline{\Gamma}})$.} \\
{To write down an explicit family of Sobolev norms, introduce a partition of unity $\alpha_i$ subordinate to a covering of $\widetilde{\Gamma}$ by open sets $B_i$. For diffeomorphisms $\varphi_i$ mapping each $B_i$ into the unit cube $\subset \mathbb{R}^n$, a family of Sobolev norms is induced from $\mathbb{R}^n$, \textcolor{black}{with parameter $\omega \in \mathbb{C}\setminus \{0\}$}:
\begin{equation*}
 ||u||_{r,\omega,{\widetilde{\Gamma}}}=\left( \sum_{i=1}^p \int_{\mathbb{R}^n} (|\omega|^2+|\xi|^2)^r|\mathcal{F}\left\{(\alpha_i u)\circ \varphi_i^{-1}\right\}(\xi)|^2 d\xi \right)^{\frac{1}{2}}\ .
\end{equation*}
The norms for different $\omega \in \mathbb{C}\setminus \{0\}$ are equivalent and $\mathcal{F}$ denotes the Fourier transform. They induce norms on $H^r(\Gamma)$, $||u||_{r,\omega,\Gamma} = \inf_{v \in \widetilde{H}^r(\widetilde{\Gamma}\setminus\overline{\Gamma})} \ ||u+v||_{r,\omega,\widetilde{\Gamma}}$ and on $\widetilde{H}^r(\Gamma)$, $||u||_{r,\omega,\Gamma, \ast } = ||e_+ u||_{r,\omega,\widetilde{\Gamma}}$. $e_+$ extends the distribution $u$ by $0$ from $\Gamma$ to $\widetilde{\Gamma}$. \textcolor{black}{As the norm $||u||_{r,\omega,\Gamma, \ast }$ corresponds to extension by zero, while $||u||_{r,\omega,\Gamma}$ allows  extension by an arbitrary $v$, $||u||_{r,\omega,\Gamma, \ast }$ is stronger than $||u||_{r,\omega,\Gamma}$. Like in the time-independent case the norms are not equivalent whenever $r \in \frac{1}{2} + \mathbb{Z}$ \cite{gwinsteph}.}

{We now define a class of space-time anisotropic Sobolev spaces:
\begin{definition}\label{sobdef}
For {$\sigma>0$ and} $s, r \in\mathbb{R}$ define
\begin{align*}
 H^s_\sigma(\mathbb{R}^+,{H}^r(\Gamma))&=\{ u \in \mathcal{D}^{'}_{+}(H^r(\Gamma)): e^{-\sigma t} u \in \mathcal{S}^{'}_{+}(H^r(\Gamma))  \textrm{ and }   ||u||_{s,r,\sigma,\Gamma} < \infty \}\ , \\
 H^s_\sigma(\mathbb{R}^+,\widetilde{H}^r({\Gamma}))&=\{ u \in \mathcal{D}^{'}_{+}(\widetilde{H}^r({\Gamma})): e^{-\sigma t} u \in \mathcal{S}^{'}_{+}(\widetilde{H}^r({\Gamma}))  \textrm{ and }   ||u||_{s,r,\sigma,\Gamma, \ast} < \infty \}\ .
\end{align*}
$\mathcal{D}^{'}_{+}(E)$ respectively~$\mathcal{S}^{'}_{+}(E)$ denote the spaces of distributions, respectively~tempered distributions, on $\mathbb{R}$ with support in $[0,\infty)$, taking values in $E = {H}^r({\Gamma}), \widetilde{H}^r({\Gamma})$. The relevant norms are given by
\begin{align*}
\|u\|_{s,r,\sigma}:=\|u\|_{s,r,\sigma,\Gamma}&=\left(\int_{-\infty+i\sigma}^{+\infty+i\sigma}|\omega|^{2s}\ \|\hat{u}(\omega)\|^2_{r,\omega,\Gamma}\ d\omega \right)^{\frac{1}{2}}\ ,\\
\|u\|_{s,r,\sigma,\ast} := \|u\|_{s,r,\sigma,\Gamma,\ast}&=\left(\int_{-\infty+i\sigma}^{+\infty+i\sigma}|\omega|^{2s}\ \|\hat{u}(\omega)\|^2_{r,\omega,\Gamma,\ast}\ d\omega \right)^{\frac{1}{2}}\,.
\end{align*}
\end{definition}
\textcolor{black}{They are Hilbert spaces, and we note that the basic case $s=r=0$ is the weighted $L^2$-space with scalar product $\langle u,v \rangle_\sigma := \int_0^\infty e^{-2\sigma t} \int_\Gamma u \overline{v} ds_x\ dt$. Because $\Gamma$ is Lipschitz, like in the case of standard Sobolev spaces  \cite{necas}} these spaces are independent of the choice of $\alpha_i$ and $\varphi_i$ \textcolor{black}{when $|r|\leq 1$}.~\textcolor{black}{ We further introduce the set of nonnegative distributions 
${H}^{r}_\sigma(\R^+, \tilde{H}^{s}(G))^+$}.\\
}
{The boundary integral operators} obey the following mapping properties between these spaces:
\begin{theorem}[\cite{HGEPSN}]\label{mappingproperties}
The following operators are continuous for $r\in \R$:
\begin{align*}
& V:  {H}^{r+1}_\sigma(\R^+, \tilde{H}^{-\frac{1}{2}}(\Gamma))\to  {H}^{r}_\sigma(\R^+, {H}^{\frac{1}{2}}(\Gamma)) \ ,
\\ & K':  {H}^{r+1}_\sigma(\R^+, \tilde{H}^{-\frac{1}{2}}(\Gamma))\to {H}^{r}_\sigma(\R^+, {H}^{-\frac{1}{2}}(\Gamma)) \ ,
\\ & K:  {H}^{r+1}_\sigma(\R^+, \tilde{H}^{\frac{1}{2}}(\Gamma))\to {H}^{r}_\sigma(\R^+, {H}^{\frac{1}{2}}(\Gamma)) \ ,
\\ & W:  {H}^{r+1}_\sigma(\R^+, \tilde{H}^{\frac{1}{2}}(\Gamma)))\to {H}^{r}_\sigma(\R^+, {H}^{-\frac{1}{2}}(\Gamma)) \ .
\end{align*}
\end{theorem}

\noindent When $\Omega = \mathbb{R}^n_+$, Fourier methods yield improved estimates for $V$ and $W$, {see also Section \ref{formulation}}:
\begin{theorem}[\cite{haduongcrack}, pp. 503-506]\label{mapimproved} The following operators are continuous for $r,s \in \mathbb{R}$:
\begin{align*}
& V:{H}^{r+\frac{1}{2}}_\sigma(\R^+, \tilde{H}^{s}(\Gamma))\to {H}^{r}_\sigma(\R^+, {H}^{s+1}(\Gamma)) \ , \\
& W: {H}^{r}_\sigma(\R^+, \tilde{H}^{s}(\Gamma))\to {H}^{r}_\sigma(\R^+, {H}^{s-1}(\Gamma))\ .
\end{align*}
\end{theorem} 

{Theorems \ref{mappingproperties} and \ref{mapimproved} imply the corresponding mapping properties for the composition with the restriction $p_Q$ to $Q = \mathbb{R}\times G$. For example, from Theorem \ref{mappingproperties} we obtain $p_Q V:  {H}^{r+1}_\sigma(\R^+, \tilde{H}^{-\frac{1}{2}}(\Gamma))\to  {H}^{r}_\sigma(\R^+, {H}^{\frac{1}{2}}(G))$ and $p_Q V:  {H}^{r+1}_\sigma(\R^+, \tilde{H}^{-\frac{1}{2}}(G)) \hookrightarrow {H}^{r+1}_\sigma(\R^+, \tilde{H}^{-\frac{1}{2}}(\Gamma))\to  {H}^{r}_\sigma(\R^+, {H}^{\frac{1}{2}}(G))$.}

\textcolor{black}{As noted by Bamberger and Ha Duong \cite{BamHa}, when composed with a time derivative $V$ satisfies a coercivity estimate in the norm of ${H}^{0}_\sigma(\R^+, \tilde{H}^{-\frac{1}{2}}(\Gamma))$:  $\|\phi\|^2_{0,-\frac{1}{2}, \sigma, \ast} \lesssim_\sigma \langle V \phi, \partial_t \phi\rangle$. On the other hand the mapping properties of Theorem \ref{mappingproperties} imply the continuity of the bilinear form associated to $V \partial_t$ in the bigger norm of ${H}^{1}_\sigma(\R^+, \tilde{H}^{-\frac{1}{2}}(\Gamma))$: $\langle V \phi, \partial_t \phi\rangle \lesssim \|\phi\|^2_{1,-\frac{1}{2}, \sigma, \ast}$. These estimates are a crucial ingredient in the numerical analysis of time-domain boundary integral equations. To study the equation $V \phi = f$, Bamberger and Ha Duong \cite{BamHa} consider the weak form of the differentiated equation $V \partial_t \phi = \partial_t f$ with the operator $V \partial_t$.}\\  \textcolor{black}{Similar estimates with different norms in the upper and lower bounds hold for $W\partial_t$: $\|\psi\|^2_{0,\frac{1}{2}, \sigma, \ast} \lesssim_\sigma \langle W \psi, \partial_t \psi\rangle\lesssim \|\psi\|^2_{1,\frac{1}{2}, \sigma, \ast}$. See \cite{HGEPSN, haduong} for proofs and further information.}\\ 
\textcolor{black}{The few known analytical results for their well-posedness, like  \cite{cooper, lebeau}, are restricted to geometric situations where refined coercivity estimates without time derivatives are available. Indeed, while the  equation $V \phi = f$ is equivalent to the differentiated form $V \partial_t \phi = \partial_t f$, this is not true for inequality conditions like \eqref{1.1}.}

\textcolor{black}{For flat contact area, when $\Omega = \mathbb{R}^n_+$, in this article we use the coercivity estimate $$\|\phi\|_{-\frac{1}{2},-\frac{1}{2},\sigma,\ast}^2 \lesssim_\sigma \langle p_Q V \phi,\phi\rangle_\sigma\lesssim \|\phi\|_{\frac{1}{2},-\frac{1}{2},\sigma,\ast}^2$$ from \cite{cooper, haduongcrack}, which does not involve a time derivative $\partial_t$. }

\section{\textcolor{black}{Contact problem: Boundary integral formulation and well-posedness}}
\label{formulation}

\textcolor{black}{As in Cooper \cite{cooper}, we start with a regularized contact problem with parameter $\sigma >0$. The analysis lets $\sigma \to 0^+$ at the end, to recover the existence of weak solutions to the contact problem \eqref{1.1}.} We let $w_\sigma=e^{-\sigma t}w$ and $h_\sigma=e^{-\sigma t}h$. Using appropriate units, we may also assume $c_s=1$. Multiplying (\ref{1.1}) by $e^{-\sigma t}$, we then obtain
\begin{align}\label{1.3}
\begin{cases}
\left(\frac{\partial}{\partial t}+\sigma\right)^2w_\sigma= \Delta w_\sigma\ ,\hspace{1cm}  &\text{for}\ (t,x)\in\mathbb{R}\times\Omega\ ,
\\ w_\sigma=0\ ,\hspace{1cm}  &\text{on} \ \textcolor{black}{\mathbb{R}\times} \Gamma \setminus \overline{G} \ ,
\\ w_\sigma\geq0\ ,\ -\mu \frac{\partial w_\sigma}{\partial \nu}\geq h_\sigma\ ,\hspace{1cm}  &\text{on} \ \textcolor{black}{\mathbb{R}\times}G \ ,
\\\big(-\mu \frac{\partial w_\sigma}{\partial \nu}-h_\sigma)\ w_\sigma=0, &\text{on} \ \textcolor{black}{\mathbb{R}\times}G \ ,
\\ w_\sigma=0, &\text{for}\ (t,x)\in(-\infty,0)\times\Omega\ .
\end{cases}
\end{align}
\textcolor{black}{We apply the Fourier transform in $(t,x')$ to the first equation of \eqref{1.3}, where $x' = (x_1, \dots, x_{n-1})$, and obtain the ordinary differential equation
\begin{align*}
-c_s^{-2}(\xi_0+i\sigma)^2 \hat{w}_\sigma=-|\xi'|^2 \hat{w}_\sigma+\frac{\partial^2}{\partial x_n^2} \hat{w}_\sigma .
\end{align*}
It has the solution
\begin{align*}
\hat{w}_\sigma(\xi_0+i\sigma,\xi',x_n)=C_1(\xi_0+i\sigma,\xi')e^{i\Gamma x_n}+C_2(\xi_0+i\sigma,\xi')e^{-i\sigma \Gamma x_n},
\end{align*}
where $\Gamma(\xi_0+i\sigma,\xi')=\sqrt{c_s^{-2}(\xi_0+i\sigma)^2 -|\xi'|^2}$ and the branch of the square root is chosen such that
$\sqrt {c_s^{-2}(\xi_0+i\sigma)^2 -|\xi'|^2} \approx c_s^{-1}(\xi_0+i\sigma)$ for $|\xi_0+i\sigma| \gg |\xi'|$. The condition that $\hat{w}_\sigma$ is square integrable in $x_n$ implies $C_2 =0$.
From the trace $u_\sigma(x_0,x')=w_\sigma(x_0,x', x_n=0^+)=\mathcal{F}^{-1}(C_1(\xi_0+i\sigma,\xi'))$ we see that
$\hat{w}_\sigma(\xi_0+i\sigma,\xi',x_n)=\hat{u}_\sigma(\xi_0+i\sigma,\xi')e^{i \Gamma x_n}$.
We define the Dirichlet-to-Neumann operator by} \begin{align}\label{1.4}
\gamadi_\sigma w_\sigma|_{\textcolor{black}{\mathbb{R}\times}\Gamma} := -\mu\frac{\partial w_\sigma}{\partial \nu}\Big|_{\textcolor{black}{\mathbb{R}\times}\Gamma}\ .
\end{align}
\textcolor{black}{Note that
 \begin{align*}
\frac{\partial w_\sigma}{\partial \nu}& =  \partial x_n\mathcal{F}^{-1}(\hat{u}_\sigma(\xi_0+i\sigma,\xi')e^{i \Gamma x_n})\\ &=(2\pi)^{-n}\int_{\R^n}e^{ix_0\xi_0+ix'\xi'}i \Gamma(\xi_0+i\sigma,\xi')\hat{u}_\sigma(\xi_0+i\sigma,\xi')e^{i \Gamma x_n} \dxi ,
\end{align*} 
so that $\gamadi_\sigma$ is a generalized pseudodifferential operator with symbol $-i\mu\Gamma$:
\begin{align}\label{gamadidef}
&\gamadi_\sigma u_\sigma=(2\pi)^{-n}\int_{\R^n} e^{ix_0\xi_0+ix'\xi'}(-i\mu\Gamma(\xi_0+i\sigma,\xi'))\hat{u}_{\sigma}(\xi_0,\xi') \dxi \ .
\end{align}}
\textcolor{black}{From this explicit formula, one notices the estimates (\cite{haduongcrack}, p.~499):
\begin{align}\label{symbolest1}
 \textstyle{\frac{\sigma^{\frac{1}{2}}}{|\xi_0+i\sigma|^{\frac{1}{2}}}}(|\xi_0+i\sigma|^2+|\xi'|^2)^{\frac{1}{2}}\leq |\Gamma(\xi_0+i\sigma,\xi')| &\leq (|\xi_0+i\sigma|^2+|\xi'|^2)^{\frac{1}{2}} \ ,
\\ 
Im \ \Gamma(\xi_0+i\sigma,\xi')&\geq \textstyle{ \frac{\sigma}{|\xi_0+i\sigma|}}(|\xi_0+i\sigma|^2 +|\xi'|^2)^{\frac{1}{2}}. \label{symbolest2}
\end{align}
They translate into the following coercivity and mapping properties for the Dirichlet-to-Neumann operator in the case of the half space $\Omega = \mathbb{R}^n_+$: }
\begin{theorem}
\label{mappingProperties} 
\textcolor{black}{$p_Q \gamadi_\sigma: {H}^{{s}}_\sigma(\R^+, \tilde{H}^{{\frac{1}{2}}}(G)) \to {H}^{{s}}_\sigma(\R^+, {H}^{{-\frac{1}{2}}}(G))$ continuously and} $\|\phi\|_{-\frac{1}{2},\frac{1}{2},\sigma,\ast}^2 \lesssim_\sigma \langle p_Q \gamadi_\sigma \phi,\phi\rangle\lesssim \|\phi\|_{0,\frac{1}{2},\sigma,\ast}^2$.
\end{theorem}
\begin{proof}\textcolor{black}{By the Plancherel theorem we observe
\begin{align*}\|p_Q \gamadi_\sigma \phi\|_{s,-\frac{1}{2},\sigma,\ast}^2 &= (2\pi)^{-2n}\langle -i\mu  |\xi_0+i\sigma|^{2s} (|\xi_0+i\sigma|^2+|\xi'|^2)^{-\frac{1}{2}}\Gamma(\xi_0+i\sigma,\xi') \hat{\phi}, -i\mu \Gamma(\xi_0+i\sigma,\xi')  \hat{\phi}\rangle \\&=  (2\pi)^{-2n}\langle |\mu|^2  |\xi_0+i\sigma|^{2s} (|\xi_0+i\sigma|^2+|\xi'|^2)^{-\frac{1}{2}}|\Gamma(\xi_0+i\sigma,\xi')|^2 \hat{\phi}, \hat{\phi}\rangle\ .\end{align*}
From the upper bound in \eqref{symbolest1}, one concludes 
$$\|p_Q \gamadi_\sigma \phi\|_{s,-\frac{1}{2},\sigma,\ast}^2\leq (2\pi)^{-2n}\langle |\mu|^2  |\xi_0+i\sigma|^{2s} (|\xi_0+i\sigma|^2+|\xi'|^2)^{\frac{1}{2}} \hat{\phi}, \hat{\phi}\rangle \lesssim \|\phi\|_{s,\frac{1}{2},\sigma,\ast}^2\ .$$
Similarly, using
\begin{equation}\label{coerfourier}
\langle p_Q \gamadi_\sigma \phi,\phi\rangle = (2\pi)^{-n}\langle -i\mu \Gamma(\xi_0+i\sigma,\xi') \hat{\phi}, \hat{\phi}\rangle
\end{equation}
and the upper bound in \eqref{symbolest1},  
$$\langle p_Q \gamadi_\sigma \phi,\phi\rangle \lesssim \langle (|\xi_0+i\sigma|^2+|\xi'|^2)^{\frac{1}{2}}\hat{\phi}, \hat{\phi} \rangle =  \|\phi\|_{0,\frac{1}{2},\sigma,\ast}^2\ .$$ Finally, the lower bound follows from \eqref{coerfourier} and the lower bound for $\Gamma$ in \eqref{symbolest2}:
\begin{align*} \mathrm{Re}\ \langle p_Q \gamadi_\sigma \phi,\phi\rangle & = (2\pi)^{-n} \langle \mu (\mathrm{Im}\ \Gamma(\xi_0+i\sigma,\xi')) \hat{\phi}, \hat{\phi}\rangle \\
& \gtrsim_\sigma \langle |\xi_0+i\sigma|^{-1}(|\xi_0+i\sigma|^2 +|\xi'|^2)^{\frac{1}{2}} \hat{\phi}, \hat{\phi}\rangle \simeq \|\phi\|_{-\frac{1}{2},\frac{1}{2},\sigma,\ast}^2
\end{align*}
} 
\end{proof}
Substituting the definition \eqref{1.4} of the Dirichlet-to-Neumann operator $\gamadi_\sigma$ into
the boundary conditions of the contact problem \eqref{1.3}, the problem \eqref{1.3} reduces to an equivalent inequality in terms of the trace $u_\sigma := w_\sigma|_\Gamma$ on the boundary: \\
Find $u_\sigma$ with $\text{supp }u_\sigma \subset Q_0 = \overline{\mathbb{R}^+}\times G$ such that
\begin{align}\label{1.5}
u_\sigma \geq 0\ , \ \gamadi_\sigma u_\sigma \geq h_\sigma \ , \ (\gamadi_\sigma u_\sigma-h_\sigma)\ u_\sigma=0 \ \text{on}\ Q = \mathbb{R} \times G\ .
\end{align}

More precisely, in terms of the restriction $p_Q$ to $Q$, we obtain the following weak formulation as a variational inequality for suitably smooth $h_\sigma$: \\
Find $u_\sigma\in {H}^{\textcolor{black}{\frac{1}{2}}}_\sigma(\R^+, \tilde{H}^{\textcolor{black}{\frac{1}{2}}}(G))$ such that:
\begin{align}\label{varI}
u_\sigma \geq 0\  \text{ and } \forall v \in  {H}^{\textcolor{black}{\frac{1}{2}}}_\sigma(\R^+, \tilde{H}^{\textcolor{black}{\frac{1}{2}}}(G)) \text{ with }v\geq 0:\ \langle p_Q\gamadi_\sigma u_\sigma,v-u_\sigma\rangle_\sigma \geq \langle h_\sigma, v-u_\sigma \rangle_{\textcolor{black}{\sigma}}\ .
\end{align}

\begin{theorem}\label{VIequiv}
The contact problem \eqref{1.3} is equivalent to the variational inequality \eqref{varI}.  
\end{theorem}
\begin{proof}
\textcolor{black}{By definition of the Dirichlet-to-Neumann operator $\gamadi_\sigma$ in \eqref{1.4}, the contact problem \eqref{1.3} is formally equivalent to \eqref{1.5}. \\
To deduce \eqref{varI} from \eqref{1.5}, multiply the second inequality in \eqref{1.5} by $v \geq 0$ to see that $$\langle p_Q \gamadi_\sigma u_\sigma - h_\sigma, v\rangle_\sigma\geq 0, $$ while from the third equality in \eqref{1.5} we note $\langle p_Q \gamadi_\sigma u_\sigma - h_\sigma, u_\sigma\rangle_\sigma = 0$. The variational inequality  \eqref{varI} follows.\\
Conversely, only the second and third (in-)equalities in \eqref{1.5} need to be shown. For the former, set $v=u_\sigma + v'$ with $v'\geq 0$ in \eqref{varI} to deduce $\langle p_Q\gamadi_\sigma u_\sigma - h_\sigma,v'\rangle_{\textcolor{black}{\sigma}}\geq 0$ for all $v'\geq 0$. Therefore we indeed obtain the inequality $\gamadi_\sigma u_\sigma - h_\sigma \geq 0$ for the integrand. To see the remaining equality in \eqref{1.5}, set $v=2 u_\sigma\geq 0$ in \eqref{varI}, so that $\langle p_Q\gamadi_\sigma u_\sigma,u_\sigma\rangle_{\textcolor{black}{\sigma}} \geq \langle h_\sigma, u_\sigma \rangle_{\textcolor{black}{\sigma}}$. On the other hand, $v=0$ leads to $\langle p_Q\gamadi_\sigma u_\sigma,u_\sigma\rangle_{\textcolor{black}{\sigma}} \leq \langle h_\sigma, u_\sigma \rangle_{\textcolor{black}{\sigma}}$. Therefore $\langle p_Q\gamadi_\sigma u_\sigma - h_\sigma,u_\sigma\rangle_{\textcolor{black}{\sigma}} = 0$. By the already established first and second inequalities in \eqref{1.5}, $p_Q\gamadi_\sigma u_\sigma - h_\sigma\geq 0$ and $u_\sigma\geq 0$, so that the integrand $(\gamadi_\sigma u_\sigma - h_\sigma) u_\sigma = 0$.}
\end{proof}
\textcolor{black}{We now restrict ourselves to the  half space $\Omega = \mathbb{R}^n_+$, \textcolor{black}{where the sharp continuity and coercivity estimates from Theorem \ref{mappingProperties} allow to show the well-posedness of the contact problem.} \textcolor{black}{We emphasize that the analysis applies in any dimension $n$.} }
\begin{theorem}[\cite{cooper}, p.~450]
\label{theoregularity}
Let $h\in  H^{\frac{3}{2}}_\sigma(\mathbb{R}^+,{H}^{-\frac{1}{2}}(G))$. Then there exists a unique  solution $u_{{\sigma}} \in H^{\frac{1}{2}}_\sigma(\mathbb{R}^+,\tilde{H}^{\frac{1}{2}}(G))^+$ of \eqref{varI}. 
\end{theorem}
In terms of the original problem \eqref{1.3} we obtain:
\begin{theorem}[\cite{cooper}, p.~451]
\label{theoclassical}
Let $h\in  H^{\frac{3}{2}}_\sigma(\mathbb{R}^+,{H}^{-\frac{1}{2}}(G))$. \textcolor{black}{Then there exists} a unique $w(\cdot,x_n)\in C(\overline{\R}^+_{x_n}; H^{\frac{1}{2}}_\sigma(\mathbb{R}^+,H^{\frac{1}{2}}(\R^{n-1})) \cap H^{0}_\sigma(\mathbb{R}^+,{H}^{1}(\mathbb{R}^n))$ satisfying \eqref{1.3}.
\end{theorem}

We also note the (simpler) existence of solutions to the corresponding equality from \cite{sako}, p.~48:
\begin{theorem}\label{sakosol}
Let $h\in  H^{\frac{3}{2}}_\sigma(\mathbb{R}^+,{H}^{-\frac{1}{2}}(G))$. Then there exists a unique $u_\sigma\in {H}^{\frac{1}{2}}_\sigma(\R^+, \tilde{H}^{\frac{1}{2}}(G))$ which solves:	
\begin{align*}
\langle p_Q \gamadi_\sigma u_\sigma,v \rangle_{\textcolor{black}{\sigma}} =\langle h,v \rangle_{\textcolor{black}{\sigma}} \qquad \forall v \in {H}^{-\frac{1}{2}}_\sigma(\R^+, \tilde{H}^{\frac{1}{2}}(G)).
\end{align*} 
\end{theorem}

\section{Discretization and a priori error estimates}
\label{apriori}

{\textcolor{black}{For the discretization, we restrict to $n=3$ and assume that $\Gamma$ is approximated by a piecewise polygonal surface. The approximation is again denoted by $\Gamma$. We consider a triangulation $\mathcal{T}_{S}:= \{S_1,\ldots,S_{N}\}$ of $\Gamma$ into $N$ closed triangular faces $S_i$.} The triangulation is assumed to be quasi-uniform and compatible with the area of contact $G$: \textcolor{black}{For all $i=1, \dots, N$,} if $S_i \cap G \neq \emptyset$, then $\text{int }S_i \subset G$.

{Associated to the triangulation \textcolor{black}{$\mathcal{T}_{S}$}, we obtain the space $V_h^q(\Gamma)$ of piecewise polynomial functions of degree $q$. Due to the compatibility of the meshes we have $V_h^q(G)\subset V_h^q(\Gamma)$. Moreover we define $\tilde{V}_h^q(G)$ as the subspace of those functions in $V_h^q(G)$, which vanish on $\partial G$ for $q\geq 1$. }

{For the time discretization we consider a uniform decomposition of the time interval $[0,\infty)$ into subintervals $[t_{n-1}, t_n)$ with time step $\Delta t$, such that $t_n=n\Delta t \; (n=0,1,\dots)$. Associated to this mesh, the space $V^p_{t}$ consists of piecewise  polynomial  functions of degree of $p$ (continuous and vanishing at $t=0$ if $p\geq 1$).}\\

{Let $\mathcal{T}_T=\{[0,t_1),[t_1,t_2),\cdots,$ $[t_{N-1},T)\}$ \textcolor{black}{be} the time mesh for a finite subinterval $[0,T)$. In space-time we consider the \textcolor{black}{algebraic} tensor product of the approximation spaces, $V_h^q$ and $V^p_{t}$, associated to the space-time mesh $\mathcal{T}_{S,T}=\mathcal{T}_S \times\mathcal{T}_T$, and we write
\begin{align}\label{fespace}
V_{t,h}^{p,q}:=  V_t^p \otimes V_{h}^q\ .
\end{align}
We analogously define 
\begin{align}\label{fespace2}
\tilde{V}_{t,h}^{p,q}:=  V_t^p \otimes \tilde{V}_{h}^q
\end{align}
We further define the subspace $K_{t,h}^+\subset V_{t,h}^{p,q}$ as the subspace of nonnegative piecewise polynomials.}

\textcolor{black}{The discretization space $\tilde{V}_{t,h}^{p,q}$ is contained in ${H}^{\frac{1}{2}}_\sigma(\R^+, \tilde{H}^{\frac{1}{2}}(G))$ for $p,q\geq 1$, and we denote the embedding by $j_{t,h}$. $\tilde{V}_{t,h}^{p,q}$ is contained in ${H}^{\frac{1}{2}}_\sigma(\R^+, \tilde{H}^{-\frac{1}{2}}(G))$ for $p \geq 1$, and  we denote the embedding by $k_{t,h}$. The discretized Dirichlet-to-Neumann operator may then be expressed in terms of $j_{t,h}$, $k_{t,h}$ and their adjoints $j_{t,h}^*$, $k_{t,h}^*$ as $$\gamadi_{h, \Delta t} = \frac{1}{2} (j_{t,h}^*Wj_{t,h}-j_{t,h}^*(1-K') k_{t,h}(k_{t,h}^*Vk_{t,h})^{-1}k_{t,h}^* (1-K) j_{t,h})\ .$$}

For $u_{\Delta t,h} \in V_{t,h}^{p,q} $ we thus may write 
\begin{align*}
u_{\Delta t,h}(t,x)=\sum \limits_{i=0}^{N_t} \sum \limits_{j=0}^{N_s} c_j^i \beta_{\Delta t}^i(t)\xi_h^j(x)\ .
\end{align*} 
{in terms of the basis functions $\beta_{\Delta t}^i$ in time and $\xi_h^j$ in space.\\ }

{We use the following notation for piecewise linear or constant functions:}
\begin{itemize}
\item
 $\gamma_{\Delta t}^n(t)$ \text{for the basis of piecewise constant functions in time,}
\item $\beta_{\Delta t}^n(t)$ \text{for the basis of piecewise linear functions in time,}
\item $\psi_{h}^i(x)$ \text{for the basis of piecewise  constant functions in space,}
\item $\xi_{h}^i(x)$ \text{for the basis of piecewise  linear functions in space.}
\end{itemize}

We recall the formulation as a continuous variational inequality:\\
Find $u\in H^{\frac{1}{2}}_\sigma(\mathbb{R}^+,\tilde{H}^{\frac{1}{2}}(G))^+$ such that
\begin{align}\label{contVI}
\langle p_Q \gamadi_\sigma u_\sigma, v-u_\sigma \rangle_{\textcolor{black}{\sigma}} \geq \langle h,v-u_\sigma \rangle_{\textcolor{black}{\sigma}}
\end{align}
holds for all $v\in {H}^{\frac{1}{2}}_\sigma(\R^+, \tilde{H}^{\frac{1}{2}}(G))^+$.
The discretized variational inequality reads as follows. \\
Find $u_{\Delta t,h}\in \tilde{K}_{t,h}^+$ such that
\begin{align}
\label{discreteVI}
  \langle p_Q \textcolor{black}{\gamadi_{\Delta t, h}} u_{\Delta t,h}, v_{\Delta t,h}-u_{\Delta t,h} \rangle_{\textcolor{black}{\sigma}} \geq \langle h,v_{\Delta t,h}-u_{\Delta t,h} \rangle_{\textcolor{black}{\sigma}}
\end{align}
holds for all $v_{\Delta t,h}\in \tilde{K}_{t,h}^+$.
\\ 
\textcolor{black}{In using \eqref{discreteVI}, the operator $\gamadi_{\sigma}$ is approximated by $\gamadi_{\Delta t, h}$ which inverts the equations \eqref{discdirneu1}, \eqref{discdirneu2} with the same ansatz and test functions. In our error analysis we assume that $\gamadi_\sigma$ is computed exactly, as in basic time-independent works \cite{eck}.} We refer to \cite{banjai} for the challenges of analyzing the discretization.  While $\sigma>0$ is required for the theoretical analysis\textcolor{black}{, and $\gamadi_{\sigma}$ can be computed from layer operators $V_\sigma$, $K_\sigma$, $K_\sigma'$ and $W_\sigma$ with a modified Green's function,} practical computations directly use $\sigma=0$ \textcolor{black}{\cite{BamHa,costabel, gimperleinreview, haduong}. We refer to \cite{jr} for a detailed discussion of the challenges in the analysis for $\sigma=0$.}

Using a conforming  ansatz space, we are able to derive {an a priori estimate} for the variational inequality. {It is} the hyperbolic counterpart of the elliptic {estimate} proved by Falk \cite{falk}.
\begin{theorem}
\label{apriori1}
 Let $h \in H^{\frac{3}{2}}_\sigma(\mathbb{R}^+,{H}^{-\frac{1}{2}}(G))$ and let $u \in H^{\frac{1}{2}}_\sigma(\mathbb{R}^+,\tilde{H}^{\frac{1}{2}}(G))^+$, respectively $u_{\Delta t,h}\in \tilde{K}_{t,h}^+\textcolor{black}{\subset H^{\frac{1}{2}}_\sigma(\mathbb{R}^+,\tilde{H}^{\frac{1}{2}}(G))^+ }$ be the solutions of (\ref{contVI}), respectively (\ref{discreteVI}). Then the following estimate holds:
 \begin{align}\label{aprioriestimateVI}
 \|u-u_{\Delta t,h}\|_{-\frac{1}{2},\frac{1}{2},\sigma,\star }^2 \lesssim_\sigma \inf \limits_{\phi_{\Delta t,h}\in \tilde{K}_{t,h}^+}(\textcolor{black}{\|h-p_Q \gamadi_\sigma u\|_{\frac{1}{2},-\frac{1}{2},\sigma}} \|u-\phi_{\Delta t,h}\|_{-\frac{1}{2},\frac{1}{2},\sigma,\star}+ 
 \|u-\phi_{\Delta t,h}\|_{\frac{1}{2},\frac{1}{2},\sigma,\star}^2).
\end{align}  
 \end{theorem}
 \begin{proof}
Rewriting (\ref{contVI}) and (\ref{discreteVI}), we note that
 \begin{align}\label{umgestellt}
 \langle p_Q\gamadi_\sigma u,u \rangle_{\textcolor{black}{\sigma}} \leq \langle h,u-\phi\rangle + \langle p_Q\gamadi_\sigma u, \phi \rangle_{\textcolor{black}{\sigma}}
 \end{align}
 and
 \begin{align}\label{discreteUmgestellt}
  \langle p_Q \gamadi_\sigma u_{\Delta t,h},u_{\Delta t,h} \rangle_{\textcolor{black}{\sigma}} \leq \langle h, u_{\Delta t,h}-\phi_{\Delta t,h} \rangle_{\textcolor{black}{\sigma}} + \langle p_Q \gamadi_\sigma u_{\Delta t,h}, \phi_{\Delta t,h} \rangle_{\textcolor{black}{\sigma}}.
 \end{align}
Using the coercivity in the $\|\cdot\|_{-\frac{1}{2},\frac{1}{2},\sigma,\star}$-norm, \textcolor{black}{as stated in Theorem \ref{mappingProperties}}, we obtain
\begin{align*}
 \|u-u_{\Delta t,h}\|_{-\frac{1}{2},\frac{1}{2},\sigma,\star }^2 & \lesssim_\sigma \langle p_Q\gamadi_\sigma(u-u_{\Delta t,h}),u-u_{\Delta t,h}\rangle_{\textcolor{black}{\sigma}} \\ & \leq \langle h,u-\phi\rangle_{\textcolor{black}{\sigma}}+ \langle h, u_{\Delta t,h}-\phi_{\Delta t,h} \rangle_{\textcolor{black}{\sigma}}   + \langle p_Q \gamadi_\sigma u_{\Delta t,h}, \phi_{\Delta t,h} \rangle_{\textcolor{black}{\sigma}} + \langle p_Q\gamadi_\sigma u, \phi \rangle_{\textcolor{black}{\sigma}} 
\\ & \qquad-\langle p_Q \gamadi_\sigma u,u_{\Delta t,h} \rangle_{\textcolor{black}{\sigma}} - \langle p_Q \gamadi_\sigma u_{\Delta t,h},u \rangle_{\textcolor{black}{\sigma}}
\\ &= \langle h,u-\phi\rangle +\langle h, u_{\Delta t,h}-\phi_{\Delta t,h} \rangle_{\textcolor{black}{\sigma}} + \langle p_Q \gamadi_\sigma u, \phi -u_{\Delta t,h} \rangle_{\textcolor{black}{\sigma}} \\ &\qquad+ \langle p_Q \gamadi_\sigma u_{\Delta t,h}, \phi_{\Delta t,h}-u\rangle_{\textcolor{black}{\sigma}} .
\end{align*}
We rewrite $$\langle p_Q \gamadi_\sigma u_{\Delta t,h}, \phi_{\Delta t,h}-u\rangle_{\textcolor{black}{\sigma}} = \langle p_Q \gamadi_\sigma u- u_{\Delta t,h}, u-\phi_{\Delta t,h}\rangle_{\textcolor{black}{\sigma}} - {\langle p_Q \gamadi_\sigma u, u-\phi_{\Delta t,h}\rangle_{\textcolor{black}{\sigma}}}\ ,$$ so that
\begin{align*}
\|u-u_{\Delta t,h}\|_{-\frac{1}{2},\frac{1}{2},\sigma,\star }^2 &\lesssim_\sigma 
\langle h -p_Q \gamadi_\sigma u,u-\phi_{\Delta t,h}\rangle_{\textcolor{black}{\sigma}} +\langle h -p_Q \gamadi_\sigma u_{\Delta t,h}, u_{\Delta t,h}-\phi \rangle_{\textcolor{black}{\sigma}} 
\\ &\qquad+  \langle p_Q \gamadi_\sigma u- u_{\Delta t,h}, u-\phi_{\Delta t,h}\rangle_{\textcolor{black}{\sigma}}.
\end{align*}
Because of the conforming discretization, we may choose $\phi=u_{\Delta t,h}$ and conclude
\begin{align}
\label{IntoEq}
\|u-u_{\Delta t,h}\|_{-\frac{1}{2},\frac{1}{2},\sigma,\star }^2
 & \lesssim_\sigma 
\langle h -p_Q \gamadi_\sigma u,u-\phi_{\Delta t,h}\rangle_{\textcolor{black}{\sigma}} +  \langle p_Q \gamadi_\sigma u- u_{\Delta t,h}, u-\phi_{\Delta t,h}\rangle_{\textcolor{black}{\sigma}}.
\end{align}
We estimate both terms by duality:
\begin{align*}
\|u-u_{\Delta t,h}\|_{-\frac{1}{2},\frac{1}{2},\sigma,\star}^2 &\lesssim_\sigma \|h-p_Q \gamadi_\sigma u\|_{\frac{1}{2},-\frac{1}{2},\sigma}\|u-\phi_{\Delta t,h}\|_{-\frac{1}{2},\frac{1}{2},\sigma, \star }\\ & \qquad+ \|p_Q \gamadi_\sigma(u-u_{\Delta t,h})\|_{-\frac{1}{2},-\frac{1}{2},\sigma}\|u-\phi_{\Delta t,h}\|_{\frac{1}{2},\frac{1}{2},\sigma, \star }\ .
\end{align*}
From the \textcolor{black}{continuity $p_Q \gamadi_\sigma: {H}^{-{\frac{1}{2}}}_\sigma(\R^+, \tilde{H}^{{\frac{1}{2}}}(G)) \to {H}^{-{\frac{1}{2}}}_\sigma(\R^+, {H}^{{-\frac{1}{2}}}(G))$, see Theorem \ref{mappingProperties} with $s=-\frac{1}{2}$,} one then sees that
\begin{align*}
&\|u-u_{\Delta t,h}\|_{-\frac{1}{2},\frac{1}{2},\sigma,\star }^2\lesssim_\sigma \|h-p_Q \gamadi_\sigma u\|_{\frac{1}{2},-\frac{1}{2},\sigma }\|u-\phi_{\Delta t,h}\|_{-\frac{1}{2},\frac{1}{2},\sigma, \star }\\ &\qquad \qquad\qquad\qquad\qquad\qquad +\|u-u_{\Delta t,h}\|_{-\frac{1}{2},\frac{1}{2},\sigma, \star }\|u-\phi_{\Delta t,h}\|_{\frac{1}{2},\frac{1}{2},\sigma, \star }.
\end{align*}
We conclude with the help of Young's inequality
\begin{align*}
\|u-u_{\Delta t,h}\|_{-\frac{1}{2},\frac{1}{2},\sigma,\star}^2 \lesssim_\sigma \textcolor{black}{\|h-p_Q \gamadi_\sigma u\|_{\frac{1}{2},-\frac{1}{2},\sigma}} \|u-\phi_{\Delta t,h}\|_{-\frac{1}{2},\frac{1}{2},\sigma,\star}+ \|u-\phi_{\Delta t,h}\|_{\frac{1}{2},\frac{1}{2},\sigma,\star}^2.
\end{align*}
Taking the infimum \textcolor{black}{over all $\phi_{\Delta t,h}$} yields the assertion.
\end{proof}
\textcolor{black}{In general, the estimate contains an additional consistency term $\gamadi_{\sigma} - \gamadi_{\Delta t, h}$ from the discretization error of the Dirichlet-to-Neumann operator, which would also appear in Theorem \ref{aprioriEq} for the variational equality and Theorem \ref{mixedtheorem} for the mixed formulation. It is known to be small for the time-independent problem \cite{gwinsteph} and neglected as in \cite{eck}.}

\textcolor{black}{The theorem implies explicit convergence rates \textcolor{black}{for the proposed boundary element method, using results for the best approximation of the solution $u$ in the anisotropic Sobolev space by the piecewise polynomial functions $V_{t,h}^{p,q}$}, as stated e.g.~in \cite{glaefke}.
\begin{corollary}\label{convrate}
Let $u\in H^{\frac{1}{2}+\epsilon}_\sigma(\mathbb{R}^+,\tilde{H}^{\frac{1}{2}+\epsilon}(G))$ for some $\epsilon>0$ \textcolor{black}{and $\mathcal{T}_{S,T} = \mathcal{T}_{S}\times \mathcal{T}_{T}$ a shape regular space-time mesh}. \textcolor{black}{Then} 
\begin{align*}
\|u-u_{\Delta t,h}\|_{-\frac{1}{2},\frac{1}{2},\sigma,\star}^2\lesssim_\sigma (h^{\epsilon} + (\Delta t)^{\frac{1}{2}+\epsilon})\|u\|_{\frac{1}{2}+\epsilon,\frac{1}{2}+\epsilon,\sigma} + (h^{2\epsilon}+ (\Delta t)^{2\epsilon})\|u\|_{\frac{1}{2}+\epsilon,\frac{1}{2}+\epsilon,\sigma}^2.
\end{align*}
\end{corollary}
\begin{proof}
We estimate the first term $\|u-\phi_{\Delta t,h}\|_{-\frac{1}{2},\frac{1}{2},\sigma,\star}$ on the right hand side in Theorem \ref{apriori1} by $\|u-\phi_{\Delta t,h}\|_{0,\frac{1}{2},\sigma,\star}$. Then we apply Proposition 3.56 in \cite{glaefke}. \textcolor{black}{While the Proposition is stated for a quasi-uniform mesh $\mathcal{T}_{S}$ and a uniform time step there, it extends to shape regular meshes as in \cite{karmel}.}
\end{proof}
\textcolor{black}{In particular, Corollary \ref{convrate} applies to locally quasi-uniform meshes which are of the product form $\mathcal{T}_{S,T} = \mathcal{T}_{S}\times \mathcal{T}_{T}$.}}
\textcolor{black}{
We also consider the continuous and discrete variational equalities for the Dirichlet-to-Neumann operator.\\ 
Find $u\in H^{\frac{1}{2}}_\sigma(\mathbb{R}^+,\tilde{H}^{\frac{1}{2}}(G))$ such that
\begin{align}\label{contVE}
\langle p_Q \gamadi_\sigma u_\sigma, v \rangle_{\textcolor{black}{\sigma}} = \langle h,v \rangle_{\textcolor{black}{\sigma}}
\end{align}
holds for all $v\in {H}^{-\frac{1}{2}}_\sigma(\R^+, \tilde{H}^{\frac{1}{2}}(G))$.\\
 Find $u_{\Delta t,h} \in \tilde{V}_{t,h}^{p,q}$ such that 
\begin{align}
\label{discreteeq}
 \langle p_Q \gamadi_{\textcolor{black}{\Delta t,h}} u_{\Delta t,h}, \phi_{\Delta t,h} \rangle_{\textcolor{black}{\sigma}} =\langle h, \phi_{\Delta t,h} \rangle_{\textcolor{black}{\sigma}}
\end{align}
holds for all $\phi_{\Delta t,h} \in \tilde{V}_{t,h}^{{p},{q}}$. \\
}

An a priori estimate for the variational equality is obtained from the previous arguments as a special case.
\begin{theorem}
\label{aprioriEq}
 Let $u\in H^{\frac{1}{2}}_\sigma(\mathbb{R}^+,\tilde{H}^{\frac{1}{2}}(G))$ and $u_{\Delta t,h} \in \tilde{V}_{t,h}^{p,q}\subset H^{\frac{1}{2}}_\sigma(\mathbb{R}^+,\tilde{H}^{\frac{1}{2}}(G))$ be the solutions of \eqref{contVE}, respectively \eqref{discreteeq}. We have the following a priori estimate:
 \begin{align*}
  \|u-u_{\Delta t,h}\|_{-\frac{1}{2},\frac{1}{2},\sigma,\star} \lesssim_\sigma \inf \limits_{\phi_{\Delta t,h} \in \widetilde{V}_{t,h}^{p,q}} \|u-\phi_{\Delta t,h} \|_{\frac{1}{2},\frac{1}{2},\sigma,\star}\ .
 \end{align*}
\end{theorem}
\begin{proof}
For the conforming test space $\tilde{V}_{t,h}^{p,q}$ we have 
\begin{align*}
\langle p_Q \gamadi_\sigma u, \phi_{\Delta t,h} \rangle_{\textcolor{black}{\sigma}}= \langle h, \phi_{\Delta t,h} \rangle_{\textcolor{black}{\sigma}}, 
\end{align*}
for all $\phi_{\Delta t,h} \in \tilde{V}_{t,h}^{p,q}$. Therefore (\ref{IntoEq}) becomes
\begin{align*}
\|u-u_{\Delta t,h}\|_{-\frac{1}{2},\frac{1}{2},\sigma,\star }^2
 & \lesssim_\sigma 
\langle h -p_Q \gamadi_\sigma u,u-\phi_{\Delta t,h}\rangle_{\textcolor{black}{\sigma}} + \langle p_Q \gamadi_\sigma (u- u_{\Delta t,h}), u-\phi_{\Delta t,h}\rangle_{\textcolor{black}{\sigma}}
\\ &= \langle p_Q \gamadi_\sigma (u- u_{\Delta t,h}), u-\phi_{\Delta t,h}\rangle_{\textcolor{black}{\sigma}}\\
& \leq \|p_Q \gamadi_\sigma(u-u_{\Delta t,h})\|_{-\frac{1}{2},-\frac{1}{2},\sigma}\|u-\phi_{\Delta t,h}\|_{\frac{1}{2},\frac{1}{2},\sigma, \star} \ .
\end{align*}
The continuity of $p_Q\gamadi_\sigma$, $\|p_Q \gamadi_\sigma(u-u_{\Delta t,h})\|_{-\frac{1}{2},-\frac{1}{2},\sigma} \lesssim \|u-u_{\Delta t,h}\|_{-\frac{1}{2},\frac{1}{2},\sigma, \ast}$, yields the assertion.
\end{proof}
\textcolor{black}{Analogous to Corollary \ref{convrate} we derive the following rate of convergence \textcolor{black}{for shape regular meshes}.
\begin{corollary}
Let $u\in H^{\frac{1}{2}+\epsilon}_\sigma(\mathbb{R}^+,\tilde{H}^{\frac{1}{2}+\epsilon}(G))$ for some $\epsilon>0$ \textcolor{black}{and $\mathcal{T}_{S,T} = \mathcal{T}_{S}\times \mathcal{T}_{T}$ a shape regular space-time mesh}. \textcolor{black}{Then}
\begin{align*}
 \|u-u_{\Delta t,h}\|_{-\frac{1}{2},\frac{1}{2},\sigma,\star}^2\lesssim_\sigma ( h^{2\epsilon} + (\Delta t)^{2\epsilon}) \|u\|_{\frac{1}{2}+\epsilon, \frac{1}{2}+ \epsilon,\sigma}^2\ .
\end{align*}
\end{corollary}
}

\section{Mixed formulation}\label{mixedsection}

{We reformulate the variational inequality as  an equivalent mixed system. The Lagrange multiplier $\lambda= \gamadi_\sigma u -h$ in this formulation provides a measure to which extent the variational inequality is not an equality; physically, \textcolor{black}{$\lambda$ is the difference between the traction $\gamadi_\sigma u = -\mu\frac{\partial w_\sigma}{\partial \nu}$ of the elastic body and the prescribed forces $h$ in the contact area. This difference corresponds to the additional forces due to contact, and  their non-vanishing identifies the contact area within the computational domain. Both the contact forces and a precise knowledge of the contact area are of interest in applications, which motivates mixed methods as they compute $\lambda$ in addition to $u$.}
\begin{theorem}[Mixed formulation]
Let $h\in H^{\frac{3}{2}}_\sigma(\mathbb{R}^+,{H}^{-\frac{1}{2}}(G))$. The variational inequality formulation (\ref{contVI}) is equivalent to the following formulation:\\ 
 Find $(u,\lambda)\in H^{\frac{1}{2}}_\sigma(\mathbb{R}^+,\tilde{H}^{\frac{1}{2}}(G))\times H^{\frac{1}{2}}_\sigma(\mathbb{R}^+,{H}^{-\frac{1}{2}}(G))^+$ such that 
\begin{align}
\label{mixedFormulationProof}
\begin{cases}
(a) ~ \langle \gamadi_\sigma u, v\rangle_{\textcolor{black}{\sigma}} - \langle \lambda,v\rangle_{\textcolor{black}{\sigma}}= \langle h,v \rangle_{\textcolor{black}{\sigma}} \\
(b)~\langle u ,\mu- \lambda \rangle_{\textcolor{black}{\sigma}} \geq 0,
\end{cases}
\end{align}
for all $(v,\mu)\in H^{\frac{1}{2}}_\sigma(\mathbb{R}^+,\tilde{H}^{\frac{1}{2}}(G))\times H^{\frac{1}{2}}_\sigma(\mathbb{R}^+,{H}^{-\frac{1}{2}}(G))^+$.
\end{theorem}
\begin{proof}
We first note, that (\ref{contVI}) is equivalent to the following problem.  Find $u\in H^{\frac{1}{2}}_\sigma(\mathbb{R}^+,\tilde{H}^{\frac{1}{2}}(G))^+$ solving 
\begin{align}
\label{contVIequiv}
 \begin{cases} 
(a)~ \langle \gamadi_\sigma u,u \rangle_{\textcolor{black}{\sigma}} = \langle h,u \rangle_{\textcolor{black}{\sigma}} \\
 (b)~\langle \gamadi_\sigma u, v \rangle_{\textcolor{black}{\sigma}} \geq \langle h, v \rangle_{\textcolor{black}{\sigma}},
 \end{cases}
\end{align}
for all $v\in H^{\frac{1}{2}}_\sigma(\mathbb{R}^+,\tilde{H}^{\frac{1}{2}}(G))^+$. \textcolor{black}{Setting $v=2u$, respectively $v=0$ in the variational inequality (\ref{contVI}), we obtain 
\begin{align*}
 \langle \gamadi_\sigma u,u \rangle_{\textcolor{black}{\sigma}} \geq  \langle h,u \rangle_{\textcolor{black}{\sigma}}\ , \text{ respectively }  \langle \gamadi_\sigma u,u \rangle_{\textcolor{black}{\sigma}} \leq  \langle h,u \rangle_{\textcolor{black}{\sigma}} \ ,
\end{align*}
so that $\langle \gamadi_\sigma u,u \rangle_{\textcolor{black}{\sigma}} = \langle h,u \rangle_{\textcolor{black}{\sigma}}$}. If we add this to \eqref{contVI} we obtain the second line in \eqref{contVIequiv} To get \eqref{contVI} from \eqref{contVIequiv} we subtract  (\ref{contVIequiv}a) from (\ref{contVIequiv}b). 

We now show the equivalence of \eqref{contVIequiv} and {\eqref{mixedFormulationProof}}: \\ 
 (\ref{contVIequiv}) $\Rightarrow$ {(\ref{mixedFormulationProof})}:
 If we set $\lambda= \gamadi_\sigma u -h$ we have by (\ref{contVIequiv}b) : $\langle \gamadi_\sigma u -h,v\rangle_{\textcolor{black}{\sigma}}  \geq 0$ for all $v\in H^{\frac{1}{2}}_\sigma(\mathbb{R}^+,\tilde{H}^{\frac{1}{2}}(G))^+$ and therefore $\lambda \in H^{\frac{1}{2}}_\sigma(\mathbb{R}^+,{H}^{-\frac{1}{2}}(G))^+$. The first line in {(\ref{mixedFormulationProof})} holds trivially.
 \\ By (\ref{contVIequiv}a) we have that $\langle \lambda,u \rangle_{\textcolor{black}{\sigma}}=0$. Therefore, $\langle u, \mu - \lambda  \rangle_{\textcolor{black}{\sigma}} = \langle u ,\mu \rangle_{\textcolor{black}{\sigma}} \geq 0 $, as $u$ and $\mu$ are positive. \\ {(\ref{mixedFormulationProof})} $\Rightarrow$ (\ref{contVIequiv}): 
 Now let $(u,\lambda) \in  H^{\frac{1}{2}}_\sigma(\mathbb{R}^+,\tilde{H}^{\frac{1}{2}}(G))\times H^{\frac{1}{2}}_\sigma(\mathbb{R}^+,{H}^{-\frac{1}{2}}(G))^+$ be the solution to {(\ref{mixedFormulationProof})}. Setting $\mu=2 \lambda$ and $\mu=0$ yields $\langle u,\lambda \rangle_{\textcolor{black}{\sigma}} \geq 0$, $\langle u,\lambda \rangle_{\textcolor{black}{\sigma}}\leq 0$. Therefore $\langle u, \lambda \rangle_{\textcolor{black}{\sigma}} =0$. \\ Assume that $u$ is not $\geq 0$. Then there exists $\mu \in H^{\frac{1}{2}}_\sigma(\mathbb{R}^+,{H}^{-\frac{1}{2}}(G))^+$ such that $\langle u,\mu \rangle_{\textcolor{black}{\sigma}} <0$, and we obtain the contradiction 
 \begin{align*}
 0 \leq \langle u, \mu - \lambda \rangle_{\textcolor{black}{\sigma}} = \langle u, \mu \rangle_{\textcolor{black}{\sigma}} - \langle u, \lambda \rangle_{\textcolor{black}{\sigma}} = \langle u, \mu \rangle_{\textcolor{black}{\sigma}} <0\ .
 \end{align*}
Therefore $u\in H^{\frac{1}{2}}_\sigma(\mathbb{R}^+,\tilde{H}^{\frac{1}{2}}(G))^+$.
 \\ We now insert $\tilde{v}=v-u$ for  $u,~v\in H^{\frac{1}{2}}_\sigma(\mathbb{R}^+,\tilde{H}^{\frac{1}{2}}(G))^+$ into  {(\ref{mixedFormulationProof}a)}. Note
$\langle v-u, \lambda \rangle_{\textcolor{black}{\sigma}} = \langle v, \lambda\rangle_{\textcolor{black}{\sigma}} - \langle u,\lambda \rangle_{\textcolor{black}{\sigma}}=  \langle v, \lambda\rangle_{\textcolor{black}{\sigma}} \geq 0$.   Inserting $v-u$ in {(\ref{mixedFormulationProof}a)}, we have
$$\langle \gamadi_\sigma u, v-u\rangle_{\textcolor{black}{\sigma}} - \langle \lambda,v-u\rangle_{\textcolor{black}{\sigma}}= \langle h,v-u \rangle_{\textcolor{black}{\sigma}} \ ,$$
or equivalently
$$\langle \gamadi_\sigma u-h, v-u\rangle_{\textcolor{black}{\sigma}} = \langle \lambda,v-u\rangle_{\textcolor{black}{\sigma}} \geq 0.$$
\end{proof}

The discrete formulation reads as follows:\\ 
Find $(u_{\Delta t_1,h_1}, \lambda_{\Delta t_2,h_2}) \in \tilde{V}_{ t_1,h_1}^{1,1} \times (V_{ t_2,h_2}^{0,0})^+$ such that 
\begin{align}\label{discMixedFormulation}\begin{cases} 
(a) ~ \langle \gamadi_{\textcolor{black}{\Delta t_1,h_1}} u_{\Delta t_1,h_1}, v_{\Delta t_1,h_1}\rangle_{\textcolor{black}{\sigma}} - \langle \lambda_{\Delta t_2,h_2} ,v_{\Delta t_1,h_1}\rangle_{\textcolor{black}{\sigma}}= \langle h,v_{\Delta t_1,h_1} \rangle_{\textcolor{black}{\sigma}} \\
(b)~\langle u_{\Delta t_1,h_1} ,\mu_{\Delta t_2,h_2}- \lambda_{\Delta t_2,h_2} \rangle_{\textcolor{black}{\sigma}} \geq 0
\end{cases}
\end{align}
holds for all $(v_{\Delta t_1,h_1},\mu_{\Delta t_2,h_2})\in \tilde{V}_{ t_1,h_1}^{1,1} \times (V_{ t_2,h_2}^{0,0})^+$. \\

Like for the variational inequality \eqref{discreteVI}, we assume that $\gamadi_\sigma$ is computed exactly. Note that, as in the elliptic case, we allow possibly different meshes for the displacement and the Lagrange multiplier. If the meshes for the Lagrange multiplier and the solution are sufficiently different, we obtain a discrete inf-sup condition in the space-time Sobolev spaces:
\begin{theorem}\label{discInfSup}
Let $C>0$ sufficiently small, and $\frac{\max\{h_1, \Delta t_1\}}{\min\{h_2, \Delta t_2\}}<C$. Then there exists $\alpha>0$ such that for all $\lambda_{\Delta t_2, h_2}$:
$$\sup_{\mu_{\Delta t_1, h_1}} \frac{\langle {\mu}_{\Delta t_1, h_1}, \lambda_{\Delta t_2,h_2}\rangle_{\textcolor{black}{\sigma}}}{\|{\mu}_{\Delta t_1,h_1}\|_{ 0,\frac{1}{2}, \sigma, \ast}} \geq \alpha \|\lambda_{\Delta t_2,h_2}\|_{0, -\frac{1}{2}, \sigma }\ .$$
\end{theorem}
{
\begin{proof}
Let $z$ be the solution to the equation $z-\partial_t^2z - \Delta z = 0$ in $\R^+ \times \Omega$ with boundary conditions 
$$\frac{\partial z}{\partial t} = 0 \quad \text{at $t=0$,} \ \ \frac{\partial z}{\partial \nu}=\lambda_{\Delta t_2,h_2}\quad \text{on $\R^+ \times G$},  \ \ z=0\quad \text{on $\R^+ \times \Gamma\setminus \overline{G}$}, \ \ z \to 0 \quad \text{as $t \to \infty$} \ .$$
We know from the coercivity of the Neumann-to-Dirichlet operator, \textcolor{black}{which follows from \eqref{symbolest2}}, that
\begin{align} \label{coerc}
\langle z, \lambda_{\Delta t_,h_2}\rangle_{\textcolor{black}{\sigma}} = \langle z, \partial_\nu z \rangle_{\textcolor{black}{\sigma}} \geq C' \|\partial_\nu z\|_{0, -\frac{1}{2}, \sigma}^2=C' \|\lambda_{\Delta t_2,h_2}\|_{0, -\frac{1}{2}, \sigma}^2.
\end{align}

Let $\delta$ such that $ \|z\|_{0, \frac{1}{2}+\delta, \sigma }<\infty$. From the approximation properties \textcolor{black}{of the space-time anisotropic Sobolev spaces, Proposition 3.56 in \cite{glaefke},} we  note that there exists a function $\hat{\mu}_{\Delta t_1,h_1}$ with
\begin{align*}
\|z - \hat{\mu}_{\Delta t_1,h_1}\|_{0, \frac{1}{2}, \sigma, \ast } \lesssim (\max\{h_1, \Delta t_1\})^\delta \|z\|_{0, \frac{1}{2}+\delta, \sigma }\ .
 \end{align*}
By the continuity of the Neumann-to-Dirichlet operator, \textcolor{black}{from the symbol estimates \eqref{symbolest1},}  the right hand side is estimated by 
\begin{align*}
(h_1 + \Delta t_1)^\delta  \|z\|_{0, \frac{1}{2}+\delta, \sigma} =(\max\{h_1, \Delta t_1\})^\delta  \| \lambda_{\Delta t_2,h_2}\|_{0, -\frac{1}{2}+\delta, \sigma}\ .
\end{align*}
Using an inverse inequality for $\lambda_{\Delta t_2,h_2}$, we conclude
\begin{equation}\label{auxil}
\| z - \hat{\mu}_{\Delta t_1,h_1}\|_{0, \frac{1}{2}, \sigma, G,\ast } \lesssim \frac{(\max\{h_1, \Delta t_1\})^\delta }{(\min\{h_2, \Delta t_2\})^\delta } \| \lambda_{\Delta t_2,h_2}\|_{0, -\frac{1}{2}, \sigma }\ . \end{equation}
From the continuity of the Neumann-to-Dirichlet operator and (\ref{auxil}), we obtain 
\begin{align}\label{estSup}
{\|\hat{\mu}_{\Delta t_1, h_1}\|_{0, \frac{1}{2}, \sigma, \ast }}&{
 \leq \|z - \hat{\mu}_{\Delta t_1, h_1}\|_{0, \frac{1}{2}, \sigma, \ast} + \|z\|_{0, \frac{1}{2}, \sigma,\ast }}\\& {\lesssim \frac{(\max\{h_1, \Delta t_1\})^\delta }{(\min\{h_2, \Delta t_2\})^\delta } \| \lambda_{\Delta t_2,h_2}\|_{0, -\frac{1}{2}, \sigma }+\|z\|_{0, \frac{1}{2},\sigma, \ast}} \nonumber \\ 
& {
\lesssim \frac{(\max\{h_1, \Delta t_1\})^\delta }{(\min\{h_2, \Delta t_2\})^\delta } \| \lambda_{\Delta t_2,h_2}\|_{0, -\frac{1}{2}, \sigma } + \|\lambda_{\Delta t_2,h_2}\|_{0,-\frac{1}{2},\sigma}\ .}
\end{align}
{
Using $\hat{\mu}_{\Delta t_1,h_1}$ from above and  (\ref{estSup}), we now estimate:}
\begin{align*}
{
\sup_{\mu_{\Delta t_1, h_1}} \frac{\langle {\mu}_{\Delta t_1, h_1}, \lambda_{\Delta t_2,h_2}\rangle_{\textcolor{black}{\sigma}}}{\|{\mu}_{\Delta t_1,h_1}\|_{ 0,\frac{1}{2}, \sigma, \ast}}}  &{
\geq   \frac{\langle \hat{\mu}_{\Delta t_1, h_1}, \lambda_{\Delta t_2,h_2}\rangle_{\textcolor{black}{\sigma}}}{\|\hat{\mu}_{\Delta t_1,h_1}\|_{ 0,\frac{1}{2}, \sigma, \ast}}  \gtrsim \frac{\langle \hat{\mu}_{\Delta t_1, h_1}, \lambda_{\Delta t_2,h_2}\rangle_{\textcolor{black}{\sigma}}}{\left(1+\frac{(\max\{h_1, \Delta t_1\})^\delta }{(\min\{h_2, \Delta t_2\})^\delta }\right)\ \|\lambda_{\Delta t_2,h_2}\|_{ 0,-\frac{1}{2}, \sigma}} }\\ &{
= \frac{1}{\left(1+\frac{(\max\{h_1, \Delta t_1\})^\delta }{(\min\{h_2, \Delta t_2\})^\delta }\right)\ \|\lambda_{\Delta t_2,h_2}\|_{ 0,-\frac{1}{2}, \sigma}} \left( \langle  z,\lambda_{\Delta t_2,h_2} \rangle_{\textcolor{black}{\sigma}} - \langle z -\hat{\mu}_{\Delta t_1, h_1},\lambda_{\Delta t_2,h_2}\rangle_{\textcolor{black}{\sigma}} \right)\ .}
\end{align*}
The first term  $\langle z,\lambda_{\Delta t_2,h_2} \rangle_{\textcolor{black}{\sigma}}$ is estimated from below by $\|\lambda_{\Delta t_2,h_2}\|_{0,-\frac{1}{2}, \sigma}$, with the help of (\ref{coerc}), while for the second we have 
\begin{align}
 \langle z -\hat{\mu}_{\Delta t_1, h_1}, \lambda_{\Delta t_2,h_2}\rangle_{\textcolor{black}{\sigma}} \leq \| z -\hat{\mu}_{\Delta t_1, h_1}\|_{0,\frac{1}{2},\sigma,\ast} \|\lambda_{\Delta t_2,h_2}\|_{0,-\frac{1}{2},\sigma} \lesssim \frac{(\max\{h_1, \Delta t_1\})^\delta }{(\min\{h_2, \Delta t_2\})^\delta } \| \lambda_{\Delta t_2, h_2}\|_{0, -\frac{1}{2}, \sigma}^2,
 \end{align}
where we used (\ref{auxil}). Therefore
\begin{align*}
{
\sup_{\mu_{\Delta t_1, h_1}} \frac{\langle {\mu}_{\Delta t_1, h_1}, \lambda_{\Delta t_2,h_2}\rangle_{\textcolor{black}{\sigma}}}{\|{\mu}_{\Delta t_1,h_1}\|_{ 0,\frac{1}{2}, \sigma, \ast}}} \gtrsim \frac{C' - \frac{(\max\{h_1, \Delta t_1\})^\delta }{(\min\{h_2, \Delta t_2\})^\delta } }{1+\frac{(\max\{h_1, \Delta t_1\})^\delta }{(\min\{h_2, \Delta t_2\})^\delta }} \| \lambda_{\Delta t_2, h_2}\|_{0, -\frac{1}{2}, \sigma}\ .
\end{align*}
The assertion follows. 
\end{proof}
\textcolor{black}{For the time-independent elliptic problem, results related to Theorem \ref{discInfSup} may be found in \cite{gatica}.}

\begin{theorem}\label{mixedtheorem}
The discrete mixed formulation \eqref{discMixedFormulation} admits a unique solution. The following a priori estimates hold: 
\begin{align}
\| \lambda -\lambda_{\Delta t_2,h_2} \|_{0,-\frac{1}{2}, \sigma}  &\lesssim  \inf \limits_{\tilde{\lambda}_{\Delta t_2,h_2}} \| \lambda - \tilde{\lambda}_{\Delta t_2,h_2} \|_{0,-\frac{1}{2},\sigma} +(\Delta t_1)^{-\frac{1}{2}}  \|u- u_{\Delta t_1,h_1}\|_{-\frac{1}{2},\frac{1}{2},\sigma, \ast} \ ,\\
\|u-u_{\Delta t_1,h_1}\|_{-\frac{1}{2},\frac{1}{2},\sigma, \ast} &\lesssim_\sigma \inf \limits_{v_{\Delta t_1,h_1}}
\|u-v_{\Delta t_1,h_1}\|_{\frac{1}{2},\frac{1}{2},\sigma, \ast}\nonumber \\& \qquad
+ \inf \limits_{\tilde{\lambda}_{\Delta t_2,h_2}}\left\{\|\tilde{\lambda}_{\Delta t_2,h_2} - \lambda\|_{\frac{1}{2},-\frac{1}{2},\sigma} +\|\tilde{\lambda}_{\Delta t_2,h_2} - {\lambda}_{\Delta t_2,h_2}\|_{\frac{1}{2},-\frac{1}{2},\sigma}\right\} \ . 
\end{align}
\end{theorem}
\begin{proof}
For a fixed mesh, the weak coercivity implies that the discretization of $\mathcal{S}_\sigma$ is positive definite. The existence therefore follows from standard results for elliptic problems.\\
For the a priori estimate we first note that for arbitrary $\tilde{\lambda}_{\Delta t_2, h_2}$ the following identity holds:
\begin{align} \label{lambdaEq}
\langle \lambda_{\Delta t_2,h_2}-\tilde{\lambda}_{\Delta t_2,h_2}, v_{\Delta t_1,h_1} \rangle_{\textcolor{black}{\sigma}} & = \langle \mathcal{S}_\sigma u_{\Delta t_1,h_1}, v_{\Delta t_1,h_1} \rangle_{\textcolor{black}{\sigma}}- \langle h, v_{\Delta t_1,h_1}\rangle_{\textcolor{black}{\sigma}} - \langle \tilde{\lambda}_{\Delta t_2,h_2}, v_{\Delta t_1,h_1} \rangle_{\textcolor{black}{\sigma}} \nonumber \\  &= \langle \mathcal{S}_\sigma u_{\Delta t_1,h_1}, v_{\Delta t_1,h_1} \rangle_{\textcolor{black}{\sigma}} - \langle \mathcal{S}_\sigma u, v_{\Delta t_1,h_1} \rangle_{\textcolor{black}{\sigma}} +\langle  \lambda , v_{\Delta t_1,h_1} \rangle_{\textcolor{black}{\sigma}} -\langle \tilde{\lambda}_{\Delta t_2,h_2}, v_{\Delta t_1,h_1} \rangle_{\textcolor{black}{\sigma}} \nonumber 
\\ &= \langle \mathcal{S}_\sigma (u_{\Delta t_1,h_1} -u),v_{\Delta t_1,h_1} \rangle_{\textcolor{black}{\sigma}} + \langle \lambda - \tilde{\lambda}_{\Delta t_2,h_2}, v_{\Delta t_1,h_1} \rangle_{\textcolor{black}{\sigma}},
\end{align} 
where we made use of (\ref{mixedFormulationProof}) and (\ref{discMixedFormulation}).
By the inf-sup condition (\ref{discInfSup}) and (\ref{lambdaEq}), we have:
\begin{align*}
\alpha \| \lambda_{\Delta t_2 h_2}-\tilde{\lambda}_{\Delta t_2,h_2}\|_{0,-\frac{1}{2},\sigma} &\leq \sup_{ v_{\Delta t_1,h_1}} \frac{ \langle  \lambda_{\Delta t_2,h_2}-\tilde{\lambda}_{\Delta t_2,h_2}, v_{\Delta t_1,h_1} \rangle_{\textcolor{black}{\sigma}}}{\|v_{\Delta t_1,h_1}\|_{0,\frac{1}{2}, \sigma, \ast} }
\\ & =
\sup_{ v_{\Delta t_1,h_1}} \frac{\langle \mathcal{S}_\sigma (u_{\Delta t_1,h_1} -u),v_{\Delta t_1,h_1} \rangle_{\textcolor{black}{\sigma}} + \langle \lambda-\tilde{\lambda}_{\Delta t_2,h_2}, v_{\Delta t_1,h_1} \rangle_{\textcolor{black}{\sigma}}}{\|v_{\Delta t_1,h_1}\|_{0,\frac{1}{2}, \sigma, \ast}}\ .
\end{align*}
We estimate both terms separately. From duality and an inverse inequality in time we obtain for the first term
\begin{align*}
|\langle \mathcal{S}_\sigma (u_{\Delta t_1,h_1} -u),v_{\Delta t_1,h_1} \rangle_{\textcolor{black}{\sigma}}| & \leq \| \mathcal{S}_\sigma(u_{\Delta t_1,h_1}-u)\|_{-\frac{1}{2},-\frac{1}{2},\sigma} \|{v_{\Delta t_1,h_1}}\|_{\frac{1}{2},\frac{1}{2},\sigma, \ast} 
\\ &\lesssim
\| u_{\Delta t_1,h_1} -u\|_{-\frac{1}{2},\frac{1}{2},\sigma,\ast}(\Delta t_1)^{-\frac{1}{2}} \|v_{\Delta t_1,h_1}\|_{0,\frac{1}{2},\sigma, \ast}\ .\end{align*}
A similar argument applied to the second term yields:
\begin{align*}
|\langle \lambda-\tilde{\lambda}_{\Delta t_2,h_2}, v_{\Delta t_1,h_1} \rangle_{\textcolor{black}{\sigma}}|& \leq \| \lambda-\tilde{\lambda}_{\Delta t_2,h_2} \|_{0,-\frac{1}{2},\sigma} \|v_{\Delta t_1,h_1}\|_{0,\frac{1}{2},\sigma, \ast}
\ .
\end{align*} 
We obtain the a priori estimate 
\begin{align}\label{lambdaapriori}
& \| \lambda -\lambda_{\Delta t_2,h_2} \|_{0,-\frac{1}{2}, \sigma}  \lesssim_\sigma  \inf \limits_{\tilde{\lambda}_{\Delta t_2,h_2}} \| \lambda - \tilde{\lambda}_{\Delta t_2,h_2} \|_{0,-\frac{1}{2},\sigma} +(\Delta t_1)^{-\frac{1}{2}}  \| u_{\Delta t_1,h_1}-u\|_{-\frac{1}{2},\frac{1}{2},\sigma, \ast} \ .
\end{align}
Next we combine the Galerkin orthogonality
\begin{equation*}
\langle\gamadi_\sigma(u-u_{\Delta t_1,h_1}),V_{\Delta t_1,h_1}\rangle_{\textcolor{black}{\sigma}} = \langle \lambda -{\lambda}_{\Delta t_2,h_2}, V_{\Delta t_1,h_1}\rangle_{\textcolor{black}{\sigma}}
\end{equation*}
with the coercivity of the Dirichlet-to-Neumann operator to obtain
\begin{align*}
\|u_{\Delta t_1,h_1} - v_{\Delta t_1,h_1}\|_{-\frac{1}{2},\frac{1}{2},\sigma, \ast}^2 &\lesssim_\sigma\langle\gamadi_\sigma(u_{\Delta t_1,h_1}-v_{\Delta t_1,h_1}), u_{\Delta t_1,h_1}- v_{\Delta t_1,h_1}\rangle_{\textcolor{black}{\sigma}} \\
& = \langle\gamadi_\sigma(u-v_{\Delta t_1,h_1}), u_{\Delta t_1,h_1}- v_{\Delta t_1,h_1}\rangle_{\textcolor{black}{\sigma}}+ \langle\gamadi_\sigma(u_{\Delta t_1,h_1}-u), u_{\Delta t_1,h_1}- v_{\Delta t_1,h_1}\rangle_{\textcolor{black}{\sigma}}\\ 
& = \langle \gamadi_\sigma(u-v_{\Delta t_1,h_1}),u_{\Delta t_1,h_1}- v_{\Delta t_1,h_1}\rangle_{\textcolor{black}{\sigma}}\\& \qquad+\langle \tilde{\lambda}_{\Delta t_2,h_2} - \lambda+{\lambda}_{\Delta t_2,h_2}-\tilde{\lambda}_{\Delta t_2,h_2}, u_{\Delta t_1,h_1}-  v_{\Delta t_1,h_1}\rangle_{\textcolor{black}{\sigma}} 
\end{align*}
for all $v_{\Delta t_1,h_1}$ and $\tilde{\lambda}_{\Delta t_2,h_2}$. From the mapping properties and the continuity of the dual pairing, we conclude
\begin{align*}
\|u_{\Delta t_1,h_1} - v_{\Delta t_1,h_1}\|_{-\frac{1}{2},\frac{1}{2},\sigma, \ast}^2 &\lesssim \|u-v_{\Delta t_1,h_1}\|_{\frac{1}{2},\frac{1}{2},\sigma, \ast}\|u_{\Delta t_1,h_1}- v_{\Delta t_1,h_1}\|_{-\frac{1}{2},\frac{1}{2},\sigma, \ast}\\
& \quad +\|\tilde{\lambda}_{\Delta t_2,h_2} - \lambda\|_{\frac{1}{2},-\frac{1}{2},\sigma} \| u_{\Delta t_1,h_1}-  v_{\Delta t_1,h_1}\|_{-\frac{1}{2},\frac{1}{2},\sigma, \ast} \\ & \quad + \|\tilde{\lambda}_{\Delta t_2,h_2}-{\lambda}_{\Delta t_2,h_2}\|_{\frac{1}{2},-\frac{1}{2},\sigma} \| u_{\Delta t_1,h_1}-  v_{\Delta t_1,h_1}\|_{-\frac{1}{2},\frac{1}{2},\sigma, \ast}\ . 
\end{align*}
Therefore,
\begin{align*}
\|u_{\Delta t_1,h_1} - v_{\Delta t_1,h_1}\|_{-\frac{1}{2},\frac{1}{2},\sigma, \ast} &\lesssim_\sigma \|u-v_{\Delta t_1,h_1}\|_{\frac{1}{2},\frac{1}{2},\sigma, \ast}
+\|\tilde{\lambda}_{\Delta t_2,h_2} - \lambda\|_{\frac{1}{2},-\frac{1}{2},\sigma}  + \|\tilde{\lambda}_{\Delta t_2,h_2} - {\lambda}_{\Delta t_2,h_2}\|_{\frac{1}{2},-\frac{1}{2},\sigma}\ . 
\end{align*}
It follows that
\begin{align*}
\|u-u_{\Delta t_1,h_1}\|_{-\frac{1}{2},\frac{1}{2},\sigma, \ast} &\lesssim_\sigma \|u-v_{\Delta t_1,h_1}\|_{\frac{1}{2},\frac{1}{2},\sigma, \ast}
+\|\tilde{\lambda}_{\Delta t_2,h_2} - \lambda\|_{\frac{1}{2},-\frac{1}{2},\sigma} +\|\tilde{\lambda}_{\Delta t_2,h_2} - {\lambda}_{\Delta t_2,h_2}\|_{\frac{1}{2},-\frac{1}{2},\sigma} \ . 
\end{align*}
\end{proof}

\section{A variational inequality for the single layer operator}\label{Vsection}

\textcolor{black}{We now consider the punch problem from Section \ref{physics}, as given by \eqref{stampprob1}, which models a rigid body indenting an elastic half-space. If the shearing strains vanish at $x_n=0$, one may represent $w$ in terms of the single layer potential $$w(t,x) = 2\int_{\mathbb{R}^+ \times \mathbb{R}^{n-1}}  \textcolor{black}{\gamma}(t- \tau,x,y)\ \sigma_{x_n}(\tau,y)\ d\tau\ ds_y\ ,$$ where $\textcolor{black}{\gamma}$ is a fundamental solution to the wave equation. Using this ansatz, the punch conditions become 
$$V \sigma_{x_n} = \textcolor{black}{2} (\phi + \eta), \quad \sigma_{x_n}\geq 0 \quad \text{ in } \mathbb{R}\times G$$
and 
$$V \sigma_{x_n} \leq  \textcolor{black}{2} (\phi + \eta), \quad \sigma_{x_n}= 0 \quad \text{ in } \mathbb{R}\times  \mathbb{R}^{n-1}\setminus \overline{G} \ .$$ The conditions may be equivalently reformulated as a variational inequality
$$\langle p_Q V \sigma_{x_n}, \sigma_{x_n} - v\rangle_{\textcolor{black}{\sigma}} \geq \langle  \textcolor{black}{2}( \phi + \eta), \sigma_{x_n} - v \rangle_{\textcolor{black}{\sigma}} \ ,$$
for arbitrary functions $v \geq 0$ in a suitable Sobolev space. Writing $u=\sigma_{x_n}$ and $h =  \textcolor{black}{2}(\phi + \eta)$, we have the following precise formulation of the punch problem \eqref{stampprob1} as a variational inequality for the single-layer operator $V$:
}

Find $u \in  {H}^{\frac{1}{2}}_\sigma(\R^+, \tilde{H}^{-\frac{1}{2}}(G))^+$ such that for all $v \in  {H}^{\frac{1}{2}}_\sigma(\R^+, \tilde{H}^{-\frac{1}{2}}(G))^+$:
\begin{align}\label{varVI}
\langle p_Q {V}{u},v-u\rangle_{\textcolor{black}{\sigma}} \geq \langle h, v-u \rangle_{\textcolor{black}{\sigma}}\ .
\end{align}
\textcolor{black}{A proof analogous to Theorem \ref{VIequiv} shows:
\begin{theorem}
The punch problem \eqref{stampprob1} is equivalent to the variational inequality \eqref{varVI}. 
\end{theorem}
}
As for the variational inequality for the Dirichlet-to-Neumann operator, a solution exists in the case where $\Omega = \mathbb{R}^n_+$ is the half space:
\textcolor{black}{\begin{theorem}
\label{theoVregularity}
Let $\sigma>0$ and $h\in  H^{\frac{3}{2}}_\sigma(\mathbb{R}^+,{H}^{\frac{1}{2}}(G))$. Then there exists a unique classical solution $u \in H^{\frac{1}{2}}_\sigma(\mathbb{R}^+,\tilde{H}^{-\frac{1}{2}}(G))^+$ of \eqref{varVI}. 
\end{theorem}
\begin{proof}
See \cite{cooper}, p.~456.
\end{proof}
The corresponding discretized variational inequality reads as follows. \\
Find $u_{\Delta t,h}\in \tilde{K}_{t,h}^+$ such that:
\begin{align}
\label{discretevarVI}
  \langle p_Q {V} u_{\Delta t,h}, v_{\Delta t,h}-u_{\Delta t,h} \rangle_{\textcolor{black}{\sigma}} \geq \langle h,v_{\Delta t,h}-u_{\Delta t,h} \rangle_{\textcolor{black}{\sigma}}
\end{align}
holds for all $v_{\Delta t,h}\in \tilde{K}_{t,h}^+$.\\
\textcolor{black}{Note that unlike for the Dirichlet-to-Neumann operator, $V$ does not need to be approximated.}  The relevant a priori estimate reads:
\begin{theorem}
\label{apriori2}
 Let $h \in H^{\frac{3}{2}}_\sigma(\mathbb{R}^+,{H}^{\frac{1}{2}}(G))$, and let $u \in H^{\frac{1}{2}}_\sigma(\mathbb{R}^+,\tilde{H}^{-\frac{1}{2}}(G))^+$, $u_{\Delta t,h}\in \tilde{K}_{t,h}^+$ be the solutions of (\ref{varVI}), respectively (\ref{discretevarVI}). Then the following estimate holds:
 \begin{align}
 \|u-u_{\Delta t,h}\|_{-\frac{1}{2},-\frac{1}{2},\sigma,\star }^2 \lesssim_\sigma \inf \limits_{\phi_{\Delta t,h}\in \tilde{K}_{t,h}^+}(\textcolor{black}{\|h-p_Q V u\|_{\frac{1}{2},\frac{1}{2},\sigma}} \|u-\phi_{\Delta t,h}\|_{-\frac{1}{2},-\frac{1}{2},\sigma,\star}+ 
 \|u-\phi_{\Delta t,h}\|_{\frac{1}{2},-\frac{1}{2},\sigma,\star}^2)\ .
\end{align}  
 \end{theorem}
}
\textcolor{black}{The proof proceeds analogous to the proof of Theorem \ref{apriori1}. It replaces Theorem \ref{mappingProperties} for the Dirichlet-to-Neumann operator by the mapping properties of  $V$ in the half space, Theorem \ref{mapimproved}, and  the coercivity $\|\phi\|_{-\frac{1}{2},-\frac{1}{2},\sigma,\ast}^2 \lesssim_\sigma \langle p_Q V \phi,\phi\rangle_\sigma$ noted there.}

Similarly to the contact problem, for the numerical implementation a mixed formulation of the variational inequality \eqref{varVI} proves useful. Its analysis is analogous to the contact problem.

\begin{theorem}[Mixed formulation]
Let $h\in H^{\frac{3}{2}}_\sigma(\mathbb{R}^+,{H}^{\frac{1}{2}}(G))$. The variational inequality formulation (\ref{varVI}) is equivalent to the following formulation:\\ 
 Find $(u,\lambda)\in H^{\frac{1}{2}}_\sigma(\mathbb{R}^+,\tilde{H}^{-\frac{1}{2}}(G))\times H^{\frac{1}{2}}_\sigma(\mathbb{R}^+,{H}^{\frac{1}{2}}(G))^+$ such that 
\begin{align}
\label{mixedFormulationV}
\begin{cases}
(a) ~ \langle V u, v\rangle_{\textcolor{black}{\sigma}} - \langle \lambda,v\rangle_{\textcolor{black}{\sigma}}= \langle h,v \rangle_{\textcolor{black}{\sigma}} \\
(b)~\langle u ,\mu- \lambda \rangle_{\textcolor{black}{\sigma}} \geq 0,
\end{cases}
\end{align}
for all $(v,\mu)\in H^{\frac{1}{2}}_\sigma(\mathbb{R}^+,\tilde{H}^{-\frac{1}{2}}(G))\times H^{\frac{1}{2}}_\sigma(\mathbb{R}^+,{H}^{\frac{1}{2}}(G))^+$.
\end{theorem}
The discrete formulation reads as follows:\\ 
Find $(u_{\Delta t_1,h_1}, \lambda_{\Delta t_2,h_2}) \in {V}_{ t_1,h_1}^{1,1} \times (V_{ t_2,h_2}^{1,1})^+$ such that 
\begin{align}\label{discMixedFormulationV}\begin{cases} 
(a) ~ \langle V u_{\Delta t_1,h_1}, v_{\Delta t_1,h_1}\rangle_{\textcolor{black}{\sigma}} - \langle \lambda_{\Delta t_2,h_2} ,v_{\Delta t_1,h_1}\rangle_{\textcolor{black}{\sigma}}= \langle h,v_{\Delta t_1,h_1} \rangle_{\textcolor{black}{\sigma}} \\
(b)~\langle u_{\Delta t_1,h_1} ,\mu_{\Delta t_2,h_2}- \lambda_{\Delta t_2,h_2} \rangle_{\textcolor{black}{\sigma}} \geq 0
\end{cases}
\end{align}
holds for all $(v_{\Delta t_1,h_1},\mu_{\Delta t_2,h_2})\in {V}_{ t_1,h_1}^{1,1} \times (V_{ t_2,h_2}^{1,1})^+$.\\

\textcolor{black}{The proof of the following a priori estimate follows the proof of Theorem \ref{mixedtheorem}.
\begin{theorem}\label{mixedVthm}
The mixed formulation \eqref{mixedFormulationV}  and the discrete mixed formulation \eqref{discMixedFormulationV} admit unique solutions $(u, \lambda)$, respectively $(u_{\Delta t_1,h_1}, \lambda_{\Delta t_2,h_2})$. The following a priori estimates hold: 
\begin{align}
\| \lambda -\lambda_{\Delta t_2,h_2} \|_{0,\frac{1}{2}, \sigma}  &\lesssim  \inf \limits_{\tilde{\lambda}_{\Delta t_2,h_2}} \| \lambda - \tilde{\lambda}_{\Delta t_2,h_2} \|_{0,\frac{1}{2},\sigma} +(\Delta t_1)^{-\frac{1}{2}}  \|u- u_{\Delta t_1,h_1}\|_{-\frac{1}{2},-\frac{1}{2},\sigma, \ast} \ ,\\
\|u-u_{\Delta t_1,h_1}\|_{-\frac{1}{2},-\frac{1}{2},\sigma, \ast} &\lesssim_\sigma \inf \limits_{v_{\Delta t_1,h_1}}
\|u-v_{\Delta t_1,h_1}\|_{\frac{1}{2},-\frac{1}{2},\sigma, \ast}\nonumber \\& \qquad
+ \inf \limits_{\tilde{\lambda}_{\Delta t_2,h_2}}\left\{\|\tilde{\lambda}_{\Delta t_2,h_2} - \lambda\|_{\frac{1}{2},\frac{1}{2},\sigma} +\|\tilde{\lambda}_{\Delta t_2,h_2} - {\lambda}_{\Delta t_2,h_2}\|_{\frac{1}{2},\frac{1}{2},\sigma}\right\} \ . 
\end{align}
\end{theorem}
}
\section{Algorithmic considerations}
\label{algo}
\subsection{Marching-on-in-time scheme for the variational equality}
\label{MOT}

A key step towards the solution of the variational inequality is to discretize the Dirichlet-to-Neumann operator and solve equations involving it. In the time-independent case the symmetric representation $\mathcal{S} = \textcolor{black}{\frac{1}{2}}(W-(1-K')V^{-1}(1-K))$ of the Dirichlet-to-Neumann operator  in terms of layer potentials is well-established and studied for contact and coupling problems, see e.g.~\cite{banz, eck, ency}. \textcolor{black}{Also non-symmetric representations of the Dirichlet-to-Neumann operator are of interest, such as the Johnson-Nedelec coupling \cite{ency}; they will be investigated in future work. The implementation of the non-symmetric coupling is simpler as it does not require the hypersingular operator $W$. However,  it  has only recently been analyzed in the time-independent case \cite{sayas0}, in special situations, and not for contact. See also \cite{ajrt, banjai, graded} for preliminary results about the Dirichlet-to-Neumann operator for the wave equation.}

\textcolor{black}{We use the symmetric representation for the  Dirichlet-to-Neumann operator  in the time-dependent case and equivalently formulate the Dirichlet-to-Neumann equation as  follows:}
\\ For given $h\in H^{\frac{3}{2}}_\sigma(\mathbb{R}^+,{H}^{-\frac{1}{2}}(G)) $, find $u\in H^{\frac{1}{2}}_\sigma(\mathbb{R}^+,\tilde{H}^{\frac{1}{2}}(G)) ,v\in {H}^{\frac{1}{2}}_\sigma(\mathbb{R}^+,\tilde{H}^{-\frac{1}{2}}(G)) $ such that
\begin{align}
\timeint \langle Wu - \textcolor{black}{(1-K')}v ,{\phi}\rangle_G \dt =\timeint \langle h,{\phi}\rangle_G\dt\ ,
\\ \timeint [\langle Vv,{\Psi}\rangle_G  - \textcolor{black}{\langle(1-K)}u, {\Psi} \rangle_G] \dt =0,
\end{align}
holds for all $\phi \in {H}^{\frac{1}{2}}_\sigma(\mathbb{R}^+,\tilde{H}^{\frac{1}{2}}(G)) $,$\Psi \in {H}^{\frac{1}{2}}_\sigma(\mathbb{R}^+,\tilde{H}^{-\frac{1}{2}}(G)) $.\\ 
\textcolor{black}{Here $\langle \cdot, \cdot \rangle_G$ denotes the inner product of $L^2(G)$. For the discretization, we use an established marching-in-on-time scheme \cite{mot, terrasse} which allows the solution of the space-time Galerkin system by a time-stepping procedure. To derive it,} let ${u_{\Delta t,h}=\sum \limits_{m,i} c_i^m\beta_{\Delta t}^m(t) \xi_h^i(x,y)} \in \tilde{V}^{1,1}_{t,h}$, $v_{\Delta t,h}=\sum \limits_{m,i} d_i^m \beta_{\Delta t}^m(t) \xi_h^i(x,y) \in V^{1,1}_{t,h}$ be  ansatz functions which are linear in both space and time. To obtain a stable marching-in-on-time scheme we test the first equation against constant test functions in time and the second equation against the time derivative of constant test functions. We let $\phi_{\Delta t,h}^{ij}:= \gamma_{[t_{i-1},t_i]}(t)\xi_h^j(x,y)=: \gamma^i(t)\xi^j(x,y)$, $\dot{\Psi}_{\Delta t,h}^{ij}= \dot{\gamma}^i(t)\xi^j(x,y)$ be test functions that are constant in time and linear in space. 
Thus, the discrete system reads as follows:
\begin{align} \label{discdirneu1}
\timeint \langle Wu_{\Delta t,h} - \textcolor{black}{(1-K')}v_{\Delta t,h} ,{\gamma}^n\xi^j\rangle_G  \dt =\timeint \langle h,{\gamma}^n\xi^j \rangle_G \dt \ ,
\\ \timeint [\langle Vv_{\Delta t,h},\dot{\gamma}^n\xi^j\rangle_G  - \langle \textcolor{black}{(1-K)}u_{\Delta t,h}, \dot{\gamma}^n\xi^j \rangle_G ] \dt =0, \label{discdirneu2}
\end{align}
for all $n=1\ldots,N_t,j=1,\ldots N_s$.

Setting $I^j = I$, if $j=0,1$, $I^j=0$ otherwise and $\hat{I}^0=(-I)$, $\hat{I}^1=I$, $\hat{I}^j=0$ otherwise, {where $I$ is the mass matrix, } we may rewrite the system as\\ 
\begin{align*}
\mathcal{M}^j := \begin{pmatrix}
 W^j && (K^j)'- \frac{\Delta t}{2}I^j\\
 K^j -\hat{I}^j && V^j
 \end{pmatrix},
\end{align*}
for all $j=2,\ldots,N_t$ {(see \cite{gimperleintyre} for further details)}. From the convolution structure in time, we obtain:  For arbitrary $n\in \{1,\ldots,N_t\}$:
\begin{align}
\sum \limits_{m=1}^\infty \mathcal{M}^{n-m}\begin{pmatrix} c^m \\ d^m \end{pmatrix}
= \begin{pmatrix}
\frac{\Delta t}{2}I(h^{n-1}+h^n) \\0
\end{pmatrix}.
\end{align}
By causality, $\mathcal{M}^{j} = 0$ when $j<0$, so that the sum on the left hand side ends at $m=n$. This results in the time stepping scheme:  
\begin{align}\label{MOTcompact}
&\mathcal{M}^0 \begin{pmatrix} c^n \\ d^n \end{pmatrix} = \begin{pmatrix} 
\frac{\Delta t}{2}I(h^{n-1}+h^n)  \\0 
\end{pmatrix} - \sum \limits_{m=1}^{n-1} \mathcal{M}^{n-m} \begin{pmatrix} c^m \\ d^m \end{pmatrix} .
\end{align}
Hence, if we save the matrices from previous time steps, we only need  to calculate one new matrix $\mathcal{M}^{n-1}$ in time step $n$ to obtain the vector $[c^n ,d^n ]^T$.

\subsection{Space-time Uzawa algorithm}

{The solution of the discrete mixed formulation \eqref{discMixedFormulation} may be computed using a Uzawa algorithm in space-time.}

\begin{algorithm}[H]
\caption{Space-time Uzawa algorithm}
\label{alg1}
\begin{algorithmic}
\STATE choose $\rho>0$:
\STATE $k=0:$  $y^{(0)}= \vec{0}$
\WHILE{stopping criterion not satisfied}
\STATE \textbf{solve:} $\mathcal{S}x^{(k)}= h+y^{(k)} $
\STATE \textbf{compute:} $y^{(k+1)}= \text{Pr}_K  (y^{(k)}  - \rho x^{(k)})$, where ($\text{Pr}_K y)_i = \max\{y_i,0\}$
\STATE $k \leftarrow k+1$
\ENDWHILE
\end{algorithmic}
\end{algorithm}

\begin{lemma}
{The space-time Uzawa algorithm converges, provided that $0<\rho< 2C_\sigma$. Here $C_\sigma$ is the coercivity constant in Theorem \ref{mappingProperties}.}
\end{lemma}
\begin{proof}
From the algorithm and the contraction property of the projection \textcolor{black}{$\textrm{Pr}_K \mu_{\Delta t,h} =  \max\{\mu_{\Delta t,h},0\}$} ,
\begin{align*}
\|\lambda_{\Delta t,h}^{(k+1)} - \lambda_{\Delta t,h}\|_{0,0,\sigma}^2 &= \|\text{Pr}_K\ (\lambda_{\Delta t,h}^{(k)} -\rho u_{\Delta t,h}^{(k)}) - \text{Pr}_K\ (\lambda_{\Delta t,h} -\rho u_{\Delta t,h})\|_{0,0,\sigma}^2\\
& \leq \|\lambda_{\Delta t,h}^{(k)} -\lambda_{\Delta t,h} -\rho ( u_{\Delta t,h}^{(k)} - u_{\Delta t,h})\|_{0,0,\sigma}^2\\
& = \|\lambda_{\Delta t,h}^{(k)} -\lambda_{\Delta t,h}\|_{0,0,\sigma}^2 - 2 \rho \langle \lambda_{\Delta t,h}^{(k)} -\lambda_{\Delta t,h} , u_{\Delta t,h}^{(k)} - u_{\Delta t,h}\rangle_{\textcolor{black}{\sigma}} + \rho^2 \|u_{\Delta t,h}^{(k)} - u_{\Delta t,h}\|^2_{0,0,\sigma}\ .
\end{align*}
We conclude that 
\begin{align*}
\|\lambda_{\Delta t,h}^{(k)} - \lambda_{\Delta t,h}\|_{0,0,\sigma}^2 - \|\lambda_{\Delta t,h}^{(k+1)} - \lambda_{\Delta t,h}\|_{0,0,\sigma}^2 & \geq 2 \rho \langle  \lambda_{\Delta t,h}^{(k)} -\lambda_{\Delta t,h}, u_{\Delta t,h}^{(k)} - u_{\Delta t,h}\rangle_{\textcolor{black}{\sigma}} - \rho^2 \|u_{\Delta t,h}^{(k)} - u_{\Delta t,h}\|^2_{0,0,\sigma}\ .
\end{align*}
Further note that 
\begin{align*}
\langle \lambda_{\Delta t,h}^{(k)} -\lambda_{\Delta t,h} , u_{\Delta t,h}^{(k)} - u_{\Delta t,h}\rangle_{\textcolor{black}{\sigma}}  &= \langle \gamadi_\sigma(u_{\Delta t,h}^{(k)} - u_{\Delta t,h}) , u_{\Delta t,h}^{(k)} - u_{\Delta t,h}\rangle_{\textcolor{black}{\sigma}}\\
&\geq C_\sigma \|u_{\Delta t,h}^{(k)} - u_{\Delta t,h}\|_{-\frac{1}{2}, \frac{1}{2},\sigma, \ast}^2\ .
\end{align*}

As $\|u_{\Delta t,h}^{(k)} - u_{\Delta t,h}\|_{-\frac{1}{2}, \frac{1}{2}, \sigma,\ast} \geq \|u_{\Delta t,h}^{(k)} - u_{\Delta t,h}\|_{0, 0,\sigma}$, we conclude that $$\|\lambda_{\Delta t,h}^{(k)} - \lambda_{\Delta t,h}\|_{0,0,\sigma}^2 - \|\lambda_{\Delta t,h}^{(k+1)} - \lambda_{\Delta t,h}\|_{0,0,\sigma}^2 \geq (2 \rho C_\sigma - \rho^2) \|u_{\Delta t,h}^{(k)} - u_{\Delta t,h}\|^2_{0,0,\sigma}\ .$$ The right hand side is non-negative
provided $0<\rho<2C_\sigma$. We conclude that $\|\lambda_{\Delta t,h}^{(k)} - \lambda_{\Delta t,h}\|_{0,0,\sigma}$ is a decreasing sequence. As $\|\lambda_{\Delta t,h}^{(k)} - \lambda_{\Delta t,h}\|_{0,0,\sigma}\geq 0$, it therefore converges, and it follows that $ \|u_{\Delta t,h}^{(k)} - u_{\Delta t,h}\|_{0,0,\sigma} \to 0$.  
\end{proof}

{Practically, we choose the test functions of the mixed discretized formulation \eqref{discMixedFormulation} to be piecewise constant in time to obtain a marching-on-in-time scheme.} In this case, computing the coefficients $x^{(k)}=(\{x_i^m\}_{i,m=1}^{N_s,N_t})^{(k)}$ corresponds to solving the equality for the Dirichlet-to-Neumann operator in Section \ref{MOT}, with an augmented right hand side:
 \begin{align}\label{nonlinS}
 \langle \mathcal{S} \big(\sum \limits_{m,i} (x_i^m)^k\big), \gamma^n \phi^j \rangle = \langle \sum \limits_{m,i} (y_i^m)^{k-1} \beta^m \xi^i, \gamma^n \phi^j \rangle + \langle \sum \limits_{m,i} h_i^m \beta^m \xi^i, \gamma^n \phi^j \rangle ,
 \end{align}
 for all $j=1,\ldots N_s, n=1,\ldots, N_t $. Comparing this system with the system (\ref{MOTcompact}) we observe that (\ref{nonlinS}) may be written as:
 \begin{align}
 \label{stuzawa}
 \sum \limits_{m=1}^n \mathcal{M}^{n-m} \begin{pmatrix}
 (x^m)^k \\ (d^m)^k 
 \end{pmatrix} = \begin{pmatrix}
 \frac{\Delta t}{2}I(h^{n-1}+h^n) \\ 0
 \end{pmatrix} + \begin{pmatrix}
 \frac{\Delta t}{2}I((y^{n-1})^{k-1} + (y^n)^{k-1}) \\ 0
 \end{pmatrix},
 \end{align}
 for $n=1,\ldots,N_t$.
 \textcolor{black}{
 \begin{remark}For a Lagrange multiplier that is constant in time, i.e. $\lambda_{\Delta t,h}=\sum \limits_{m,l} y_l^m \gamma^m(t)\xi^l(x)$, the second term on the right hand side of (\ref{stuzawa}) becomes $\Delta t \begin{pmatrix}
 I (y^n)^{k-1} \\ 0
 \end{pmatrix}$. This will be relevant in the following section.
 \end{remark}
 }

 \subsection{Time-step Uzawa algorithm}
 The space-time Uzawa algorithm solves the whole space-time system in every Uzawa iteration. While it is provably convergent, a time-stepping Uzawa algorithm turns out to be more efficient and will be useful for future adaptive computations. \\
A crucial observation to derive a time-stepping algorithm is that $L^2(Q)$ continuously embeds in $H^{\frac{1}{2}}_\sigma(\mathbb{R}^+,{H}^{-\frac{1}{2}}(G))$. One may therefore use piecewise constant ansatz functions in time, $\lambda_{\Delta t,h}=\sum \limits_{m,i} y_i^m \gamma^m(t)\phi^i(x)\in L^2(Q)\subset H^{\frac{1}{2}}_\sigma(\mathbb{R}^+,{H}^{-\frac{1}{2}}(G))$, for the Lagrange multiplier.

Set
 \begin{align*}
 \tilde{I}_{i,j} = \int \limits_\Gamma \phi^i \phi^j ds_x\ , \quad {I}_{i,j} = \int \limits_\Gamma \xi^i \phi^j ds_x\ ,
 \end{align*}
 and note that $\sum \limits_{m,i} y_i^m \int \limits_0^{\infty} \int \limits_\Gamma  \gamma^m \phi^i \gamma^n \phi^j ds_x \dt=\Delta t \sum \limits_{i} y_i^n \int \limits_\Gamma \phi^i \phi^j ds_x$.
The resulting space-time variational inequality reads:
 \begin{align}
   \begin{pmatrix} \label{spacetimesystem}
\mathcal{M}^0 & &  \\
\mathcal{M}^1 & \mathcal{M}^0 & \\
\vdots & &  \\
\mathcal{M}^{N_t} & \cdots & \mathcal{M}^0
\end{pmatrix} \begin{pmatrix}
[c^1, d^1]^t\\ 
[c^2, d^2]^t \\ 
\vdots \\ 
[c^{N_t}, d^{N_t}]^t
\end{pmatrix}- \Delta t
\begin{pmatrix}
[\tilde{I} y^1, 0]^t\\
[\tilde{I} y^2 ,0]^t \\
\vdots \\
[\tilde{I} y^{N_t}, 0]^t
\end{pmatrix}
= \frac{\Delta t}{2}
\begin{pmatrix}
[I(h^0+h^1), 0]^t\\
[I(h^1+h^2), 0]^t\\
\vdots \\
[I(h^{N_t-1}+h^{N_t}), 0]^t
\end{pmatrix} \ ,
\\ \langle \sum \limits_{m,i} c_i^m \beta^m \xi^i, \mu_{\Delta t,h} - \sum \limits_{m,i} y_i^m \gamma^m \phi^i \rangle \geq 0,\quad \forall \mu_{\Delta t,h}\in (V_{ t,h}^{0,0})^+\ .
\label{stvar}
 \end{align}
With $\mu_{\Delta t,h}=\sum \limits_{m,i} \mu_i^m \gamma^m(t)\phi^i(x)$, we first consider the space-time variational inequality (\ref{stvar}) for a fixed $n=1\ldots,N_t$:
\begin{align*}
 \langle \sum \limits_{m,i} c_i^m \beta^m \xi^i, \sum \limits_{j}(\mu_j^n- y_j^n) \gamma^n \phi^j \rangle &= \frac{\Delta t}{2} \langle \sum \limits_i c_i^n \xi^i,\sum \limits_{j}(\mu_j^n- y_j^n) \phi^j\rangle + \frac{\Delta t}{2}\langle \sum \limits_i c_i^{n-1} \xi^i,\sum \limits_{j}(\mu_j^{n}- y_j^{n})\phi^j \rangle\ , \\
 &= (c^n)^\top \hat{I}(\mu^n-y^n) + (c^{n-1})^\top \hat{I}(\mu^{n}-y^{n}).
\end{align*}
Here we used
\begin{align*}\langle \beta^m \xi^i,\gamma^n \phi^j \rangle=
\int \limits_\Gamma \xi^i(x)\phi^j(x) \int \limits_0^\infty \beta^m(t) \gamma^n(t) dt\ ds_x & = \int \limits_\Gamma \xi^i(x)\phi^j(x) \frac{\Delta t}{2}[\delta_{n,m}+ \delta_{n-1,m}] ds_x \ ,
\end{align*} with $\delta_{n,m}=1$ if $n=m$ and $=0$ otherwise. Also $\hat{I}_{ij}= \frac{\Delta t}{2}I_{ij}$. Therefore (\ref{stvar}) may be written as
\begin{align*}
 (c^1)^\top \hat{I} (\mu^1-y^1) + (c^0)^\top \hat{I} (\mu^1-y^1) &+ (c^2)^\top \hat{I}(\mu^2-y^2) + (c^1)^\top \hat{I} (\mu^2-y^2) + \ldots \\ &+(c^{N_t})^\top \hat{I} (\mu^{N_t}-y^{N_t}) + (c^{N_t-1})^\top \hat{I} (\mu^{N_t}-y^{N_t}) \geq 0 ~\forall \mu^j \quad~\forall j\ ,
\end{align*}
with $c^0=0$. Setting $\mu= (y^1\ldots,y^{k-1},\mu',y^{k+1},\ldots,y^{N_t})$ for $k=1,\ldots,N_t$ yields
\begin{align*}
 (c^{k})^\top \hat{I}(\mu'-y^k)+ (c^{k-1})^\top \hat{I}(\mu'-y^k) \geq 0, \quad~\forall \mu'\ ,
\end{align*}
and we see that the solution to the space-time variational inequality satisfies the  following time-stepping scheme:\\ \\
For $k=1,\ldots,N_t$, find $(c^k,d^k,y^k)$ such that
\begin{align}\label{TSUZAWA}
 &\mathcal{M}^0 
 \begin{pmatrix}
 c^k\\ d^k
 \end{pmatrix} -  \Delta t\begin{pmatrix}
 \tilde{I}y^k \\ 0
 \end{pmatrix}=
  \begin{pmatrix}
 \frac{\Delta t}{2 }{I}(h^{k-1}+h^k) \\ 0
 \end{pmatrix}- \sum \limits_{m=1}^{k-1} \mathcal{M}^{k-m} \begin{pmatrix}
 c^m\\ d^m
 \end{pmatrix}\ ,
 \\ &  (c^{k})^\top \hat{I}(\mu^k-y^k)\geq - (c^{k-1})^\top \hat{I}(\mu^k-y^k) \quad~\forall \mu^k\ .
 \label{TSinequality}
\end{align}
Conversely, if we have solutions to \eqref{TSUZAWA} and \eqref{TSinequality} for every $k=1,\ldots,N_t$, summing \eqref{TSinequality} yields \eqref{stvar}. We conclude:
\begin{lemma}
 The solution to the space-time variational inequality is also a solution to the time-step variational inequality and vice versa. 
\end{lemma}

The resulting time-step Uzawa algorithm is as follows:
\begin{algorithm}[H]
\caption{Time-step Uzawa algorithm}
\label{alg2}
\begin{algorithmic}
\STATE choose $\rho>0$:
\FOR{n=1,\ldots, $N_t$}
\STATE k=0: $(y^n)^0=\vec{0}$
\WHILE{stopping criterion not satisfied}
\STATE \textbf{solve: }$ \mathcal{M}^0 
 \begin{pmatrix}
 c^n\\ d^n
 \end{pmatrix} -  \Delta t\begin{pmatrix}
 \tilde{I}(y^n)^k \\ 0
 \end{pmatrix}=
  \begin{pmatrix}
 \frac{\Delta t}{2 }{I}(h^{n-1}+h^n) \\ 0
 \end{pmatrix}- \sum \limits_{m=1}^{n-1} \mathcal{M}^{n-m} \begin{pmatrix}
 c^m\\ d^m
 \end{pmatrix} $
\STATE \textbf{compute:} $(y^n)^{k+1}=\max\{0,(y^n)^k+\rho((c^n)+ (c^{n-1})^\top \hat{I})\}$
\STATE $k \leftarrow k+1$
\ENDWHILE
\ENDFOR
\end{algorithmic}
\end{algorithm}
{The Uzawa algorithm converges in each time step, as long as $\mathcal{M}^{0}$ is positive definite, and $\rho$ is sufficiently small.}

\begin{remark}
As for the contact problem, the mixed formulation for the punch problem \eqref{mixedFormulationV} may be discretized and solved with the above space-time or time-step Uzawa algorithms. The Dirichlet-to-Neumann operator is here replaced by the single-layer operator.
\end{remark}

\section{Numerical experiments}
\label{experiments}

As stated above, we set $\sigma = 0$ in our computations and discretize the Dirichlet-to-Neumann operator as in Section \ref{MOT}.

\subsection{Dirichlet-to-Neumann operator on unit sphere}
\begin{example}\label{DtNeq}
We solve the discretized variational equality \eqref{discreteeq} for the Dirichlet-to-Neumann operator on $\Gamma = S^2$, with a right hand side obtained from the Neumann data of a known, radially symmetric solution to the wave equation. Specifically, \begin{align*}
 & h=\partial_n u(t,x)\mid_{\{|x|=1\}}=\\ &\textstyle{(-\frac{3}{4}+\cos(\frac{\pi}{2}(4-t))+\frac{\pi}{2}\sin(\frac{\pi}{2}(4-t)) -\frac{1}{4}(\cos({\pi}(4-t))+ \pi\sin({\pi}(4-t))))[H(4-t)-H(-t)]}\ ,
\end{align*}
where $H$ is the Heaviside function. The solution $u$ of the Dirichlet-to-Neumann equation $\mathcal{S}u=h$ corresponds to the Dirichlet data of the solution to the wave equation. Hence,
\begin{align*}
\textstyle{u(t,r)\mid_{\Gamma}=(\frac{3}{4}- \cos(\frac{\pi (4-t)}{2})+\frac{1}{4}\cos({\pi(4-t)}))[H(4-t)-H(-t)].}
\end{align*}
We fix the CFL ratio $\frac{\Delta t}{h}\approx 0.6$ and set $T=5$.
\end{example}

Figure \ref{FigL2} shows the $L^2(\Gamma)$-norm of the exact, respectively numerical solution as a function of time. \textcolor{black}{Figure \ref{Figerr} depicts the error in this norm and shows that the error remains uniformly bounded in time.} As the number of degrees of freedom increases, the $L^2([0,T]\times \Gamma)$-norm of the error between the numerical approximations and the exact solution converges to $0$ at a rate $0.7$, as shown in Figure \ref{DNIcosaeder}.  Here, we compute the experimental convergence rate $\alpha$ as
\begin{align*}
\alpha= \frac{\log E(u_{1}) - \log E(u_{2})}{\log DOF_1 - \log DOF_{2}}\ ,
\end{align*}
\textcolor{black}{where $E(u_j)$ denotes the $L^2([0,T]\times \Gamma)$-norm of the error between the numerical solution $u_j$ and the exact solution $u$. For a fixed CFL ratio $ \frac{\Delta t}{h}$, we have $DOF \sim h^{-3}$. In terms of $h$, we therefore obtain a rate of convergence of $2.1$. Note the approximation error for the geometry in this example.} The results exemplify the convergence of our proposed method to approximate the Dirichlet-to-Neumann operator.  

\begin{figure}[H]
\centering
 \centerline{\includegraphics[width=12cm,height=6cm]{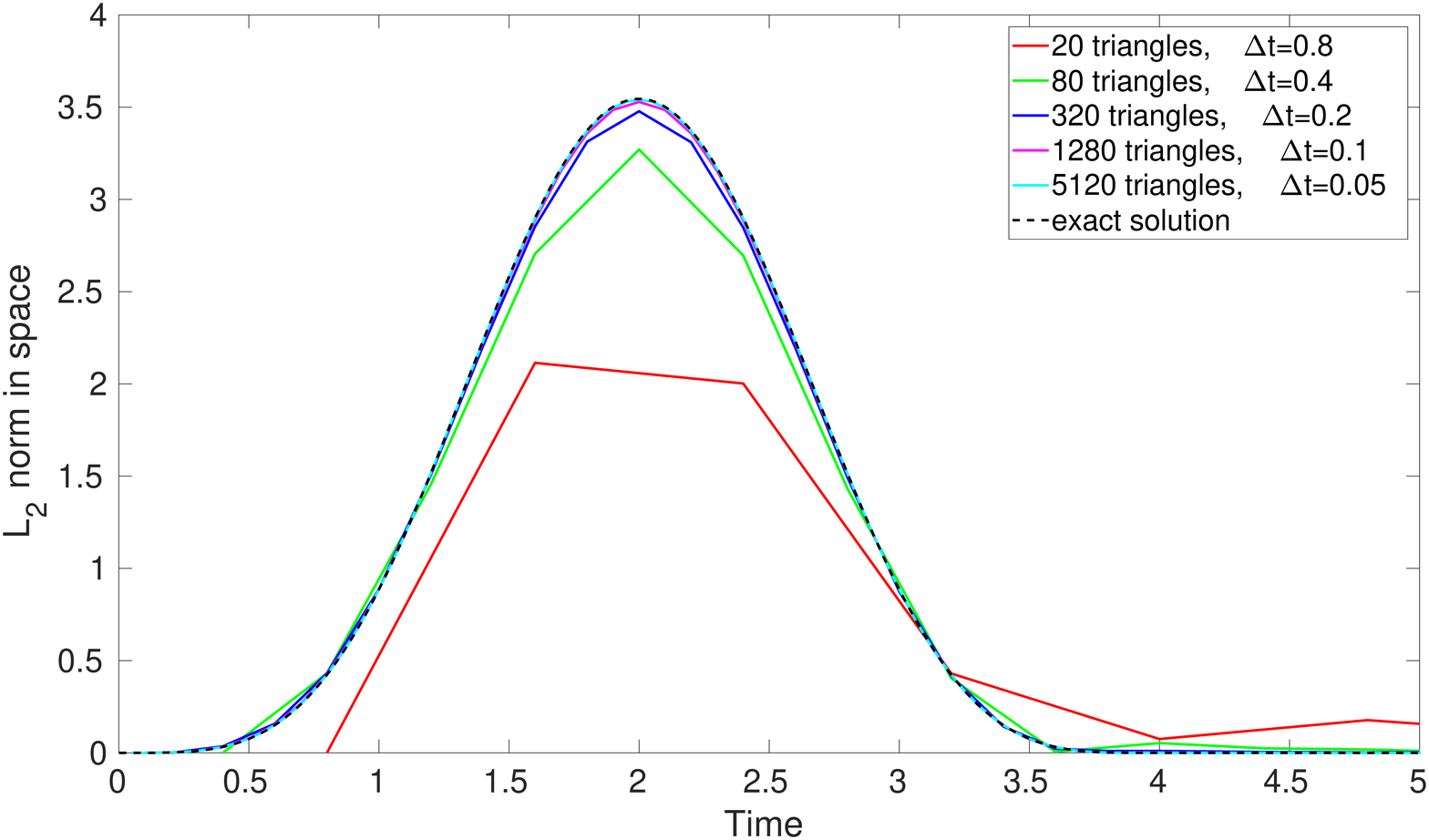}}
 \caption{$L^2(\Gamma)$-norm of the solution to $\mathcal{S}u=h$ for fixed CFL ratio $\frac{\Delta t}{h}\approx 0.6$.}
\label{FigL2}
\end{figure}
\begin{figure}[H]
\centering
 \centerline{\includegraphics[width=12cm,height=6cm]{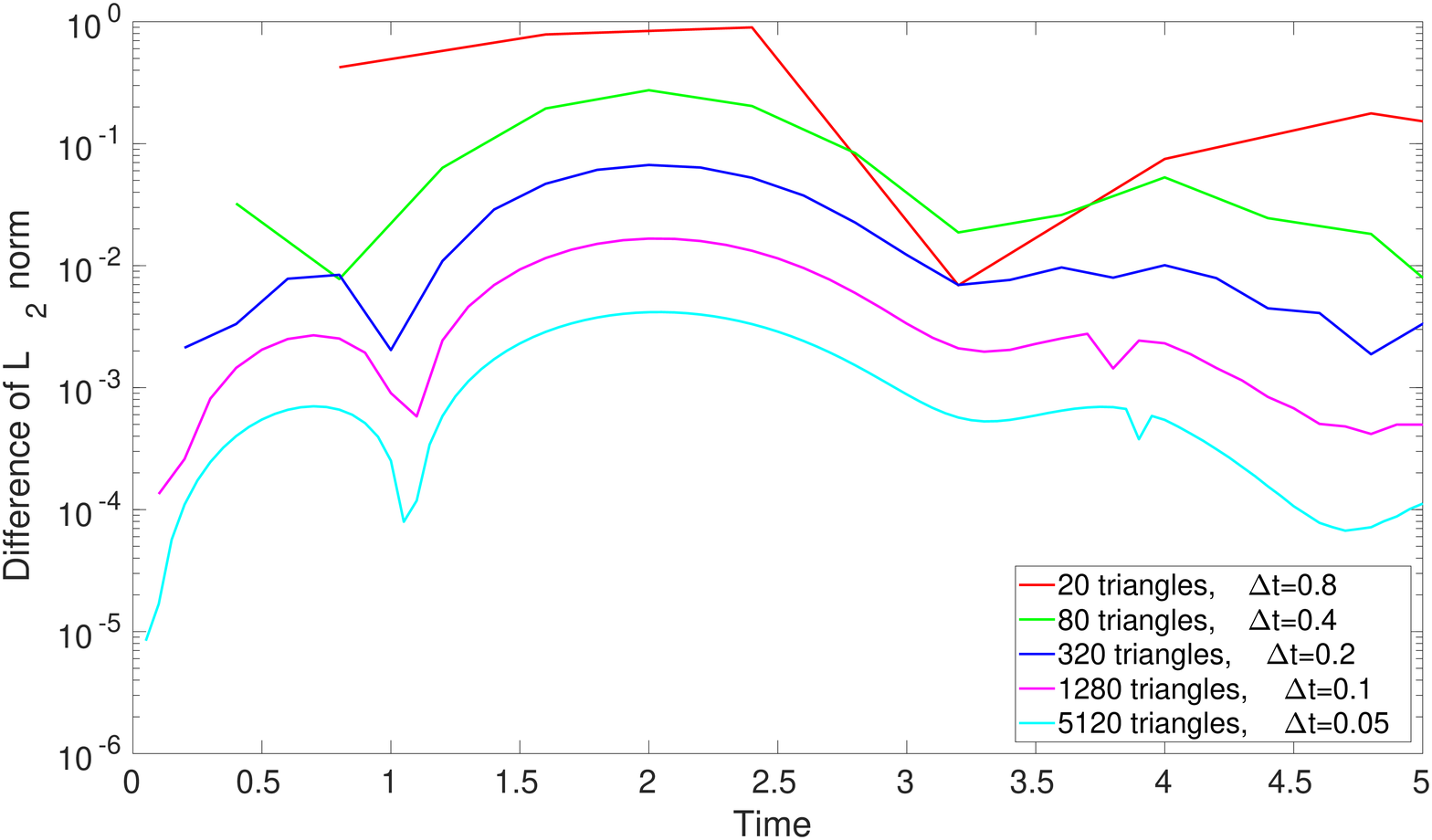}}
 \caption{Absolute error $|\|u\|_{L^2(\Gamma)} - \|u_{\Delta t,h} \|_{L^2(\Gamma)}|$ as a function of time for fixed $\frac{\Delta t}{h}$.}
\label{Figerr}
\end{figure}
\begin{figure}[H]
\centering
 \centerline{\includegraphics[width=12cm,height=6cm]{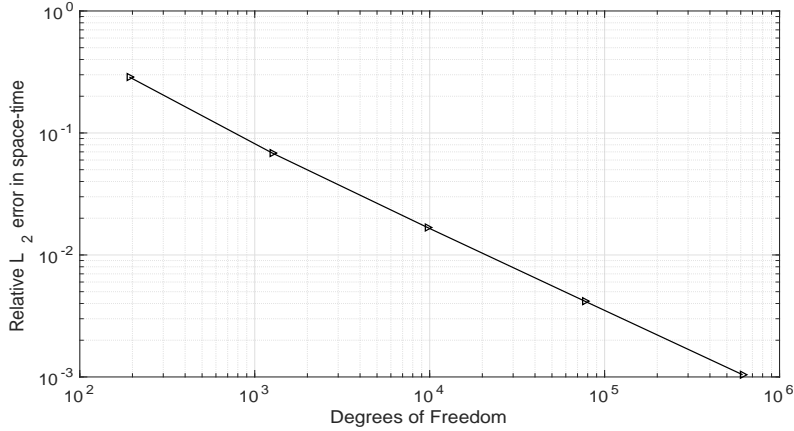}}
 \caption{$L^2([0,T]\times \Gamma)$-error vs.~degrees of freedom of the solution to $\gamadi u=h$ for fixed $\frac{\Delta t}{h}$.}
 \label{DNIcosaeder}
\end{figure}

\subsection{Contact problem: Dirichlet-to-Neumann operator}

We now consider the discretization of the nonlinear contact problem \eqref{varI} for both flat and more general contact areas. In this case, no exact solutions are known, and we compare the numerical approximations to a reference solution on an appropriately finer space-time mesh.

\begin{example} \label{contact1} 
\textcolor{black}{ We choose $\Gamma=[-2,2]^2\times \{0\}$ with contact area $G=[-1,1]^2 \times \{0\}$ for times up to $T=6$, with the CFL ratio \textcolor{black}{$\frac{\Delta t}{h} \approx 1.06$}. The right hand side of the contact problem \eqref{varI} is given by $$h(t,x)=e^{-2t}t\cos(2\pi x)\cos(2\pi y) \chi_{[-0.25,0.25]}(x) \chi_{[-0.25,0.25]}(y)\ .$$  We use the discretization from Section 8.1,  in particular with a Lagrange multiplier that is piecewise linear in space and constant in time.
The numerical solutions are compared to a reference solution on a mesh with $12800$ triangles, and we use $\Delta t=0.075$. }
  \end{example}
The inequality is solved using the space-time Uzawa algorithm \textcolor{black}{as in Algorithm \ref{alg1} and $\rho=20$}. We stop the solver when subsequent iterates have a relative difference of less than $10^{-11}$.\\

Figure \ref{Fig100} shows the solution $u_{\Delta t, h}$ to the contact problem (left column) and the corresponding Lagrange multiplier $\lambda_{\Delta t, h}$ (right column) for several time steps. The solution gradually develops a smooth bump in the center, which gives rise to a radially outgoing wave. A nonvanishing Lagrange multiplier $\lambda= \gamadi_\sigma u -h$ indicates the emergent contact forces at the depicted times \textcolor{black}{$t=4.275$ and $t=5.025$}.

\textcolor{black}{Figure \ref{l2Contactrelative} considers the relative error to the reference solution in $L^2([0,T]\times G)$. The numerical approximations converge at a rate of approximately $\alpha=0.8$ with increasing degrees of freedom.  In terms of $h$, we obtain a rate of convergence of $2.4$. This compares to the rate of convergence for the Dirichlet-to-Neumann equation in Example \ref{DtNeq}, where also the geometry needed to be approximated. 
Algorithmically, the computational cost of the nonlinear solver is dominated by the cost of computing the matrix entries.}

  \begin{figure}[H]
  \makebox[\linewidth]{
   \subfigure[t=0.075]{
\includegraphics[width=9cm, height=4cm]{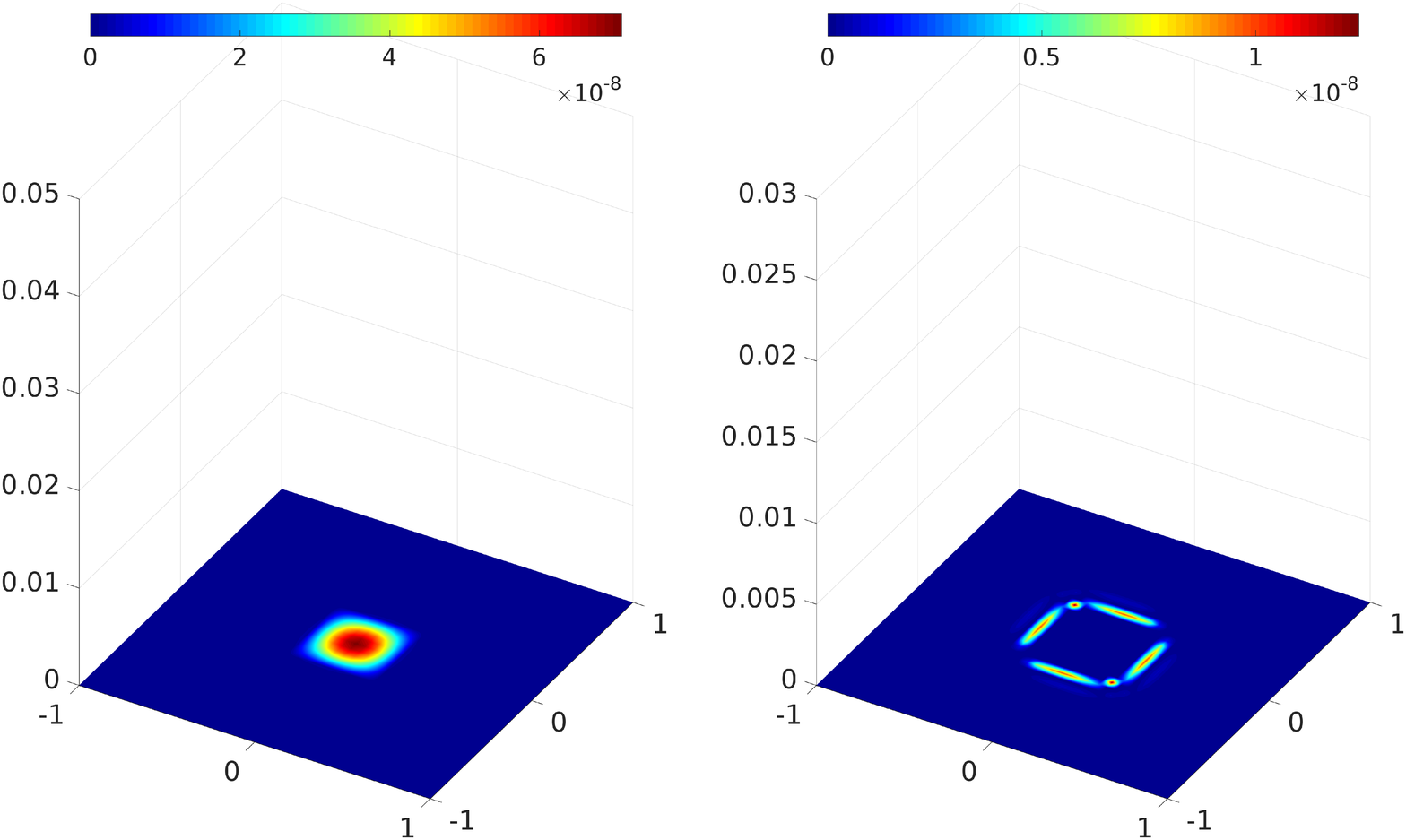} }
\subfigure[t=2.55]{
\includegraphics[width=9cm, height=4cm]{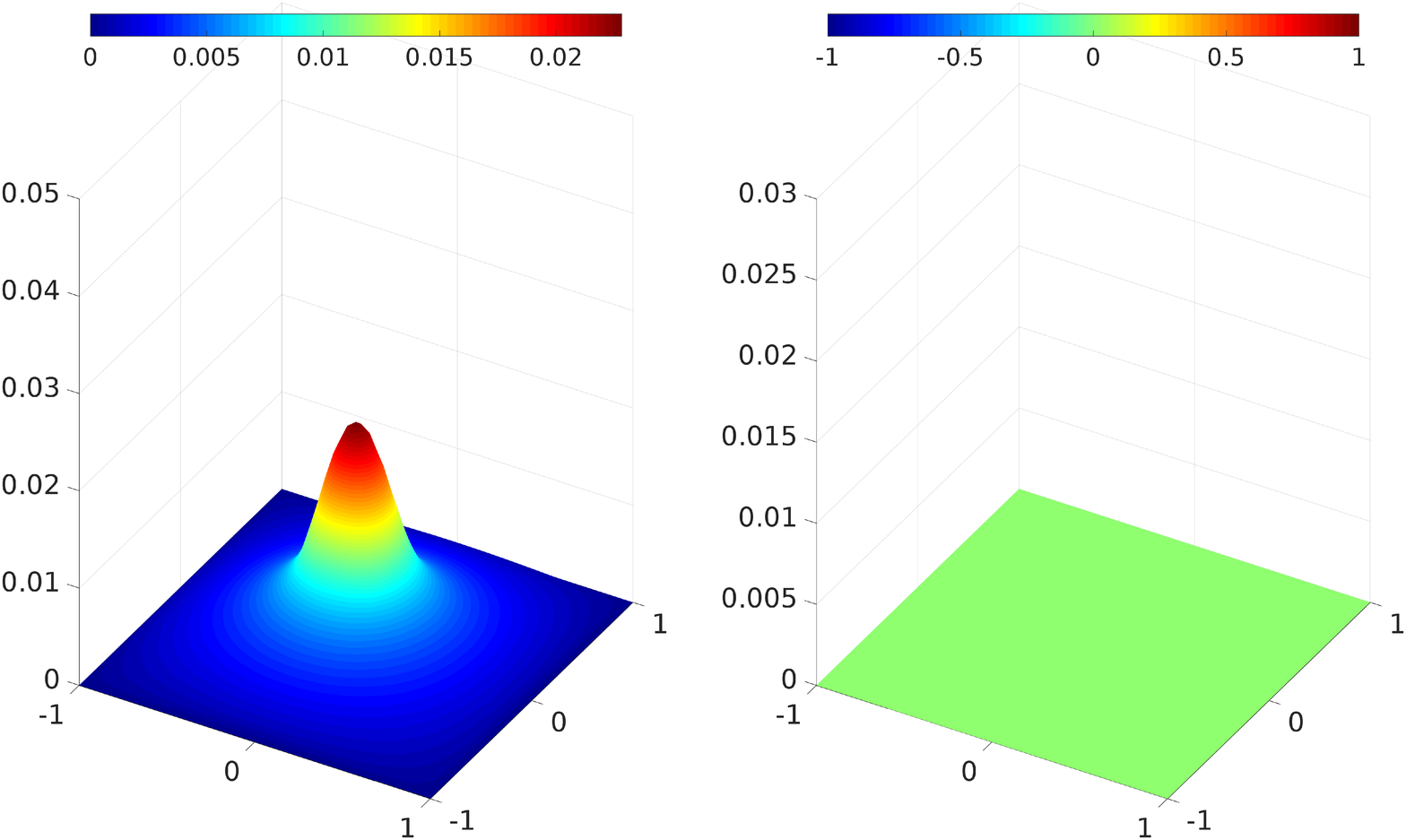} } }
\makebox[\linewidth]{
\subfigure[t=4.275]{
\includegraphics[width=9cm, height=4cm]{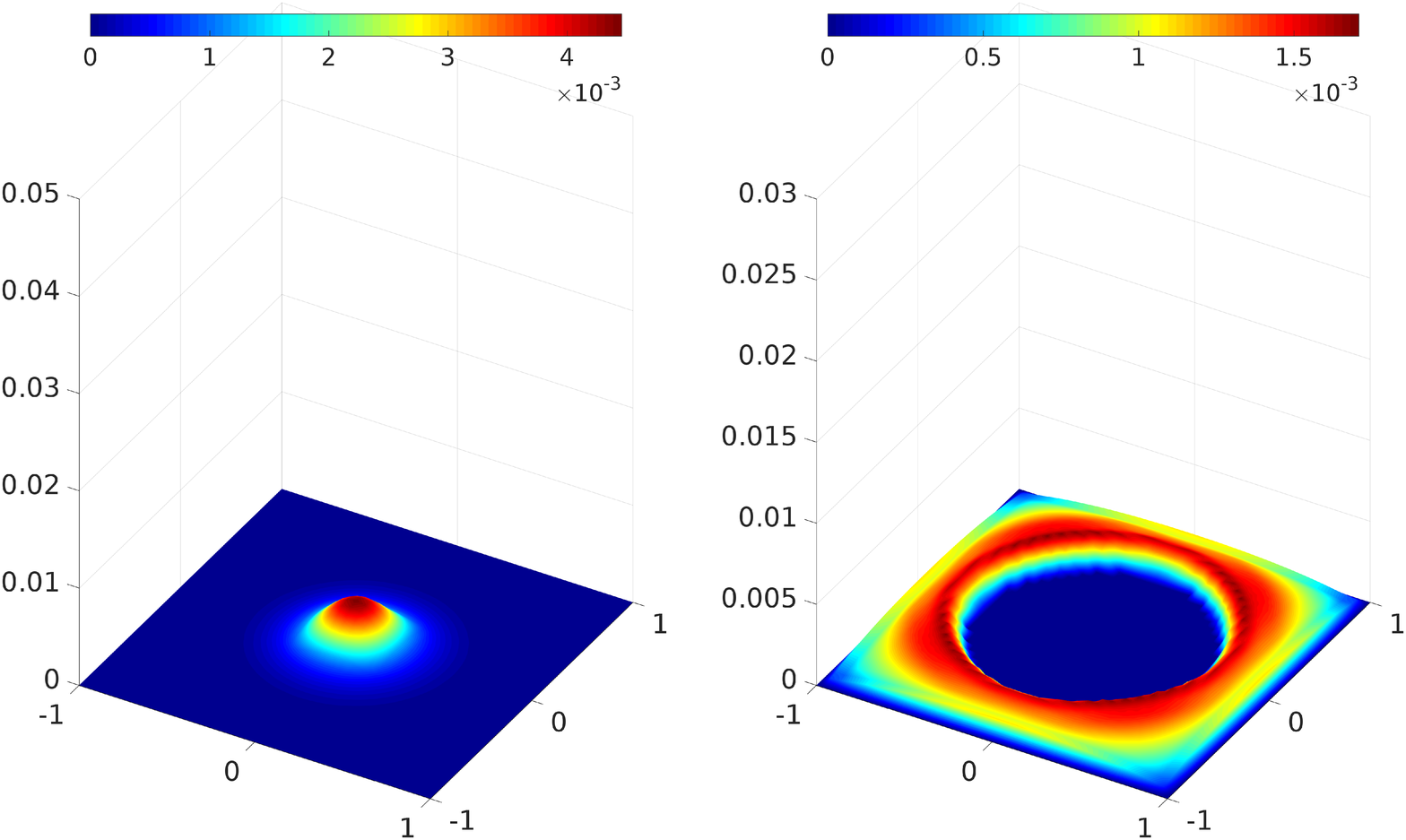} }
\subfigure[t=5.025]{
\includegraphics[width=9cm, height=4cm]{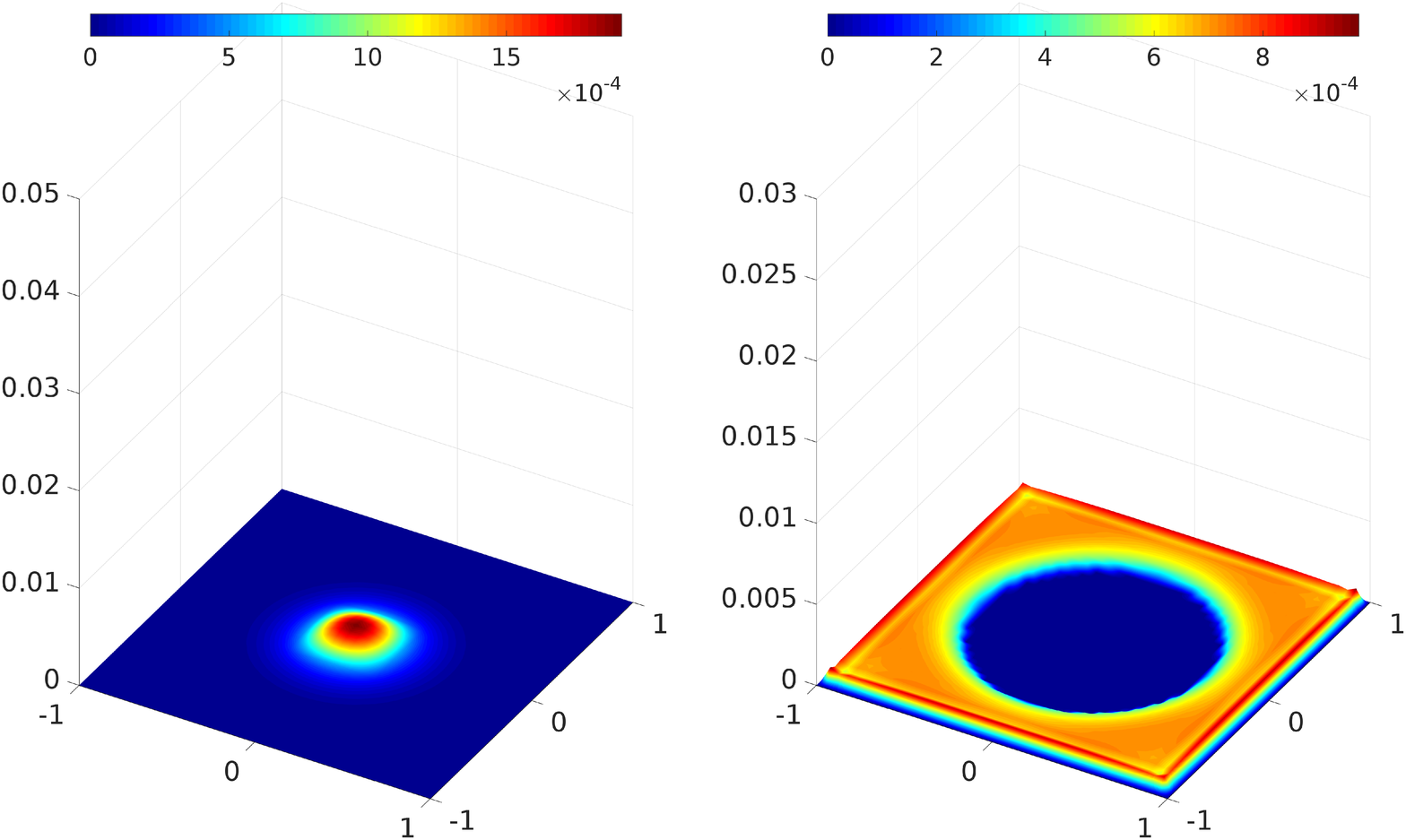} } }
\caption{{Evolution of $u$ and $\lambda$ in $G=[-1,1]^2 \times \{0\}$ for the contact problem,  Example \ref{contact1}.}}
\label{Fig100}
\end{figure}
\begin{figure}[H]
\makebox[\linewidth]{
\includegraphics[width=12cm,height=6cm]{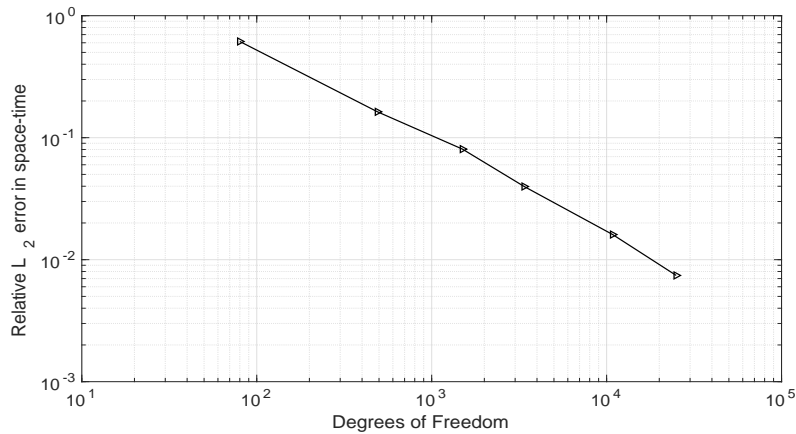} }
\caption{Relative $L^2([0,T]\times \Gamma)$-error vs.~degrees of freedom of the solutions to the contact problem for fixed $\frac{\Delta t}{h}$, {Example \ref{contact1}}.}
\label{l2Contactrelative}
\end{figure}

Next we compare the above space-time Uzawa algorithm with the time-step Uzawa variant \textcolor{black}{from Algorithm \ref{alg2}}. In a given time step, we stop the Uzawa iteration when subsequent iterates either have a relative difference of less than $10^{-12}$ or if the $\ell_\infty$-norm is less than $10^{-10}$. For the space-time Uzawa algorithm we use the stopping criterion from before.
\begin{example}\label{timestepexample}
We consider the contact problem \eqref{varI} with the geometry and right hand side from Example \ref{contact1}. On a fixed mesh of $3200$ triangles, and with $\Delta t=0.1$ and $T=6$, we investigate the difference of the approximate solutions obtained from the space-time and time-step Uzawa algorithms.
\end{example}
Figure \ref{relerror} shows the temporal evolution of the relative difference in $L^2(\Gamma)$ between the two methods in a semi-logarithmic plot. The difference is smaller than $0.01 \%$ for all times, but increases sharply around the onset of contact shortly after time $t=4$.  
\begin{figure}[H]
\makebox[\linewidth]{
\includegraphics[width=12.5cm, height=5cm]{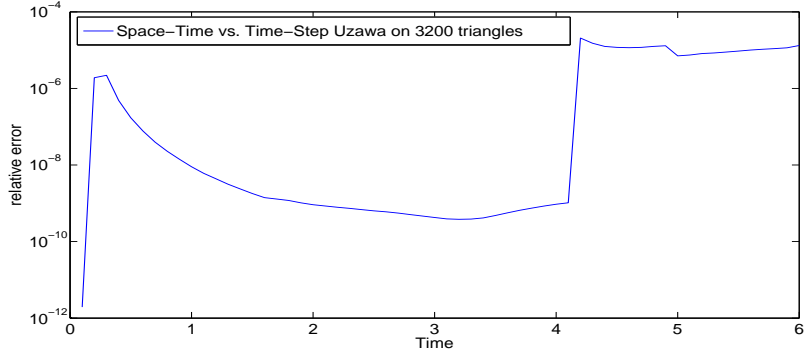} }
\caption{Relative $L^2(\Gamma)$-error between the solutions of the space-time and time-step Uzawa algorithms, {Example \ref{timestepexample}}.}
\label{relerror}
\end{figure}
A comparison of the clock times shows the computational efficiency of the time-step Uzawa algorithm: Its runtime of {455.3} cpu seconds compares to {1301.6} cpu seconds required by the space-time algorithm.\\

To illustrate our method for non-flat contact geometries, we consider a cube with three contact faces. Physically, one may think of a rigid cube which is tightly fixed to an elastic surrounding material material on its three other faces.
\begin{example}\label{contact2}
Let $\Gamma$ be the surface of the cube $[-2,2]^3$, with contact area $G$ consisting of the top, front and right faces.  We set $T=6$ 
and use the same ansatz and test functions as in Example 27. On each of the contact faces we prescribe a right hand side $$h(t,x)=e^{-2t}t^4\cos(2\pi x)\cos(2\pi y) \chi_{[-0.25,0.25]}(x) \chi_{[-0.25,0.25]}(y)\ ,$$ centered in the midpoint of each face. \textcolor{black}{The benchmark is obtained by extrapolation.}
\end{example}
The evolution of $u_{\Delta t, h}$ and the Lagrange multiplier $\lambda_{\Delta t, h}$ on the top face of the cube for are depicted in Figure \ref{Figlast} \textcolor{black}{for $\Delta t=0.1$ and CFL ratio $ \frac{\Delta t}{h}\approx 0.7$}. The nonzero displacement $u_{\Delta t, h}$ spreads from the neighboring contact faces into the shown area by time $t=5$ and eventually leads to strong contact near the upper left corner. 
  \begin{figure}[H]
  \makebox[\linewidth]{
   \subfigure[t=0.1]{
\includegraphics[width=9cm, height=4cm]{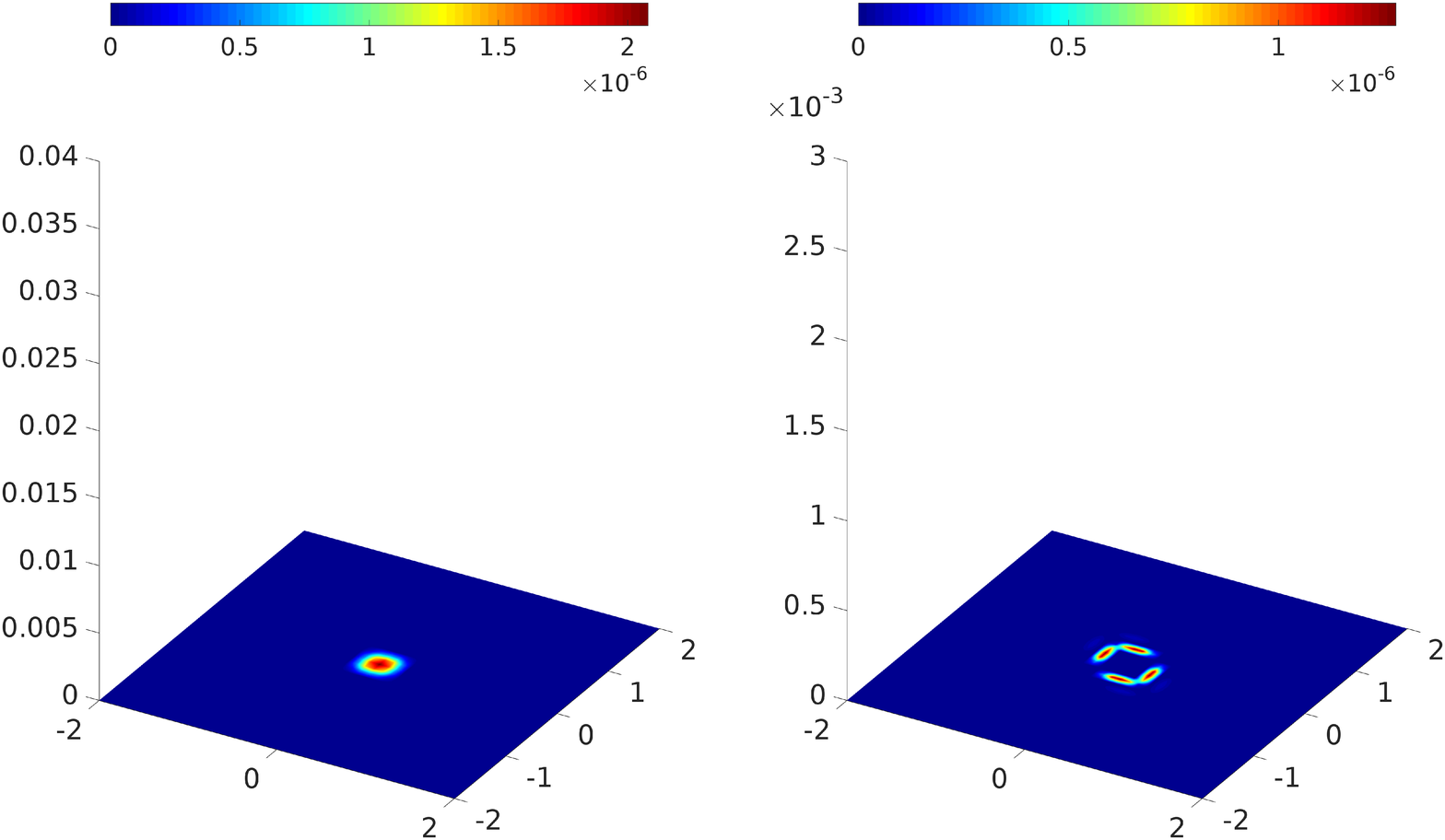} }
\subfigure[t=3]{
\includegraphics[width=9cm, height=4cm]{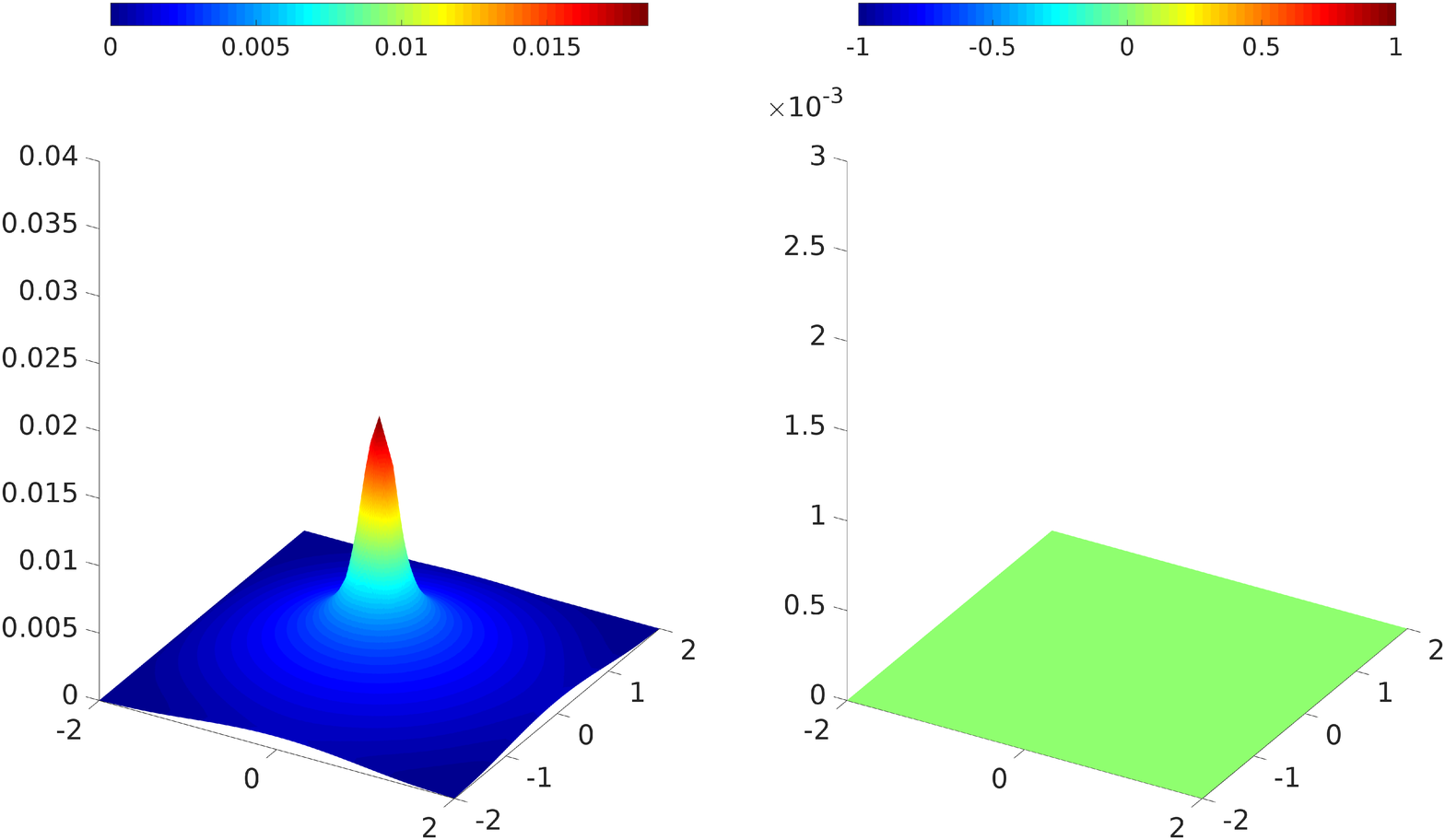} } }
\makebox[\linewidth]{
\subfigure[t=5]{
\includegraphics[width=9cm, height=4cm]{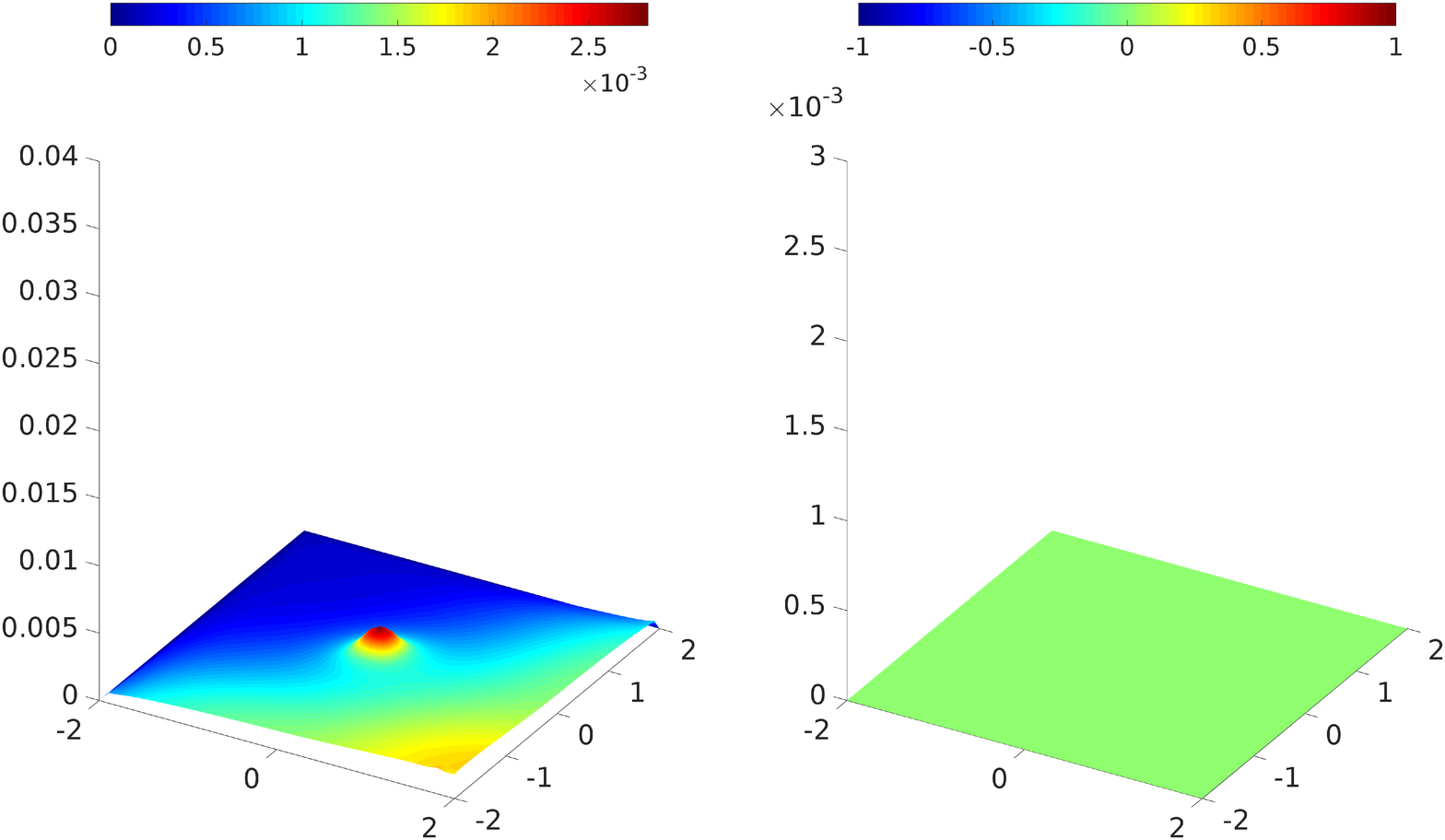} }
\subfigure[t=6]{
\includegraphics[width=9cm, height=4cm]{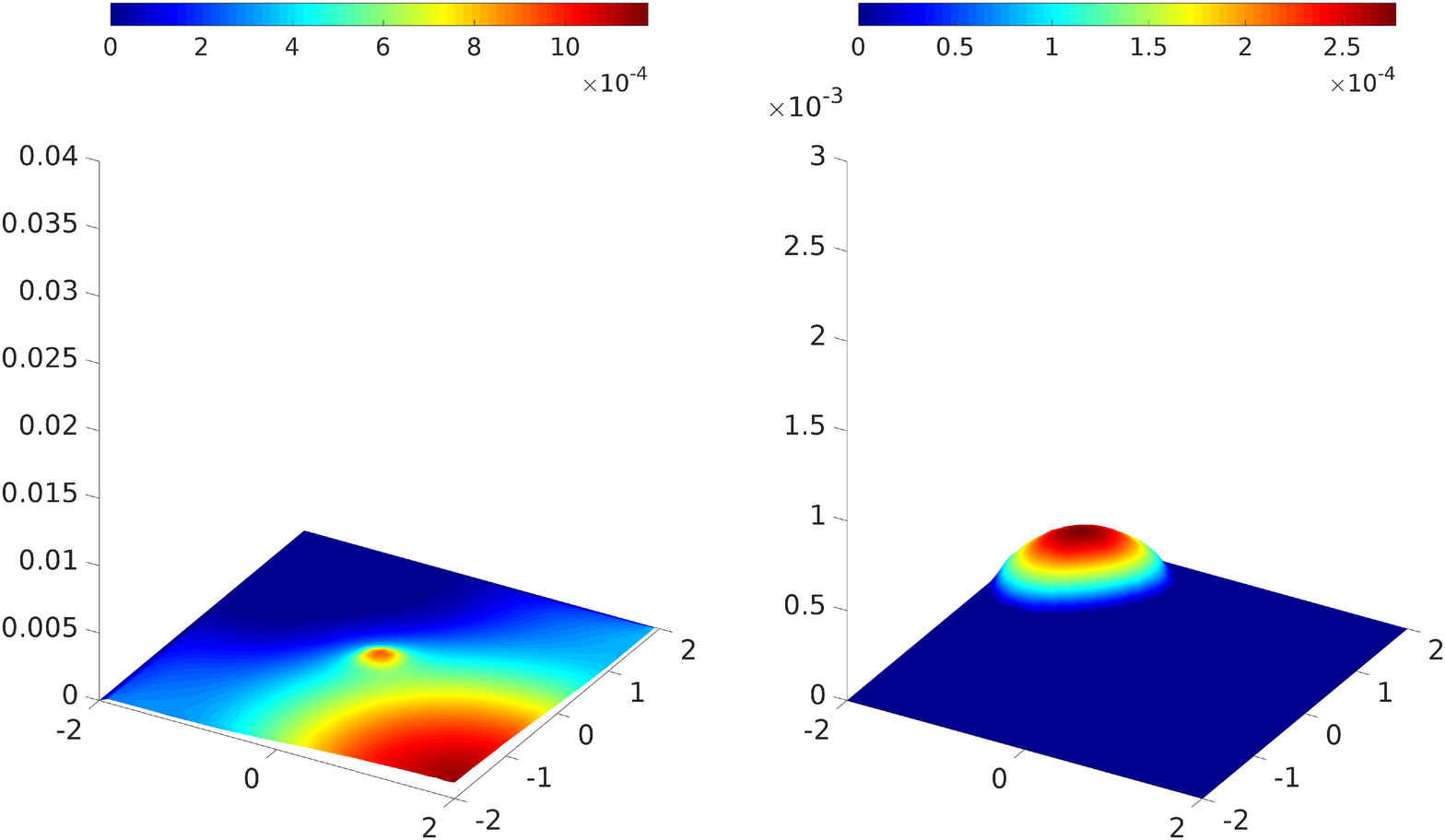} } }
\caption{{Evolution of $u$ and $\lambda$ in $[-2,2]^2\times \{2\}$ for the contact problem on $[-2,2]^3$,  Example \ref{contact2}.}}\label{Figlast}
\end{figure}
The error between the benchmark and approximate solutions on coarser meshes is shown in Figure \ref{l2ContactCuberelative} \textcolor{black}{for $\Delta t=0.075 $ with fixed CFL ratio $ \frac{\Delta t}{h}\approx 0.53$}. It shows a convergence with an approximate convergence rate of $\alpha = 0.6$, respectively $1.8$ in terms of $h$. Because the convergence deviates from a straight line, the asymptotic convergence rate might differ slightly. Note that the considered meshes are not refinements of each other, which may explain the kink in the third data point.

\begin{figure}[H]
\makebox[\linewidth]{
\includegraphics[width=12cm,height=6cm]{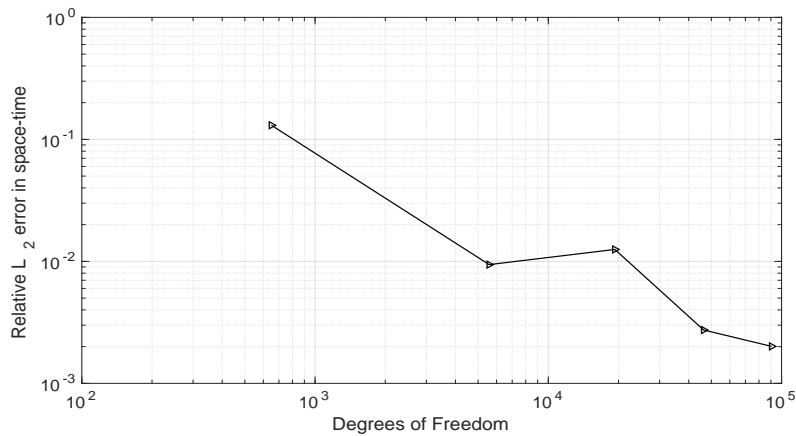} }
\caption{Relative $L^2([0,T]\times \Gamma)$-error vs.~degrees of freedom for the contact problem for fixed $\frac{\Delta t}{h} \approx 0.53$, {Example \ref{contact2}}.}
\label{l2ContactCuberelative}
\end{figure}

\subsection{Single-layer potential}

We finally address the punch problem described by the variational inequality \eqref{varVI}. Unlike for the problems involving the Dirichlet-to-Neumann operator, in this case we do not need to approximate the integral operator. As in the previous numerical examples, we consider both flat and more general contact geometries.

\begin{example}\label{punchex11}
\textcolor{black}{For the punch problem \eqref{varVI} we choose $\Gamma=[-2,2]^2\times\{0\}$ and the area of contact $G=[-1.2,1.2]^2\times\{0\}$.  We set $T=6$ and keep the CFL ratio fixed at \textcolor{black}{$\frac{\Delta t}{h}\approx 1.06$}. As right hand side \textcolor{black}{we consider as in Example \ref{contact1}} $h(t,x)=e^{-2t}t\cos(2\pi x)\cos(2\pi y) \chi_{[-0.25,0.25]}(x) \chi_{[-0.25,0.25]}(y)$, and we look for a numerical solution and Lagrange multiplier which are piecewise constant in time, linear in space.}
\end{example}
We solve the variational inequality using the time-step Uzawa algorithm. In a given time step, we stop the Uzawa iteration when subsequent iterates either have a relative difference of less than $10^{-12}$ or if the $\ell_\infty$-norm is less than $10^{-10}$. 

Figure \ref{Fig1} shows the solution $u_{\Delta t,h}$ to the punch problem and its Lagrange multiplier $\lambda_{\Delta t,h}$. In this case the mesh consists of $12800$ triangles and  $\Delta t=0.075$} \textcolor{black}{and $\rho=0.01$}. Contact is observed for most of the considered time interval.
  \begin{figure}[H]
  \makebox[\linewidth]{
   \subfigure[t=0.075]{
\includegraphics[width=9cm, height=4cm]{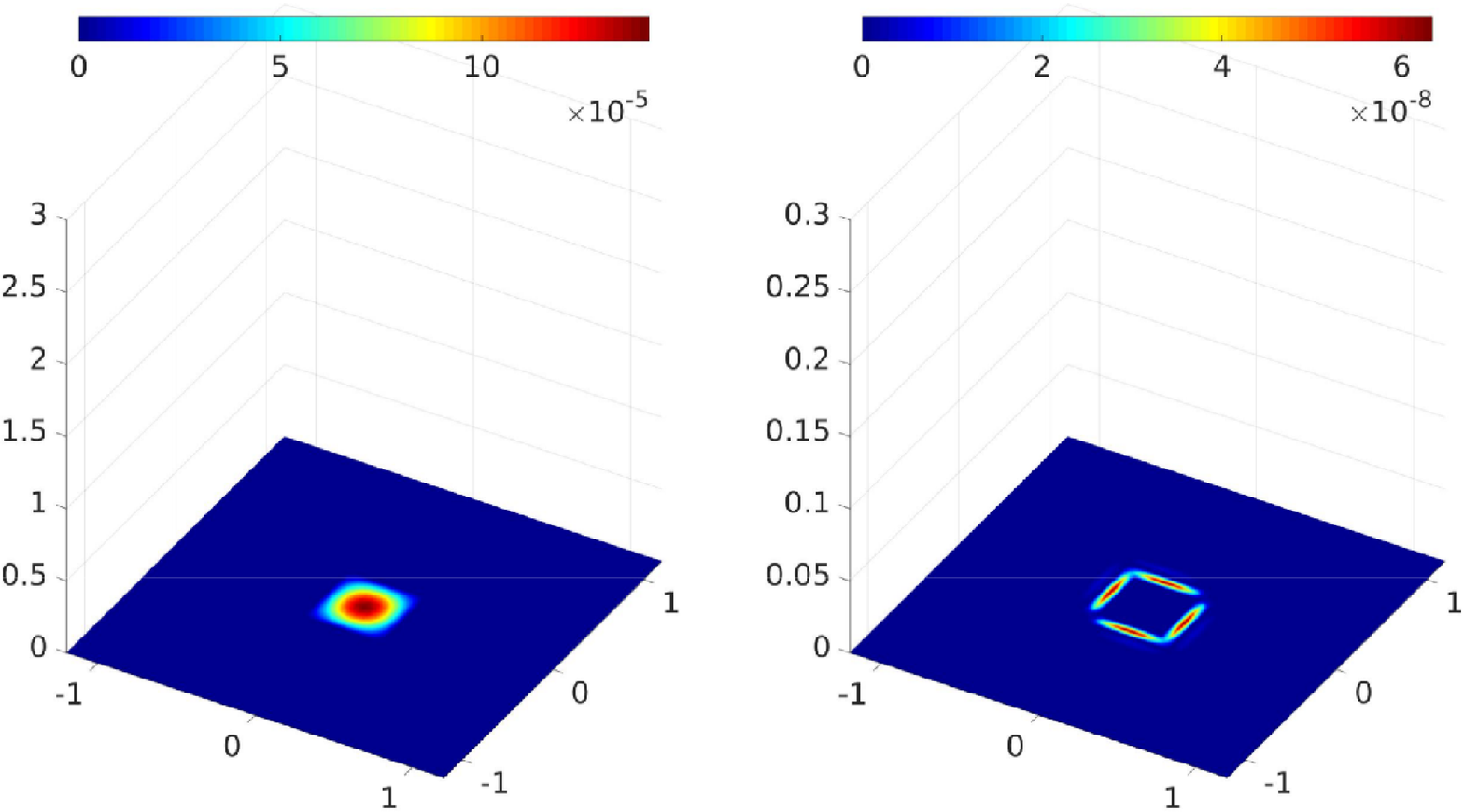} } 
\subfigure[t=1.05]{
\includegraphics[width=9cm, height=4cm]{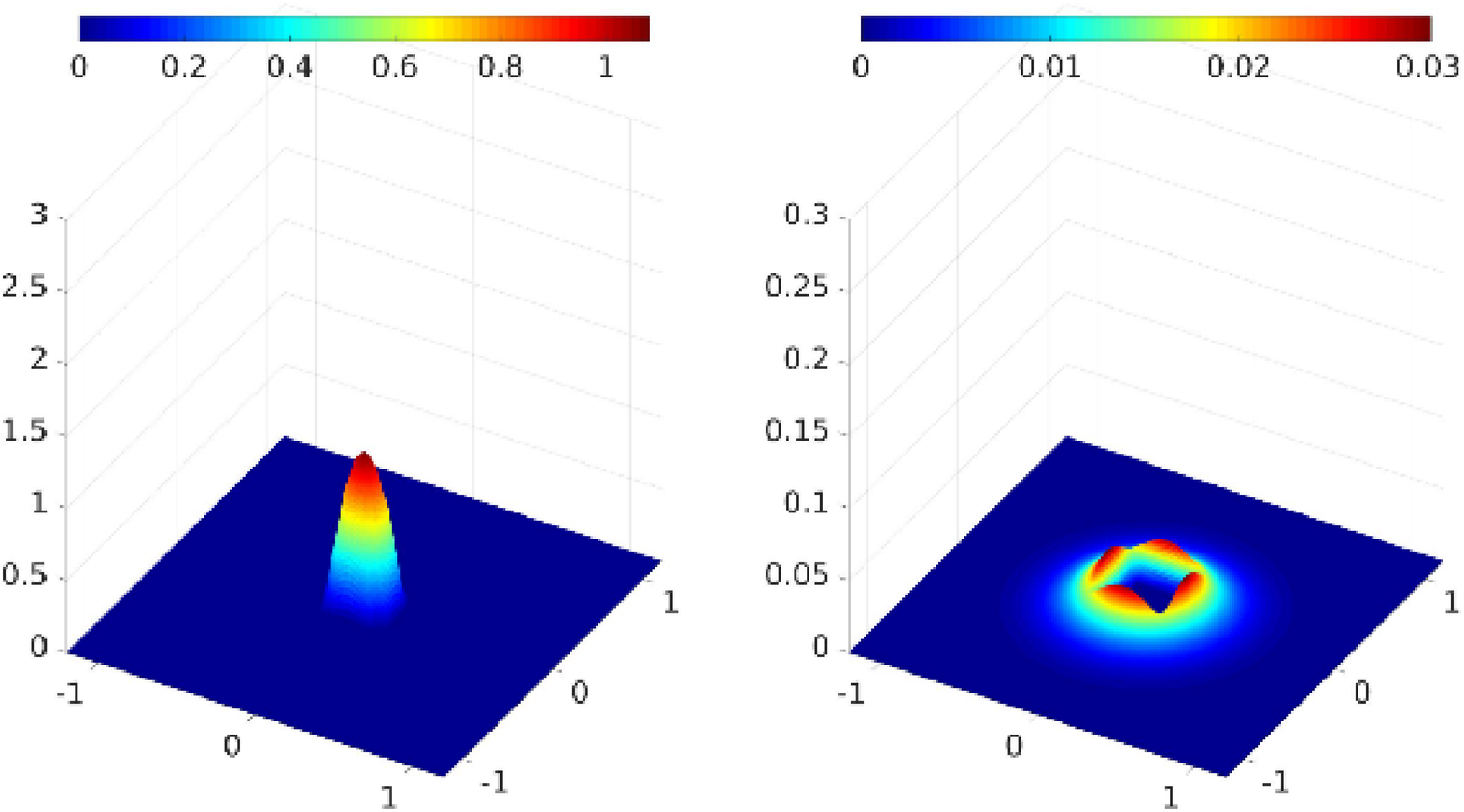} } }
\end{figure}
\begin{figure}[H]
\makebox[\linewidth]{
\subfigure[t=3]{
\includegraphics[width=9cm, height=4cm]{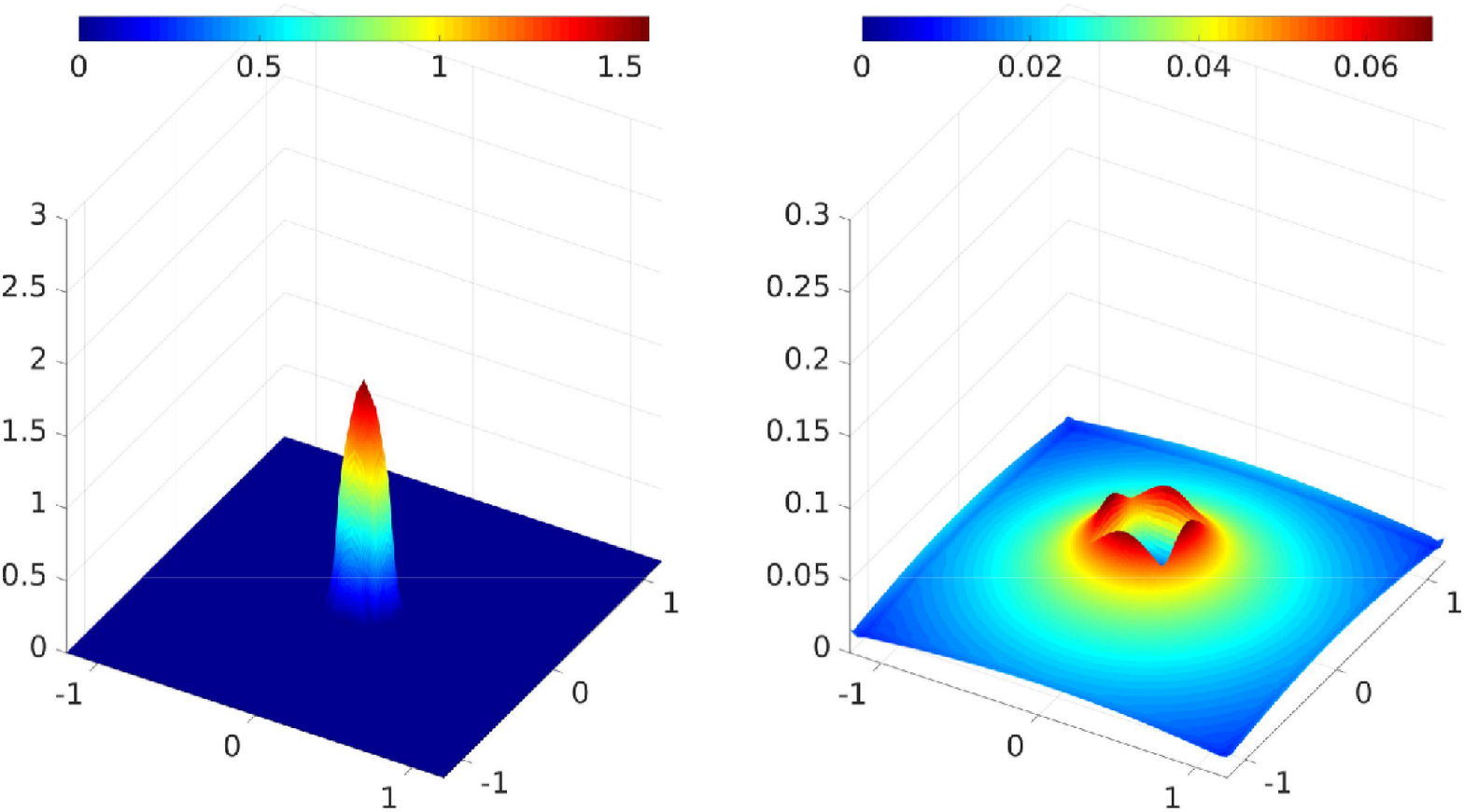} }
\subfigure[t=6]{
\includegraphics[width=9cm, height=4cm]{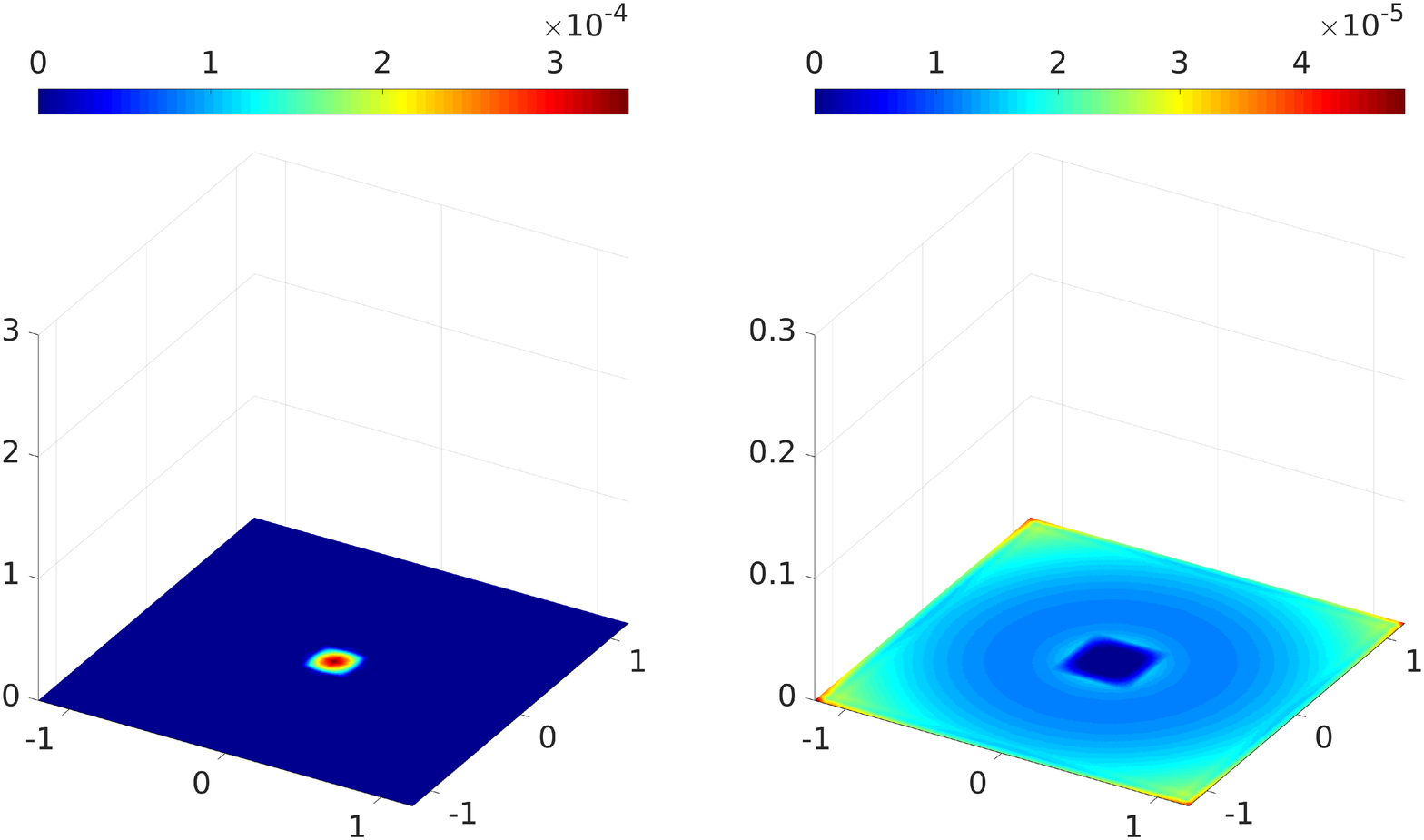} } }
\caption{\textcolor{black}{Evolution of $u$ and $\lambda$  for the punch problem on $G=[-1.2,1.2]^2\times\{0\}$, Example 30.} }
\label{Fig1}
\end{figure}
Because of the potentially low spatial regularity of the solution, which a priori only belongs to a Sobolev space with negative exponent, we do not consider the error of the numerical solutions in $L^2([0,T] \times \Gamma)$. As a weaker measure, we consider convergence in the energy norm defined by $V$.

Figure  \ref{VScreenError} shows the convergence of the numerical solutions in energy. \textcolor{black}{The energy, similar to the Example \ref{contact1}, shows a rate of convergence of $\alpha=0.76$. In terms of $h$ we have a rate of convergence of $2.28$.}
\begin{figure}[H]
\makebox[\linewidth]{
\includegraphics[width=12cm,height=6cm]{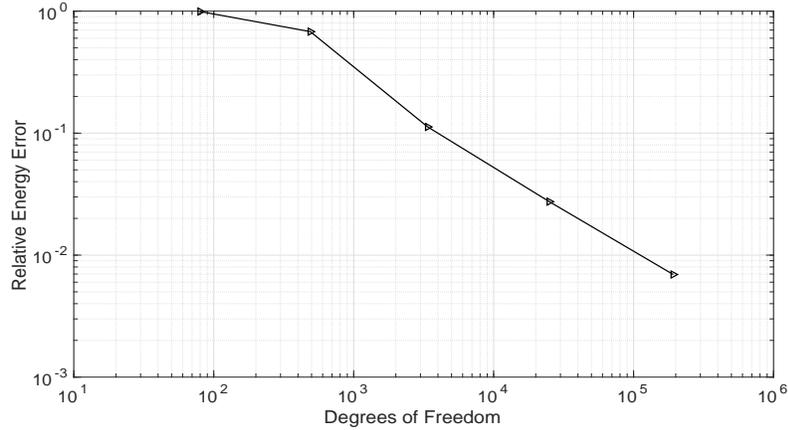} }
\caption{Relative energy error for the punch problem for fixed $\frac{\Delta t}{h}$, {Example \ref{punchex11}}.} 
\label{VScreenError}
\end{figure}

  \begin{figure}[H]
  \makebox[\linewidth]{
\subfigure[t=0.1]{
\includegraphics[width=9cm,
height=4cm]{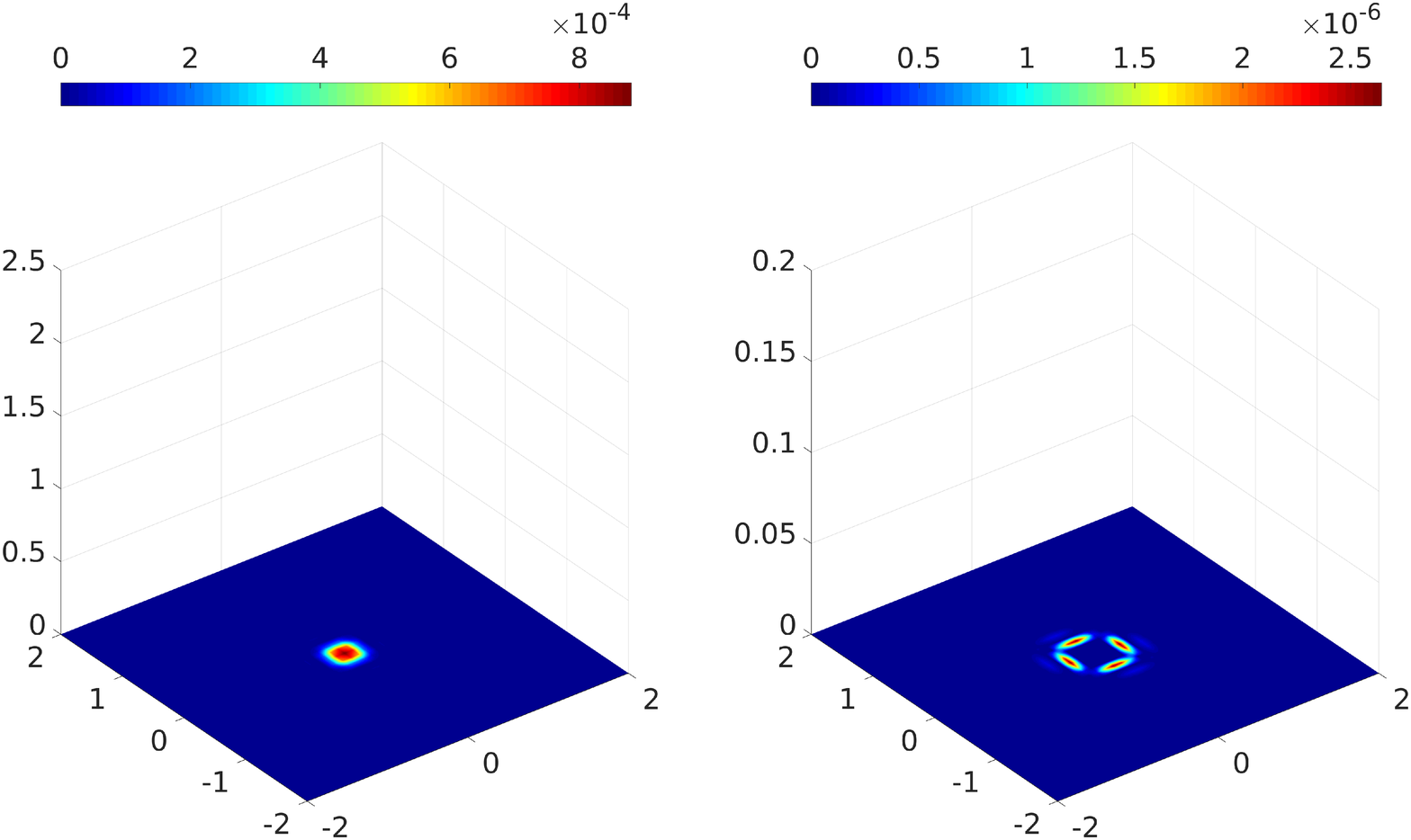} }
\subfigure[t=2]{
\includegraphics[width=9cm, height=4cm]{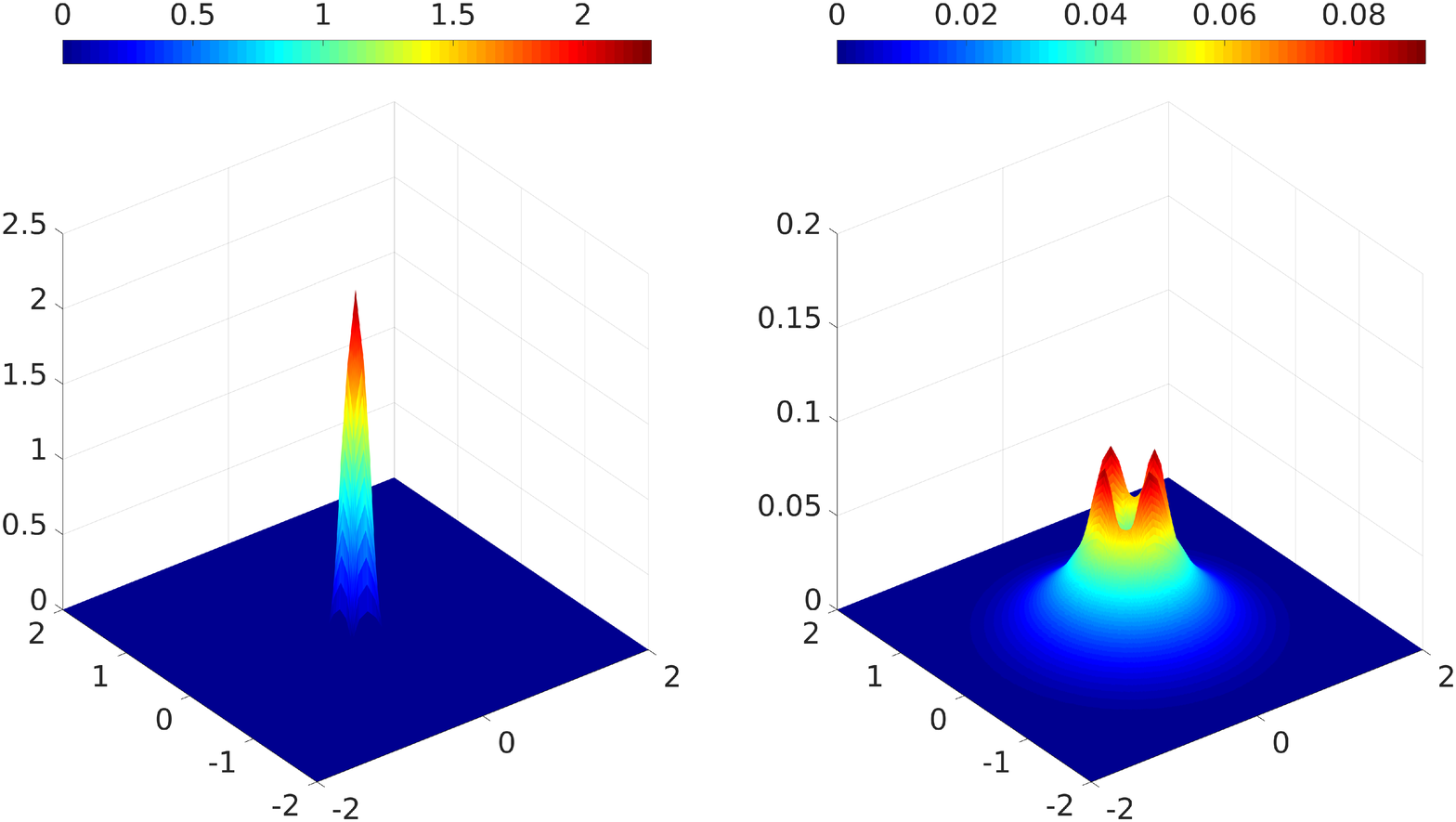} } }
\makebox[\linewidth]{
\subfigure[t=3]{
\includegraphics[width=9cm, height=4cm]{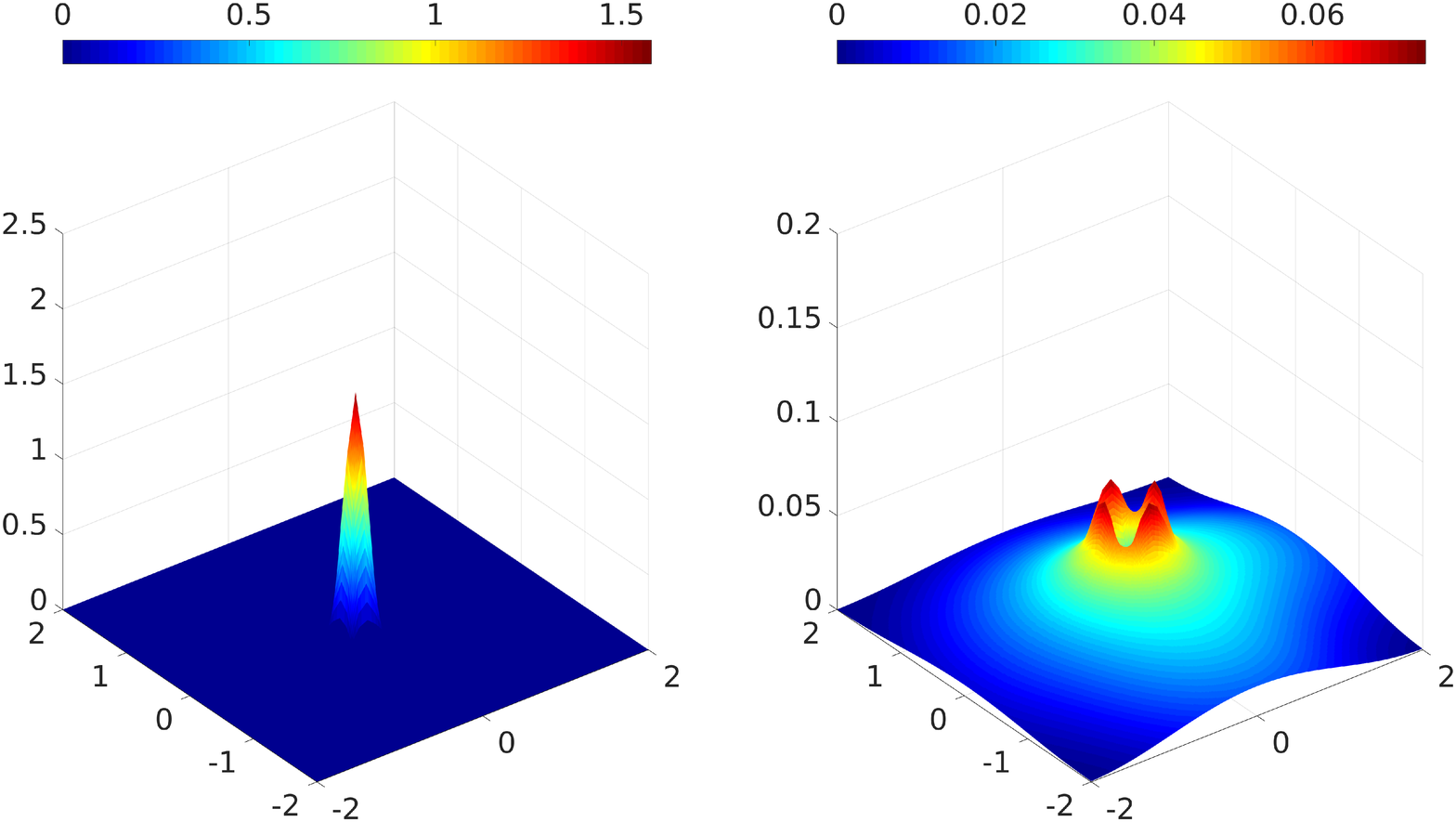} } 
\subfigure[t=3.6]{
\includegraphics[width=9cm, height=4cm]{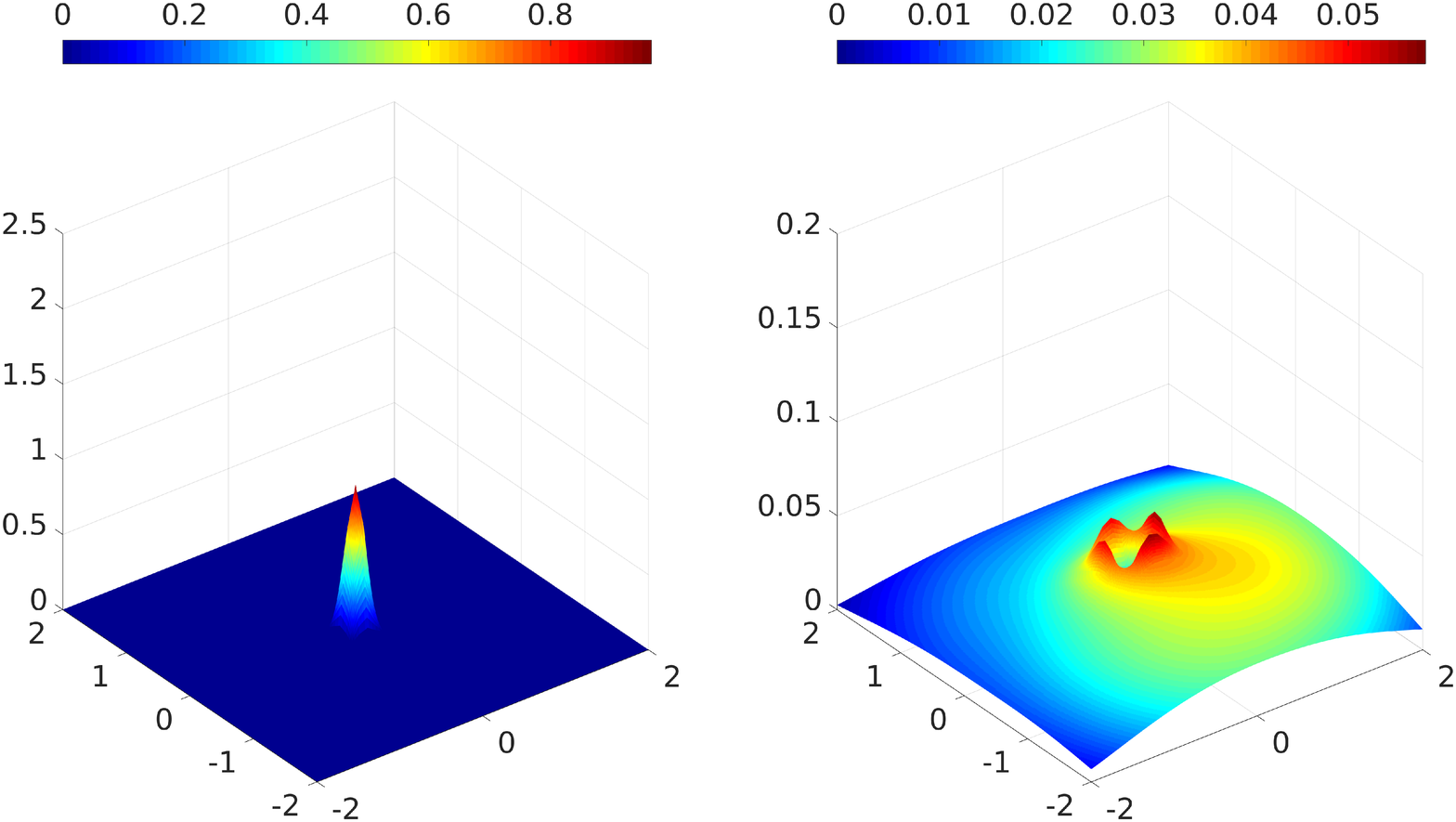} } }
\caption{\textcolor{black}{Evolution of $u$ and $\lambda$ in  $[-2,2]^2\times \{2\}$ for the punch problem on $[-2,2]^3$, Example 31.}}
\label{punchpic}
\end{figure}

We finally consider the punch problem with  the entire surface of the cube as contact area.
\begin{example}
\textcolor{black}{
For the punch problem \eqref{varVI} we choose $G=\Gamma$ to be the surface of $[-2,2]^3$. We set $T=3.6$  
As right hand side we consider the same function as in Example \ref{contact2},$$h(t,x)=e^{-2t}t^4\cos(2\pi x)\cos(2\pi y) \chi_{[-0.25,0.25]}(x) \chi_{[-0.25,0.25]}(y)\ ,$$ centered in the midpoint of the top, front and right face. We use the same ansatz and test functions as in Example 30. The benchmark energy is again computed by extrapolation.} 
\end{example}

The variational inequality is solved using the time-step Uzawa algorithm as in Example \ref{punchex11}. Figure \ref{punchpic} shows the solution $u_{\Delta t, h}$ to the punch problem and its Lagrange multiplier $\lambda_{\Delta t, h}$ on the surface of the cube \textcolor{black}{for $\Delta t=0.01$ and CFL ratio $\frac{\Delta t}{h}\approx 0.7$. The mesh consists of $19200$ triangles.} Again contact is observed for most times.

As in the previous example, in Figure \ref{VCubeError} we show the convergence in energy 
\textcolor{black}{for $\Delta t=0.075$ and $\frac{\Delta t}{h}\approx 0.53$}.
\textcolor{black}{ We here compute a convergence rate of roughly $\alpha= 0.9$ from the last $4$ points, respectively $2.7$ in terms of $h$. Because the convergence deviates from a straight line, the asymptotic convergence rate might differ slightly. Note  the  kink in the third data point corresponds to the kink in Figure \ref{l2ContactCuberelative}.}
\begin{figure}[H]
\makebox[\linewidth]{
\includegraphics[width=12cm, height=6cm]{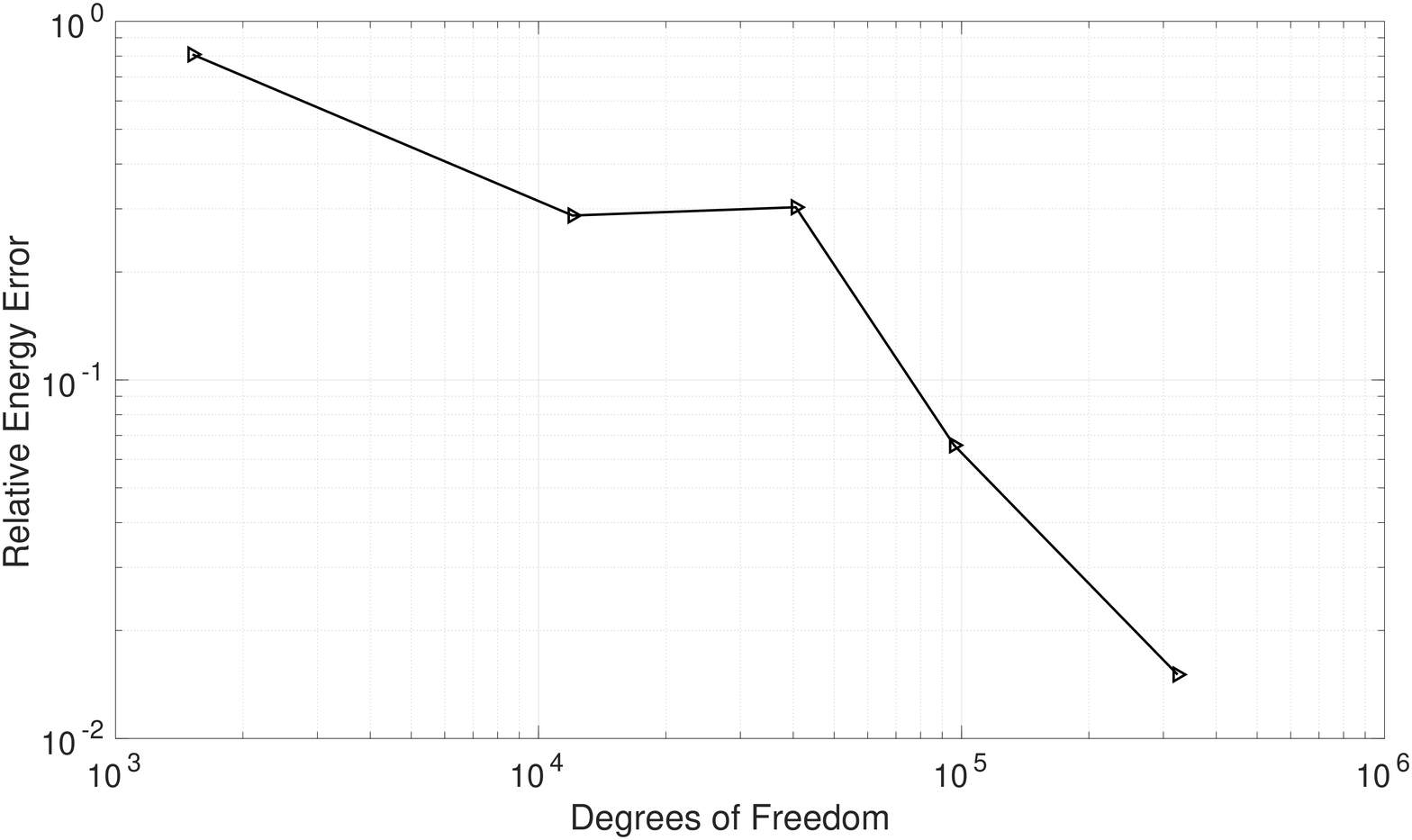} }
\caption{Relative error of the energy for the punch problem for fixed \textcolor{black}{$\frac{\Delta t}{h}$, Example 31}.  }
\label{VCubeError}
\end{figure}

\section{Conclusions}

In  this  work  we  propose and analyze a Galerkin boundary element method to solve dynamic contact problems for the wave equation. Boundary elements provide a natural and efficient formulation, as the contact takes place at the interface between two materials. \\

Analytically, we obtain a first a priori error analysis for a variational inequality involving the Dirichlet-to-Neumann  operator, as well as a similar analysis for a mixed formulation. The analysis and the stability of the method are crucially based on the weak coercivity of the formulation, and for the mixed method an inf-sup condition in space-time. The proof requires a flat contact area, as only in this case the existence of solutions to the continuous problem is known. Also a variational inequality and a mixed formulation for the single layer operator are considered, which do not require the approximation of the operator.\\

Numerical experiments demonstrate the efficiency and convergence of the proposed mixed method. A time-stepping Uzawa method for the solution of the variational inequality proves more efficient in practice, but also potentially less stable than a similar solver for the space-time system. The latter is shown to be provably convergent. As a key point, the numerical experiments indicate stability and convergence beyond flat geometries.\\

The current work provides a first, rigorous step towards efficient boundary elements for dynamic contact. Future work will focus on the a posteriori error analysis, which is essential for adaptive mesh refinements to resolve the singularities of the solution in space and time \cite{apost}, as well as on stabilized mixed space-time formulations \cite{banz, bf}. For applications to traffic noise \cite{banz2}, also the nonsmooth variational inequalities for frictional contact will be of interest.\\

\end{document}